\documentclass[microtype]{gtpart}
\usepackage[parfill]{parskip}
\usepackage{graphicx} 
\usepackage{amsmath}
\usepackage{amstext}
\usepackage{amsfonts}
\usepackage{amsthm}
\usepackage{amssymb}
\usepackage{amscd}
\usepackage{subfig}
\usepackage{hyperref}
\usepackage{mathrsfs} 
\usepackage{mathtools} 
\usepackage{comment}
\usepackage{tikz}
\usetikzlibrary{trees}
\usetikzlibrary{arrows}
\usetikzlibrary{matrix}
\usepackage{ifthen}

\title{\ \ \ \ \ \ \ \  KOSZUL DUALITY PATTERNS IN FLOER THEORY}

\newtheorem{theorem}{Theorem}
\newtheorem{lemma}[theorem]{Lemma}
\newtheorem{proposition}[theorem]{Proposition}
\newtheorem{definition}[theorem]{Definition}
\newtheorem{corollary}[theorem]{Corollary}
\newtheorem{example}[theorem]{Example}
\newtheorem{remark}[theorem]{Remark}
\newtheorem{conjecture}[theorem]{Conjecture}

\newcommand{\f}[1]{\mathbb{#1}}
\newcommand{\s}{\mathfrak{s}}

\newcommand{\K}{\mathrm{k}}

\newcommand{\QED}{\vspace{-.31in}\begin{flushright}\qed\end{flushright}}

\author{\ \ \ Tolga Etg\"u and Yank\i\ Lekili} 
\address{Ko\c{c} University \newline University of Illinois at Chicago}

\begin{document}

\begin{abstract} We study symplectic invariants of the open
	symplectic manifolds $X_\Gamma$ obtained by plumbing cotangent bundles
	of 2-spheres according to a plumbing tree $\Gamma$. For any tree $\Gamma$, we calculate (DG-)algebra models of the Fukaya category $\mathcal{F}(X_\Gamma)$ of closed exact
    Lagrangians in $X_\Gamma$ and the wrapped Fukaya category $\mathcal{W}(X_\Gamma)$.
    When $\Gamma$ is a Dynkin tree of type
    $A_n$ or $D_n$ (and conjecturally also for $E_6,E_7,E_8$), we prove that these models for the
    Fukaya category $\mathcal{F}(X_\Gamma)$ and
    $\mathcal{W}(X_\Gamma)$ are related by (derived) Koszul duality. As an
	application, we give explicit computations of symplectic cohomology of $X_\Gamma$ for
    $\Gamma=A_n,D_n$, based on the Legendrian surgery formula of Bourgeois, Ekholm and Eliashberg. 
\end{abstract} 
\maketitle

\section{Introduction}

Let us begin by recalling a simple example that we learned from Blumberg, Cohen and Teleman \cite{blucohtel}. Consider a \emph{simply-connected} smooth compact manifold
$S$ and its cotangent bundle $M = T^*S$ with its canonical symplectic
structure. The zero-section $S$ is a Lagrangian submanifold. We choose a basepoint $x \in S$ and consider the corresponding cotangent fibre $L = T^*_x S$.
This is another Lagrangian submanifold, a noncompact one. Throughout, our
Lagrangian submanifolds will be equipped with a \emph{brane structure}. This means
that they will be given an orientation, a spin structure (in particular, we assume here that $S$ is spinnable) and they will be equipped with a grading in the sense of \cite{seidelgraded}. 

Fix a coefficient field $\f{K}$. Lagrangian Floer theory gives us $\f{Z}$-graded $A_\infty$-algebras over $\f{K}$: \[ \mathscr{A} = CF^*(S, S) 
\text{ \ \ \ \  , \ \ \ \  }  \mathscr{B} = CW^*(L,L) \] 

Indeed, $S$ is an object of $\mathcal{F}(M)$, the Fukaya category of closed
exact Lagrangian branes in the Liouville manifold $M$ (see Seidel \cite{thebook}). The
endomorphisms of the object $S$ in this category are given by the Fukaya-Floer
$A_\infty$-algebra $CF^*(S,S)$. On the other hand, $L$ is an object
of $\mathcal{W}(M)$, the wrapped Fukaya category of $M$ (\cite{abouzseidel}). The
endomorphisms of the object $L$ in this category are given by the wrapped Floer cochain
complex $CW^*(L,L)$ which again has an associated $A_\infty$-structure (well-defined up to
quasi-isomorphism). 

Now, in this setting, by construction, there exists a full and faithful embedding
\[  \mathcal{F}(M) \to \mathcal{W}(M) \] 
since by definition $\mathcal{W}(M)$ allows certain non-compact Lagrangians in $M$ with controlled
behaviour at infinity, in addition to the exact compact Lagrangians in $M$. Furthermore, it is a
general fact (see \cite{abouzcot}) that a cotangent fibre generates the wrapped Fukaya category in
the derived sense. Hence, in particular, one has a Yoneda functor to the DG-category of
$A_\infty$-modules over $\mathscr{B}$: 
\[ \mathcal{Y} : \mathcal{F}(M) \to \mathscr{B}^{mod} \] 
which is a cohomologically full and faithful embedding. This sends an exact compact Lagrangian $T$
to the (right) $A_\infty$-module $\mathcal{Y}_{T}= CW^*(L,T)$ over $\mathscr{B}$. As a consequence,
one can compute $\mathscr{A}$ via its quasi-isomorphic image under $\mathcal{Y}$: \begin{equation}
\label{EM} \mathscr{A} \simeq \mathrm{hom}_{\mathscr{B}} (\f{K}, \f{K})  \end{equation}
where we write $\f{K}$ for the right $A_\infty$-module over $\mathscr{B}$ with underlying vector space $\f{K} \cdot x = CW^*(L,S)$. Equipping $S$ and $L$ with suitable brane structures, one can arrange that the degree $|x|=0$. The only non-trivial module map is the multiplication by the idempotent element in $CW^0(L,L) = \f{K} \cdot e$ which acts as the identity. The other products (including the higher products) are necessarily trivial. This can be seen from the fact that $CW^*(L,L)$ is supported in non-positive degrees (as we shall see below). Note that we are following the conventions of \cite{thebook} where, for example, the $A_\infty$-module maps are given by Floer products:
\[ CW^*(L,S) \otimes CW^*(L,L)^{\otimes k}  \to CW^*(L,S)[1-k] \ \ \ \ \ \ \ k=0,\ldots   \]

Throughout, upwards shift of grading by $n$ is written as $[-n]$.

On the other hand $CW^*(L,S)$ is also a (left) $A_\infty$-module over $CF^*(S,S)$, where $A_\infty$-module maps are given by Floer products:  
\[ CF^*(S,S)^{\otimes k} \otimes CW^*(L,S) \to CW^*(L,S)[1-k] \ \ \ \ \ \ \ k=0,\ldots   \]
To be in line with the conventions of \cite{thebook}, we prefer to view this as a right $\mathscr{A}^{op}$-module (which entails slightly different sign conventions). In fact, in our setting, it turns out that $\mathscr{A}$ is quasi-isomorphic to $\mathscr{A}^{op}$. 

Somewhat more surprisingly, one can also compute $\mathscr{B}$ as: 
\begin{equation} \label{dual} \mathscr{B}^{op} \simeq \mathrm{hom}_{\mathscr{A}^{op}} (\f{K}, \f{K}) \end{equation} 

This is an instance of \emph{Koszul duality}.

\begin{figure}[tb!]
\centering
\includegraphics[width=0.3\textwidth]{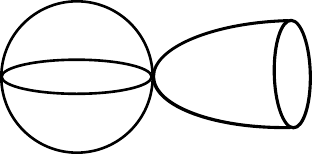}
\caption{A picture of Koszul duality}
\label{fig0}
\end{figure}

To see this, we observe that both $\mathscr{A}$ and $\mathscr{B}$ have topological models due to Abouzaid (\cite{abouzplum}, \cite{abouzwrap}). Indeed, there are $A_\infty$-equivalences:
\[ \mathscr{A} \simeq C^*(S; \f{K} ) \text{\ \ \ \ and\ \ \ \ } \mathscr{B} \simeq C_{-*} ( \Omega_x S; \f{K}), \]
where $\Omega_x S$ is the based loop space of $S$ at $x$. Notice the cohomological grading in place. In particular, $\mathscr{A}$ is supported in non-negative
degrees and $\mathscr{B}$ is supported in non-positive degrees. 

Now, Eqn.~(\ref{EM}) becomes an Eilenberg-Moore equivalence (of DGA's):
\[ \mathrm{Rhom}_{C_{-*}(\Omega_x S)}(\f{K}, \f{K}) \simeq C^*(S; \f{K})   \]
and Eqn.~(\ref{dual}) is Adams' cobar equivalence (see \cite{adams, JM}):
\[ \mathrm{Rhom}_{C^*(S)^{\, op}} (\f{K},\f{K}) \simeq C_{-*}( \Omega_x S)^{op}. \] 
($^{op}$ operations get removed from both sides, if one considers $\f{K}$ as a left $C^*(S)$-module.)

This duality is relevant to us because it induces an isomorphism at the level of Hochschild cohomology. Namely, by a general result of Keller \cite{keller} (see also \cite{FMT}) one obtains an isomorphism of Gerstenhaber algebras (in fact, of $B_\infty$-algebras at the chain level) : 
\[ HH^*(C^*(S), C^*(S)) \cong HH^*(C_{-*} (\Omega_x S), C_{-*}(\Omega_x S)). \]

In view of Abouzaid's generation result \cite{abouzwrap}, the right-hand side is in turn isomorphic
to $HH^*(\mathcal{W}(M))$ as a Gerstenhaber algebra. On the other hand, the work of
Bourgeois-Ekholm-Eliashberg \cite{BEE}, can be interpreted, over a field $\mathbb{K}$ of
characteristic 0, to give an isomorphism of Gerstenhaber algebras:
\[ HH^*(\mathcal{W}(M)) \cong SH^*(M). \]

The group on the right-hand side is called \emph{symplectic cohomology}.
Strictly speaking, the results of \cite{BEE} relate symplectic and Hochschild
\emph{homologies}. However, in our computations, we will use an explicit DG-algebra as a model for
$\mathcal{W}(M)$, which has an (open) Calabi-Yau property (in the sense of \cite{ginzburg}) implying
an isomorphism between Hochschild homology and cohomology. This allows us to use the
cohomological statement above that we find more convenient. Symplectic
(co)homology of a Liouville manifold is a symplectic invariant based on an
extension of Hamiltonian Floer (co)homology to non-compact symplectic
manifolds. It was introduced by Viterbo \cite{viterbo} in its current form. We
recommend \cite{biased} for an excellent introduction to symplectic cohomology
and the recent manuscript \cite{abouzmanuz} for more. Briefly, this is a very
interesting invariant of a Stein manifold that is sensitive to the underlying
symplectic structure (cf. \cite{groeli}). Symplectic cohomology is in general difficult to
calculate explicitly. However, \cite{BEE} and \cite{bee2} recently outlined a
proof of a surgery formula for symplectic (co)homology.  Combining this with
the very recent \cite{EkNg}, one obtains a purely combinatorial description of
symplectic cohomology of any 4-dimensional Weinstein manifold. (In the absence
of 1-handles and when the coefficient field is $\f{Z}_2$, one had
\cite{chekanov} as a precursor to \cite{EkNg}). This combinatorial description
is in general still highly complicated. It involves non-commutative and
infinite-dimensional chain complexes.

In the above setting, assuming that $\mathscr{A}= C^*(S)$ is a formal DG-algebra, that is, it is quasi-isomorphic to its homology $A= H^*(S)$, we get a promising way of computing symplectic cohomology. Namely, one has:
\[ HH^*(H^*(S), H^*(S)) = SH^*(M) \] 
By a famous result of Deligne-Griffiths-Morgan-Sullivan \cite{DGMS}, the formality assumption holds
if $S$ is a K\"ahler manifold and $\f{K}$ has characteristic zero. Note that as a consequence of
formality of $C^*(S)$, one has a bigrading on $HH^*(C^*(S), C^*(S))$; there is a cohomological grading $r$ associated with the Hochschild cochain complex and there is an internal grading $s$ coming from the grading on $H^*(S)$. To get an isomorphism to $SH^*(M)$, where the grading is given by a Conley-Zehnder type index, one has to consider the grading of the total complex corresponding to $r+s$. 

Let us note that one could arrive at the same conclusion by combining theorems of Viterbo \cite{viterbo} and Cohen-Jones \cite{cohenjones}. 

In this paper, we give a generalization of the above (in dimension 4) to Liouville manifolds $M=X_\Gamma$ obtained via plumbings of $T^*S^2$ according to a plumbing tree $\Gamma$. We will work over semisimple rings $\mathrm{k}$ of the form 
\[ \mathrm{k}= \bigoplus_v \f{K} e_v,  \] where $e_v^2 = e_v$ and $e_v e_w =0$ for $v \neq w$ and the index set of the sum is the vertex set $\Gamma_0$ of $\Gamma$.  

To wit, using Floer complexes over $\f{K}$, we set: 
\[ \mathscr{A}_\Gamma = \bigoplus_{v,w} CF^*(S_v, S_w), \] 
where the $S_v$ are the Lagrangian spheres corresponding to the zero sections of the cotangent bundles $T^*S^2$ that we are plumbing, and similarly we have 
\[ \mathscr{B}_\Gamma = \bigoplus_{v,w} CW^*(L_v, L_w) \]
where $L_v$ is a cotangent fibre at a chosen base point on $S_v$ (away from the plumbing region). 

In fact, assuming $\text{char}(\f{K}) =0$, it turns out that $\mathscr{A}_\Gamma$ tends to be a
formal DG-algebra (we can prove this when $\Gamma$ is of type $A_n$ or $D_n$, and conjecture it for
$E_6,E_7,E_8$), hence in such cases, 
we may replace it with $A_\Gamma = H^*(\mathscr{A}_\Gamma)$. Furthermore, very
early on, we will replace $\mathscr{B}_\Gamma$ with a quasi-isomorphic DG-algebra (see \cite[Prop. 4.11 and Thm. 5.8]{BEE}) which has a combinatorial description. Namely, we will use Chekanov's DG-algebra (\cite{chekanov}), with the cohomological grading, associated to a Legendrian link $\Lambda_\Gamma  = \bigcup \Lambda_v$ giving a
Legendrian surgery diagram for $X_\Gamma$ where the components are indexed by vertices $v$ of $\Gamma$ and each component $\Lambda_v$ is a
Legendrian unknot in $\f{R}^3$ (see Fig.~(\ref{fig2})).  In the
context of \cite{BEE}, the homologically graded version of this is also called
the Legendrian contact homology algebra. 

At this point, one obtains an explicit description of the DG-algebra
$\mathscr{B}_\Gamma$. A careful choice of the surgery diagram (with suitable decorations)
enables us to observe that the DG-algebra $\mathscr{B}_\Gamma$ is a
\emph{deformation} of  Ginzburg's
(cohomologically graded) DG-algebra $\mathscr{G}_\Gamma$ associated with the tree $\Gamma$ (see
Thm. (\ref{balik})).\footnote{An earlier version of this manuscript claimed an isomorphism between
    $\mathscr{B}_\Gamma$ and $\mathscr{G}_\Gamma$, due to our blindness to some higher energy
curves. We are indebted to the referee for opening our eyes.} 
    Note that in
\cite{ginzburg} Ginzburg associates a $CY3$ DG-algebra to any graph $\Gamma$
and a potential function. In this paper, $\Gamma$ is a tree and the potential
function vanishes. On the other hand, since we are plumbing copies of $T^* S^2$, our
DG-algebras are $CY2$. This generalization of the construction of
Ginzburg's DG-algebra in order to obtain a $CY2$ DG-algebra appears in \cite{vandenbergh}.
(See Def.~(\ref{ginzdga}) for the definition of $\mathscr{G}_\Gamma$).

The observation that $\mathscr{B}_\Gamma$ is a deformation of the corresponding Ginzburg DG-algebra
$\mathscr{G}_\Gamma$ enables us to use the vast literature on the study of Ginzburg's DG-algebras to
study symplectic invariants of $X_\Gamma$. Now, our discussion branches into two according to whether the underlying tree $\Gamma$ is of Dynkin type or not.

{\bf Dynkin case:} 

For $\Gamma$ of type $A_n$ or $D_n$, we prove the following theorem:
\begin{theorem} For $\Gamma= A_n$ and $\mathbb{K}$ arbitrary field, or $\Gamma=D_n$ and $\mathbb{K}$
    a field with $\mathrm{char}\ \mathbb{K} \neq 2$, there is a quasi-isomorphism of DG-algebras:
    \[ \mathscr{B}_\Gamma \simeq \mathscr{G}_\Gamma \] 
\end{theorem} 
For $\Gamma=A_n$, this result follows from a direct computation of
$\mathscr{B}_\Gamma$. However, for $\Gamma=D_n$, direct computation only shows that
$\mathscr{B}_\Gamma$ is a certain deformation of $\mathscr{G}_\Gamma$. We then appeal to standard
deformation theory arguments to show that this deformation is trivial when
$\mathrm{char}\ \mathbb{K} \neq 2$. In fact, we also prove that $\mathscr{B}_\Gamma$ and $\mathscr{G}_\Gamma$ are not
quasi-isomorphic when $\Gamma =D_n$ and $\mathrm{char}\ \mathbb{K} =2$ by showing that the
relevant obstruction class in $HH^2(\mathscr{G}_\Gamma)$ is non-trivial.  

We conjecture that $\mathscr{B}_\Gamma \simeq \mathscr{G}_\Gamma$ for $\Gamma= E_6, E_7$ if
$\mathrm{char}\ \mathbb{K} \neq 2,3$ and for $\Gamma = E_8$ if $\mathrm{char}\ \mathbb{K} \neq 2,3,
5$ but we leave the study of these exceptional cases to a future work. 

Assuming for brevity $\mathrm{char}\ \mathbb{K} \neq 2$, and $\Gamma = A_n$ or $D_n$, we can now
write $\mathscr{B}_\Gamma \simeq \mathscr{G}_\Gamma$. For $\Gamma$ of type $ADE$, $\mathscr{G}_\Gamma$ turns out to be non-formal
\cite{hermes}. Its cohomology has locally finite grading. Indeed, for an (algebraically
closed) field with characteristic zero, it was computed in \cite{hermes} that \[
H^*(\mathscr{G}_\Gamma) \cong \Pi_\Gamma \rtimes_\nu \mathrm{k}[u] \] as a $\mathrm{k}$-algebra,
where $\Pi_\Gamma$ is the preprojective algebra associated with the tree $\Gamma$, $|u|=-1$, and the multiplication is twisted by the Nakayama automorphism $\nu$ of $\Pi_\Gamma$. This is an involution, which is induced by an involution of the underlying Dynkin graph (see Sec.~(\ref{ginzburg})).

Because $\mathscr{G}_\Gamma$ is not formal, it is not immediately clear how to compute Hochschild cohomology of
$\mathscr{G}_\Gamma$. To help with this, we prove in Sec.~(\ref{seckoszul}) the following: 
\begin{theorem} Let $\mathbb{K}$ be any field. For any tree $\Gamma$, the associative
    algebra $A_\Gamma$ is Koszul dual to the DG-algebra $\mathscr{G}_\Gamma$, in the sense that there are DG-algebra
isomorphisms: \[ \operatorname{RHom}_{\mathscr{G}_\Gamma} (\K,\K) \simeq A_\Gamma \text{\ \ \ and \
\ \ } \operatorname{RHom}_{A_\Gamma^{op}} (\K,\K) \simeq \mathscr{G}_\Gamma^{op} \]
\end{theorem}
Therefore, by Keller's result \cite{keller},  we can use this to compute $SH^*(X_\Gamma)$ as: \[
    SH^*(X_\Gamma) \cong HH^*(\mathscr{G}_\Gamma) \cong HH^*(A_\Gamma). \] Since $A_\Gamma$ is a rather small finite-dimensional algebra over $\mathrm{k}$, one can find a minimal projective
resolution to compute the latter group. Indeed, Brenner, Butler and King \cite{almostkoszul} give a
minimal periodic (graded) resolution for $A_\Gamma$. However, we will find a short-cut to compute
$HH^*(A_\Gamma)$ as a bigraded algebra for $\Gamma = A_n, D_n$ over a field $\f{K}$ of arbitrary
characteristic. An explicit presentation of $HH^*(A_\Gamma)$ as a (graded) commutative
$\mathbb{K}$-algebra is provided in Thm. (\ref{hha}) for $A_n$ and in Thm. (\ref{hhd}) for
$D_n$. 

As we noted above in the case $\Gamma=D_n$ and when $\mathrm{char}\ \mathbb{K} =2$,
$\mathscr{B}_\Gamma$ is indeed a non-trivial deformation of $\mathscr{G}_\Gamma$. In this case $\mathscr{A}_\Gamma$ is not formal and indeed $\mathscr{B}_\Gamma$
and $\mathscr{A}_\Gamma$ are Koszul dual in the above sense. Therefore, it appears that a natural
statement (that applies in all characteristics) maybe that $\mathscr{A}_\Gamma$ and $\mathscr{B}_\Gamma$ are Koszul dual when $\Gamma$ is a Dynkin tree. 
\ \ 

{\bf Non-Dynkin case:} 

In this case, we only know that $\mathscr{B}_\Gamma$ is a deformation of
$\mathscr{G}_\Gamma$ and even at the formal level there are many non-trivial deformations of $\mathscr{G}_\Gamma$ since $HH^2(\mathscr{G}_\Gamma,\mathscr{G}_\Gamma)$ is big (see Thm. (\ref{sched})) and $HH^3(\mathscr{G}_\Gamma, \mathscr{G}_\Gamma)=0$. Hence, it is not clear whether the deformation
corresponding to $\mathscr{B}_\Gamma$ is trivial or not. On the other hand, as $\mathscr{B}_\Gamma$ (being a model for the wrapped Fukaya category of $X_\Gamma$) is also
a Calabi-Yau (CY) algebra, one can see the deformation of $\mathscr{G}_\Gamma$ to $\mathscr{B}_\Gamma$ as a deformation of CY2-algebras. In characteristic 0, this allows one to conclude that the corresponding formal deformation is trivial as follows.

$\mathscr{G}_\Gamma$ is in a sense simpler for $\Gamma$ non-Dynkin. Namely,
in this case, the homology $H^*(\mathscr{G}_\Gamma)$ turns out to be concentrated in degree 0 and \[ H^0(\mathscr{G}_\Gamma) \cong
\Pi_\Gamma \] is the preprojective algebra associated with the tree $\Gamma$.
For a non-Dynkin tree $\Gamma$, working over $\mathbb{K}$ of characteristic 0,  Hermes \cite{hermes} proved that $\mathscr{G}_\Gamma$ is formal, that is, $\mathscr{G}_\Gamma$ is
quasi-isomorphic to its homology $\Pi_\Gamma$ (see also Cor.~(\ref{rehermes}) for another proof that works over any field). Furthermore, it is well-known
that $\Pi_\Gamma$ is Koszul in the classical sense (cf. \cite{priddy},
\cite{bgs}) over $\mathrm{k}$. The quadratic dual to $\Pi_\Gamma$ is given by
the associative algebra $A_\Gamma = H^*(\mathscr{A}_\Gamma)$ - the zigzag
algebra associated with the tree $\Gamma$ (\cite{huekho}). 

The Gerstenhaber algebra structure of the Hochschild cohomology of the preprojective algebra
$\Pi_\Gamma$ in the non-Dynkin case has already been computed by Crawley-Boevey, Etingof, Ginzburg
in \cite{BVEG} when $\f{K}$ has characteristic zero, and by Schedler \cite{schedler} in
general. $HH^*(\Pi_\Gamma) \neq 0$ only for $*=0,1,2$. We give a brief review of these computations of $HH^*(\Pi_\Gamma)$ for completeness; see
Sec.~(\ref{nondynkin}) for a full description. Now, $\mathscr{B}_\Gamma$ can be seen as a
deformation of the CY2 algebra $\Pi_\Gamma$. If we consider the corresponding formal deformation, then the
associated Kodaira-Spencer class lives in $\mathrm{Ker}(\Delta: HH^2(\Pi_\Gamma) \to
HH^1(\Pi_\Gamma))$, where $\Delta$ is the BV-operator (see for ex. \cite{EtGi}). Now, it can be
observed from the description given in Sec.~(\ref{nondynkin}) that this kernel is trivial if $\mathrm{char}\ \mathbb{K}=0$.  This result does not hold in arbitrary characteristic, see Remark (\ref{subtree})
(cf. Remark (\ref{BVremark})) for a proof that this deformation is non-trivial over a field $\mathbb{K}$ of characteristic 2.  

Finally, let us remark that the above argument only shows that the associated formal deformation is trivial. This does not mean that $\mathscr{B}_\Gamma$ is quasi-isomorphic to $\mathscr{G}_\Gamma$ - such a quasi-isomorphism holds only after certain completion. As was shown in our subsequent work \cite{EL2}, $H^0(\mathscr{B}_\Gamma)$ is isomorphic to the multiplicative preprojective algebra associated with $\Gamma$, introduced by Crawley-Boevey and Shaw \cite{CBS}. On the other hand $H^0(\mathscr{G}_\Gamma)$ is isomorphic to the additive preprojective algebra $\Pi_\Gamma$. It is known that additive and multiplicative preprojective algebras are isomorphic only when $\mathrm{char}\ \mathbb{K}=0$ and $\Gamma$ is Dynkin, and in general, they are isomorphic when $\mathrm{char}\ \mathbb{K}=0$ only after completion as it follows from the above deformation theory argument.

In Sec.~(\ref{plumbing}), we provide geometric preliminaries on plumbings of cotangent bundles. In
Sec.~(\ref{ginzburg}), we give a computation of Legendrian contact homology of the Legendrian link
$\Lambda_\Gamma$ associated to a tree $\Gamma$, show that it is isomorphic to a deformation of
the  corresponding
CY2 Ginzburg DG-algebra $\mathscr{G}_\Gamma$ (Thm.~(\ref{balik})) and that this deformation is
trivial for $\Gamma= A_n$ or $D_n$, when $\mathrm{char}\ \mathbb{K} \neq 2$ in the latter case (Thm.~(\ref{chainmap})).  Sec.~(\ref{spheres}) computes the
Floer cohomology algebra $\mathscr{A}_\Gamma$ of the spheres in $X_\Gamma$. The main result appears
in Sec.~(\ref{seckoszul}) where we show that $\mathscr{G}_\Gamma$ and
$A_\Gamma=H^*(\mathscr{A}_\Gamma)$ are Koszul duals for any tree $\Gamma$. Finally, as an
application of our main result, in Sec.~(\ref{weshallcompute}), we explicitly compute Hochschild
cohomology of $ \mathscr{G}_\Gamma$, hence also of $\mathscr{B}_\Gamma$ for $\Gamma = A_n$ and $D_n$, assuming $\mathrm{char}\ \mathbb{K} \neq 2$ if $\Gamma=D_n$. 

{\bf Convention. } Throughout, we adhere to the following conventions.
$\f{K}$ is a field, of arbitrary characteristic unless otherwise specified, and $\mathrm{k}$ is a semisimple ring, generated over $\f{K}$ by finitely many mutually orthogonal idempotents.
Letters $A,B,\ldots$ denote associative algebras over $\mathrm{k}$.
All our modules are {\it right} modules and all our multiplications are read from {\it right to left}. 
Letters $\mathscr{A},\mathscr{B},\ldots $ denote $A_\infty$- or DG-algebras over $\mathrm{k}$. 
We follow the sign conventions as given in \cite[Ch. 1]{thebook} and its sequel \cite{thesequel}. In particular, an $A_\infty$-algebra $\mathscr{A}$ over $\mathrm{k}$ is a $\f{Z}$-graded $\mathrm{k}$-module with a collection of $\mathrm{k}$-linear maps:
\[ \mu^{d} : \mathscr{A}^{\otimes d} \to \mathscr{A}[2-d] , \ \ \text{for } d \geq 1, \]
where $[2-d]$ means $\mu^{d}$ lowers the degree by $d-2$. These maps are required to satisfy the $A_\infty$-relations:
\[ \sum_{m,n} (-1)^{|a_1|+\ldots + |a_n|-n} \mu^{d-m+1}(a_d,\ldots, a_{n+m+1}, \mu^m 
(a_{n+m},\ldots, a_{n+1}), a_n,\ldots a_1) = 0. \]
A DG-algebra over $\mathrm{k}$ is an $A_\infty$-algebra over $\mathrm{k}$ such
that $\mu^d =0$ for $d\geq 3$. In this case, we put \begin{equation}\label{dgsigns} da=(-1)^{|a|}
\mu^1(a) \ \ , \ \ a_2 a_1 = (-1)^{|a_1|} \mu^2(a_2,a_1). \end{equation} 
With this convention the $A_\infty$-equation for $d=2$ gives us the usual graded Leibniz rule: \[ d(a_2
a_1) = (da_2) a_1 + (-1)^{|a_2|} a_2 (d a_1). \]

$\mathscr{A}^{op}$ denotes the opposite of an $A_\infty$-algebra $\mathscr{A}$ and its operations are given by:
\[ \mu^{d}_{\mathscr{A}^{op}} (a_d,\ldots,a_1) = (-1)^{|a_1|+\ldots +|a_d|-d} 
\mu^d_{\mathscr{A}}(a_1,\ldots,a_d). \]

With the above conventions, a DG-algebra and its opposite are related as follows:
\[ d^{op}(a) = (-1)^{|a|-1} da \ \ \ , \ \ \ a_2 a_1 = a_1 a_2. \]

All our complexes are cohomological, i.e., the differential increases the
grading by $1$. It often happens that our complexes are bigraded. In this case,
we denote these gradings by the pair $(r,s)$ where $r$ refers to a
cohomological (or length) grading and $s$ refers to an internal grading
(the notation $|a|$ is used to denote the internal grading of a specific
element). The grading $r+s$ is referred to as the total degree. If a
second grading is not specified in the notation, for example as in
$HH^*(A_\Gamma)$, it is understood that the grading $*$ refers to the total
degree.  

The notation $HH^*(A)$ is used to denote Hochschild cohomology of a graded
$\f{K}$-algebra $A$ with coefficients in $A$. It is a bigraded algebra over $\f{K}$. We
write $\text{deg}(x)$ for the total degree $r+s$ of a specific element. There
are two binary $\f{K}$-linear operations: an associative graded commutative product of
bidegree $(0,0)$ and a Lie bracket of bidegree $(-1,0)$. 
These are called the cup product and Gerstenhaber bracket, respectively.  
The product is graded commutative: \[  xy =
	(-1)^{\text{deg}(x)\text{\ deg}(y)} yx. \] The Gerstenhaber bracket is
	graded antisymmetric on $HH^*(A)[1]$, that is: \[ [x,y] = -
	(-1)^{(\text{deg}(x)-1)(\text{\ deg}(y)-1)} [y,x] .\] Finally,
	Hochschild cohomology of a (formal) Calabi-Yau algebra can be equipped
	with a Batalin-Vilkovisky operator $\Delta$ of bidegree $(-1,0)$, and
	we have the following compatibility equation between these structures:
	\[ [x,y] = (-1)^{|x|} \Delta(xy) - (-1)^{|x|} \Delta(x) y - x \Delta(y)
\] 

{\bf Acknowledgments:} Lekili is partially supported by a Royal Society
Fellowship and the NSF grant DMS-1509141. We thank Mohammed Abouzaid, Ben Antieau, Georgios Dimitroglou Rizell, Tobias Ekholm, Sheel Ganatra, Travis Schedler, Paul Seidel and Ivan Smith. We are especially grateful to the referee for a careful reading of the manuscript: In an earlier version of this
paper, we used a more complicated Lagrangian projection than the one given in Fig.~(\ref{fig3}), which resulted in higher energy curves being immersed and elusive. We are indebted to the
referee for communicating to us the existence of these higher order contributions to the differential of $\mathscr{B}_\Gamma$.

\section{Plumbing of cotangent bundles of 2-spheres} 
\label{plumbing} 

Let $\Gamma$ be a finite tree. In the body of the paper, we will study Weinstein manifolds that are given by a plumbing of cotangent bundles of the 2-sphere according to the tree $\Gamma$. These are exact symplectic manifolds with a convexity condition at infinity. We briefly recall the construction of these manifolds (cf. \cite{abouzplum}).

Associated to each vertex of $\Gamma$, we prepare a copy of $D^*S^2$, the
unit cotangent bundle of the 2-sphere with its canonical symplectic
structure. Now, say we have an edge that connects the vertices $v$ and $w$, and
let us write $D^*S_v$ and $D^*S_w$ for the corresponding copies of $T^*S^2$ and
choose base points $s_v \in S_v$ and $s_w \in S_w$. Near $s_v$ and $s_w$ one
can find real coordinates $p_1, p_2, q_1,q_2$ where the coordinates $q_1,q_2$
correspond to variations on the base and the coordinates $p_1, p_2$
correspond to variations in the fibre direction. Furthermore, on these neighborhoods symplectic form can be identified with $dp \wedge dq$. We then glue $D^*S_v$
and $D^*S_w$ together near $s_v$ and $s_w$ via a symplectomorphism that sends $(q, p)$ to
$(p, -q)$. 

This leads to a symplectic manifold which has a boundary with corners. One
then smoothens the boundary and completes it to obtain a Weinstein manifold.
The precise details of this construction are somewhat technical; we refer to \cite[Sec. 2.3]{abouzplum} (see also \cite[Ch. 7.6]{geiges}). 

An alternative description of $X_\Gamma$ can be given via \emph{Legendrian
surgery} a la Eliashberg \cite{eli} and Gompf \cite{gompf}, which we will take
as primary.\footnote{Both the plumbing and surgery constructions lead to
homotopic Weinstein manifolds but we do not check this here. Throughout,
we use the surgery construction and appeal to the plumbing picture only for
differential topological aspects.} In this description, we exhibit $X_\Gamma$
as a surgery along a Legendrian link $\Lambda$ on $(S^3, \xi_{std})$ such that the
vertices $v$ of $\Gamma$ correspond to the components $\Lambda_v$ of this link, which are Legendrian unknots. Two such Legendrian unknots are ``plumbed together'' if
there is an edge in $\Gamma$ between the corresponding vertices. To be precise, by choosing a vertex as the \emph{root} of our tree, we put our tree $\Gamma$ in a standard form as in Fig.~(\ref{fig1}) and the corresponding Legendrian unknots in a standard form in $(\f{R}^3, dz - y dx)$, which when projected to $(x, z)$ (front projection) gives the surgery diagram as in Fig.~(\ref{fig2}). 

% Set the overall layout of the tree
\tikzstyle{level 1}=[level distance=3cm, sibling distance=1.8cm]
\tikzstyle{level 2}=[level distance=3cm, sibling distance=0.8cm]

% Define styles for bags and leafs
\tikzstyle{bag} = [circle, minimum width=3pt,fill, inner sep=0pt]
\tikzstyle{end} = [circle, minimum width=3pt,fill, inner sep=0pt]

\begin{figure}[htb!]
\centering
\begin{tikzpicture}[grow=right, sloped, scale=0.8]	
\node[bag, label=above: {$v_1$}]{}
  child {
       node[bag, label=above: {$v_4$}]{}    
       child {
       node[end, label=right:
       {$v_{10}$}]{}
       edge from parent}
   }
  child {
       node[bag, label=above: {$v_3$}]{}    
       child {
       node[end, label=right:
       {$v_9$}]{}
       edge from parent}
        child {
       node[end, label=right:
       {$v_8$}]{}
       edge from parent}
       child {
       node[end, label=right:
       {$v_7$}]{}
       edge from parent}
  }
  child {
       node[bag, label=above: {$v_2$}]{}    
       child {
       node[end, label=right:
       {$v_6$}]{}
       edge from parent}
        child {
       node[end, label=right:
       {$v_5$}]{}
       edge from parent}
};
\end{tikzpicture}
 \caption{Standard form of $\Gamma$}
 \label{fig1}     
\end{figure}
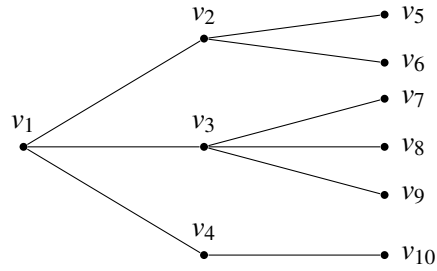

\begin{figure}[htb!]
\centering
\begin{tikzpicture}
	\draw (0,0) node[inner sep=0] {\includegraphics[scale=0.5]{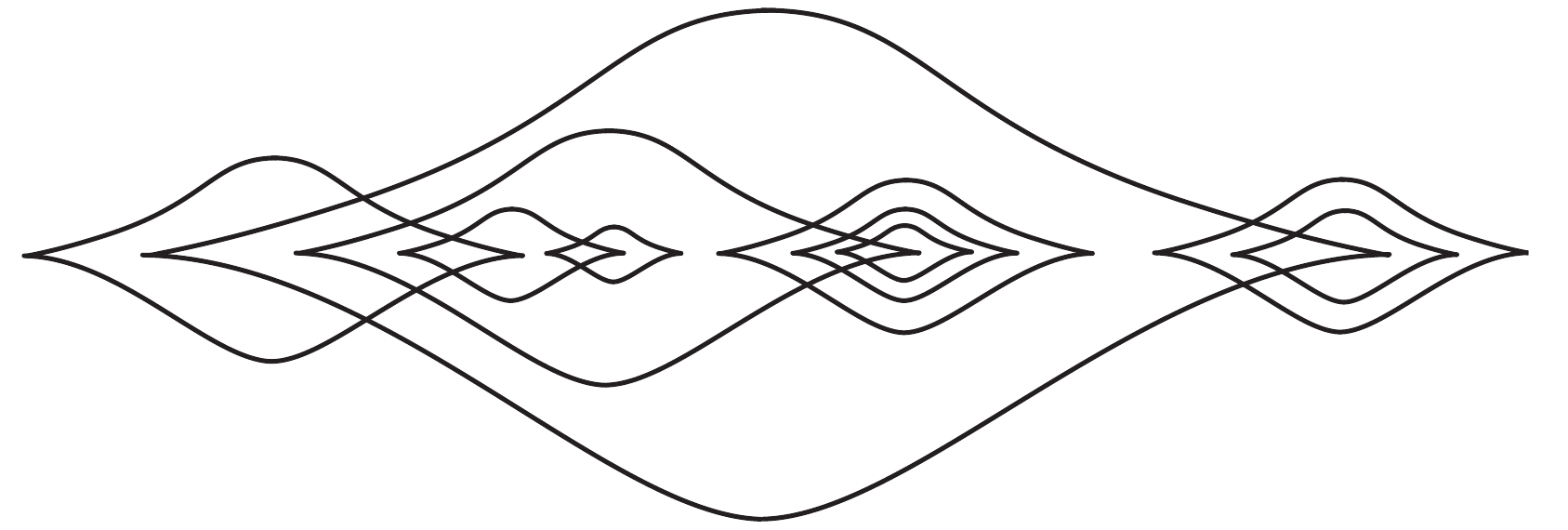}} ;
\node[] at (-5.5,1) {\footnotesize 1};
\node[] at (-2.5,1.5) {\footnotesize 2};
\node[] at (-0.5,1) {\footnotesize 3};
\node[] at (-2.1,0.6) {\footnotesize 4};
\node[] at (-1.2,0.39) {\footnotesize 10};
\node[] at (2.12,0.6) {\footnotesize 7};
\node[] at (1.7,0.51) {\footnotesize 8};
\node[] at (1.4,0.41) {\footnotesize 9};
\node[] at (4.5,0.8) {\footnotesize 5};
\node[] at (5.1,0.63) {\footnotesize 6};
\end{tikzpicture}
\caption{$X_\Gamma$ as given by Legendrian surgery along $\Lambda$}
\label{fig2}
\end{figure}
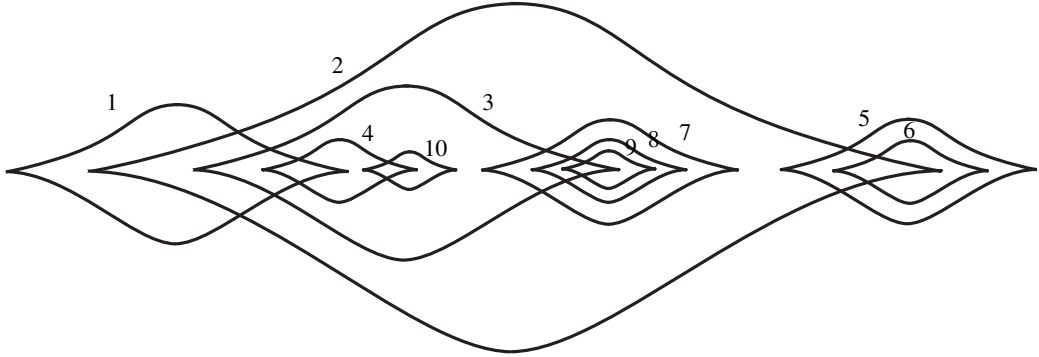

The surgery construction equips $X_\Gamma$ with a Weinstein structure (in fact,
a Stein structure) by extending the standard Weinstein structure on $D^4$ via
attaching 2-handles (\cite{weinstein}) along Legendrian unknots $\Lambda_i$.
Each such Legendrian unknot bounds an embedded Lagrangian disk in $D^4$ and
another capping disk given by the attaching disk of the corresponding Weinstein
$2$-handle. These fit together, as can be seen from the case of $T^*S^2$, to give the Lagrangian spheres $S_v$ in
$X_\Gamma$ corresponding to the vertices of $\Gamma$, whereas the edges of $\Gamma$
encode the intersection pattern of these spheres. The symplectic form $\omega$ on
$X_\Gamma$ is exact and it can be written as a primitive of a one-form $\theta$
for which the embedding of each sphere $S_v$ is an exact Lagrangian
submanifold of $X_\Gamma$. Both of these are easy facts since $H_2(X_\Gamma;\f{Z})$ is generated by the Lagrangian spheres $S_v$ and $H^1(S_v; \f{Z})=0$.

Furthermore, the cocores of the $2$-handles give non-compact (exact)
Lagrangians $L_v$ which are asymptotically Legendrian. The Lagrangian $L_v$ intersects $S_w$
only if $v=w$ in which case the intersection is transverse at a unique point
$x_v$. In the plumbing description, the $L_v$ correspond to the cotangent fibres
$T^*_{x_v} S_v \subset T^* S_v$ where the $x_v$ are base points on each $S_v$ away
from the plumbing regions.

In the next section, we will be concerned with Reeb chords between the components of $\Lambda$ in
$(\f{R}^3, dz - y dx)$. The Reeb flow is in the direction of the vector field
$\frac{\partial}{\partial z}$, hence it is more convenient for computations to consider the
Lagrangian projection, i.e., the projection to $(x,y)$ as in Fig.~(\ref{fig3}). Then, the crossings
of the projection $\Lambda$ are in one-to-one correspondence with Reeb chords from $\Lambda$ to
itself. There is some freedom in drawing the Lagrangian projection, we prefer the one given in
Fig.~(\ref{fig3}) as it makes enumeration of holomorphic curves manifest. (Notice that the diagram
has the property that each component links at most one other component on its left. Clearly this is
an artifact of the way we put our tree in a standard form and is not necessary.)

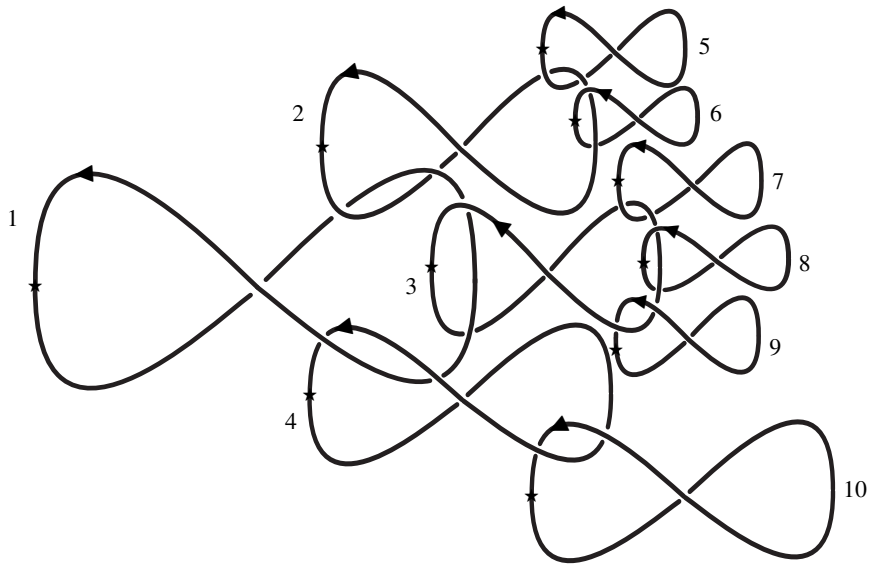
\begin{figure}[htb!]
\centering
\begin{tikzpicture}
\draw (0,0) node[inner sep=0] {\includegraphics[width=0.7\textwidth]{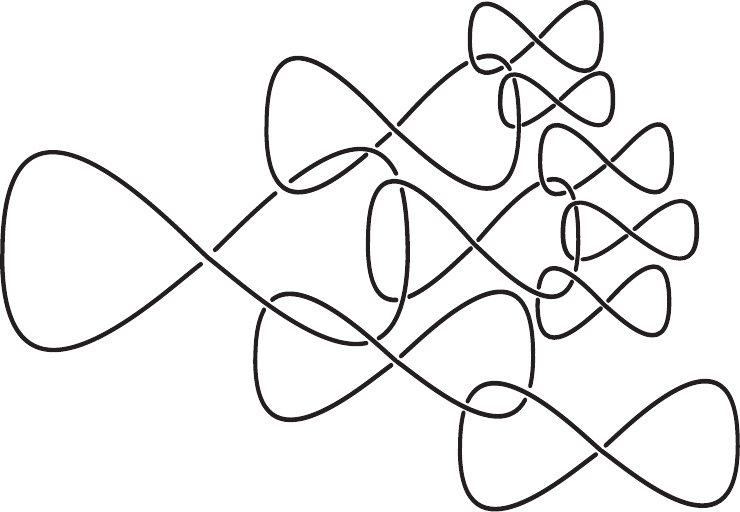} } ;

\node[] at (-5.6,0.9) {\footnotesize 1};

\node[] at (-1.80,2.3) {\footnotesize 2};
\node[] at (3.6,3.2) {\footnotesize 5};
\node[] at (3.75,2.3) {\footnotesize 6};

\node[] at (-0.3,0.) {\footnotesize 3};
\node[] at (2.58+2,1.4) {\footnotesize 7};
\node[] at (2.93+2,0.3)  {\footnotesize 8};
\node[] at (2.54+2,-0.82){\footnotesize 9};

\node[] at (-1.9,-1.8) {\footnotesize 4};
\node[] at (2.1+3.5,-2.7) {\footnotesize 10};

\node[] at (-5.3,0) {{\tiny $\bigstar$}};
\node[] at (-1.65,-1.45)  {{\tiny $\bigstar$}};
\node[] at (-1.48,1.85)  {{\tiny $\bigstar$}};
\node[] at (-0.03,0.25) {{\tiny $\bigstar$}};
\node[] at (1.3,-2.8){{\tiny $\bigstar$}};
\node[] at (1.45,3.15){{\tiny $\bigstar$}};
\node[] at (1.88,2.2){{\tiny $\bigstar$}};
\node[] at (2.45,1.4) {{\tiny $\bigstar$}};
\node[] at (2.42,-0.85)  {{\tiny $\bigstar$}};
\node[] at (2.79,0.3)  {{\tiny $\bigstar$}};

\node[] at (-4.65,1.49) {{\rotatebox{4}{$\blacktriangleleft$}}};
\node[] at (-1.65+0.44,-1.45+0.9)  {{\rotatebox{12}{$\blacktriangleleft$}}};
\node[] at (-1.48+0.35,1.85+0.98)  {{\rotatebox{15}{$\blacktriangleleft$}}};
\node[] at (0.88,0.80) {{\rotatebox{-36}{$\blacktriangleleft$}}};
\node[] at (1.3+0.35,-2.8+0.93){{\rotatebox{25}{$\blacktriangleleft$}}};
\node[] at (1.45+0.2,3.15+0.48){{\rotatebox{5}{\footnotesize $\blacktriangleleft$}}};
\node[] at (1.88+0.37,2.2+0.37){{\rotatebox{335}{\footnotesize
$\blacktriangleleft$}}};
\node[] at (2.45+0.27,1.4+0.47) {{\rotatebox{-8}{\footnotesize $\blacktriangleleft$}}};
\node[] at (2.42+0.3,-0.85+0.66)  {{\rotatebox{350}{\footnotesize
$\blacktriangleleft$}}};
\node[] at (2.79+0.35,0.3+0.45)  {{\rotatebox{340}{\footnotesize
$\blacktriangleleft$}}}; 
\end{tikzpicture}
\caption{Lagrangian projection of $\Lambda$ decorated with orientations and basepoints}
\label{fig3}
\end{figure}

In Fig.~(\ref{fig3}), besides a basepoint on each component, we also indicated an orientation on our Legendrian link $\Lambda$ by putting an arrow on each component. This, in turn, induces orientations on the Lagrangian spheres $S_v$. Notice that 
\[ S_v \cdot S_w = +1 \] 
if $v$ and $w$ are adjacent vertices. This ensures that the Floer complex $CF^*(S_v, S_w)$ is supported at an odd degree (see \cite[Sec. 2d]{seidelgraded}).

We orient the non-compact Lagrangians $L_v$ so that the algebraic intersection
number $L_v \cdot S_v$ is given by \[ L_v \cdot S_v = -1. \] As above, this ensures that the Floer complex $CF^*(L_v,
S_v)$ is supported at an even degree (which we will fix below to be 0 by
picking suitable grading structures).

The classical topology of $X_\Gamma$ is easy to study via the plumbing
description, which shows that $X_\Gamma$ deformation retracts onto a wedge of spheres formed by the union of the $S_v$. In particular, $X_\Gamma$ is simply connected and the non-zero cohomology
groups of $X_\Gamma$ are given by: \[ H^0(X_\Gamma; \f{K}) = \f{K} \ \ , \ \
	H^{2}(X_\Gamma ; \f{K} ) = \bigoplus_{v} \f{K} \cdot [S_v]^\vee . \]

The noncompact end of $X_\Gamma$ is a symplectization of a contact 3-manifold
$Y_\Gamma$ which is topologically a plumbing of circle bundles over $S^2$ with
Euler number $-2$. By abuse of notation, we will write $\partial X_\Gamma =
Y_\Gamma$. 

To equip our Lagrangians with a brane structure, so as to have
$\f{Z}$-gradings, we need : \begin{lemma} $c_1(X_\Gamma, \omega) = 0$.
\end{lemma} \begin{proof} We have $\langle c_1 (X_\Gamma) , [S_v] \rangle =
	\mathrm{rot}(\Lambda_v)$ (see \cite[Prop.~2.3]{gompf}). Now, each
	$\Lambda_v$ is an oriented Legendrian unknot in $(S^3, \xi_{std})$ and
	as such its rotation number can be computed to be
	$\mathrm{rot}(\Lambda_v) = 0$.  \end{proof} 

Therefore, the canonical bundle $\mathcal{K} = \Lambda_{\f{C}}^2 (T^* X_{\Gamma})$
representing $-c_1(X_\Gamma)$ is trivial. To define $\f{Z}$-gradings in various
Floer type invariants, one needs to fix a trivialization of
$\mathcal{K}^{\otimes 2}$. Of course, since $H^1(X_\Gamma) =0$, there is
actually only one homotopy class of trivializations. We can induce a trivialization by picking a complexified volume form $\Omega \in \Lambda^2_{\f{C}}(T^* X_\Gamma)$.

In this setup, a grading structure on a Lagrangian $L$ can be thought of as a lift of the squared-phase map: \[ \alpha_{L}: L \to S^1 \ \ \ , \ \ \ \alpha_{L}(x) = \frac{\Omega (T_x L)^2}{ | \Omega (T_x L)^2 |} \] 
to a map $\tilde{\alpha}_{L} : L \to \f{R}$. The fact that $S_v$ and $L_v$ are simply connected ensures that such a lift exists for our Lagrangians. 

A grading structure allows one to associate an absolute Maslov index in $\f{Z}$ to an intersection point $x \in S_v \cap S_w$ (see \cite[Sec. (2d)]{seidelgraded}).  
In our situation, all our Lagrangians $S_v$ are simply-connected and if any two of them intersect they do so at a unique point.
If $x$ is the intersection point of $S_v$ and $S_w$, then for any
given $d \in \f{Z}$ we can ensure that $x \in CF^{*}(S_v, S_w)$ lies in degree $d$ by shifting the grading structure on, say
$S_w$. When viewed as a generator of
$CF^*(S_w, S_v)$, the same intersection point would then be forced to have degree $2-d$ by Poincar\'e
duality in Floer cohomology of compact Lagrangians (see \cite[Sec.~12e]{thebook}). Furthermore, since $\Gamma$
is a tree, we can grade our Lagrangians inductively using the standard form of
$\Gamma$ as in Fig.~(\ref{fig1}). Therefore, we can grade all of our
Lagrangians $S_v$ at once such that for any pair of intersecting Lagrangians
$S_v$ and $S_w$ we are free to pick the gradings $(d,2-d)$ as we would like.
Collapsing a grading structure on a Lagrangian to a $\f{Z}_2$-grading, we get
an orientation of the underlying Lagrangian. To be compatible with the above
choice of orientations for the Lagrangian spheres $S_v$, we will need to demand
that the gradings $d$ be odd. Throughout, a convenient choice will be to simply
demand that $d=1$, that is: \[ CF^*(S_v, S_w) = \f{K} [-1]  \ \ \ \text{if}
\ \ \ v,w  \ \ \ \text{are adjacent.} \] Having graded the Lagrangian spheres
$S_v$ for all $v$, we now pick grading structures for the non-compact
Lagrangians $L_v$. As $L_v$ is simply-connected as well, we have the
freedom to choose a grading such that \[ CF^*(L_v, S_v) = \f{K} [0]. \]  
This is compatible with our choice of orientations on $L_v$ and $S_v$ as given
before. 

These considerations fix the orientations and the grading data up to an overall shift (which does
not change the degrees of intersection points) on our Lagrangians. (Note that there is a unique
choice of Spin structures as our Lagrangians are simply connected.)

Somewhat more nontrivially, these choices force that if $v$ and $w$ are adjacent vertices, then we have the following.
\begin{lemma} \label{nonpositive} For $v$ and $w$ adjacent vertices of the tree $\Gamma$, the shortest Reeb chord between $L_v$ and $L_w$ lies in the degree 0 part of $CW^*(L_v,L_w)$.  Furthermore, for any pair $v,w$, the complex $CW^*(L_{v},L_{w})$ is supported in nonpositive degrees. 
\end{lemma}
\begin{proof} The first claim follows from a rigidity of a certain holomorphic square that contributes to the higher multiplication
	\[ \mu^3: HF^0(L_{v}, S_v) \otimes HW^0(L_w,L_v) \otimes HF^2(S_w, L_w) \to HF^1(S_w, S_v) \] as explained in  \cite[Sec. 4.2]{abouzsmith}. The second claim is a consequence of the first by additivity properties of the Maslov grading (see \cite[Lem. 4.11]{abouzsmith}). 
\end{proof} 

We do not use the above result in our computations below. We have stated and
proved it as it helps motivate various grading choices (see also Rem.~(\ref{sftshit})). Let us also note that Thm.~(\ref{koszulness}) below provides an
indirect check of this lemma.

\section{Ginzburg DG-algebra of $\Gamma$ and Legendrian cohomology DG-algebra of $\Lambda_\Gamma$} \label{ginzburg}

\subsection{Ginzburg DG-algebra of $\Gamma$}
\label{ginzburgDGA}

A \emph{quiver} $Q$ is a directed graph with a vertex set $Q_0$ and an arrow
set $Q_1$. A rooted tree $\Gamma$ in a standard form, as in Fig.~(\ref{fig1}),
gives rise to a quiver by orienting the edges so that they point \emph{away}
from the root. We will denote this quiver again by $\Gamma$ unless otherwise
specified. Recall that the path algebra $\f{K}\Gamma$ of quiver $\Gamma$ is
defined as a vector space having all the paths in the quiver as basis
(including, for each vertex $v$ of the quiver $\Gamma$, a trivial path $e_v$ of
length $0$), and multiplication is given by concatenation of paths. As
mentioned before, throughout we concatenate paths from right-to-left, when we
express them as a product. 

The cohomologically graded 2-Calabi-Yau \emph{Ginzburg DG-algebra} $\mathscr{G}_\Gamma$ of $\Gamma$ (with zero potential) is defined as follows (cf. \cite{ginzburg}, \cite{vandenbergh} \cite{hermes}). 

\begin{definition} 
	\label{ginzdga} 
	Consider the extended quiver $\widehat{\Gamma}$ with vertices $\widehat{\Gamma}_0 = \Gamma_0$ and arrows $\widehat{\Gamma}_1$ consisting of 
\begin{itemize} 
 \item the original arrows $g$ in $\Gamma_1$ in bidegree (1, -1) 
 \item the opposite arrows $g^*$ to $g$ in $\Gamma_1$ in bidegree (1, -1)
 \item loops $h_v$ at the vertex $v \in \Gamma_0$ of bidegree (1, -2) 
\end{itemize}

We define $\mathscr{G}_\Gamma$ to be the DG-algebra over the semisimple ring $\mathrm{k} = \bigoplus_{v \in \Gamma_0} \mathbb{K}e_v$ given by the path algebra $\mathbb{K}\widehat{\Gamma}$ with the differential $d$ of bidegree (1,0) defined as a $\mathrm{k}$-bimodule map by	
\[ dg = dg^* = 0  \ \mbox{ and } \  dh = \sum_{g\in \Gamma_1} g^* g - g g^*  \]
where $h = \sum_{v\in \Gamma_0} h_v$. 
\end{definition}

In the notation $(r,s)$ for bigraded complexes, $r$ corresponds to the path-length grading and as usual we will call $r+s$ the total degree. In particular, the notation $H^*(\mathscr{G})$ will stand for the cohomology graded by the total degree. Note also that with respect to the total grading $\mathscr{G}_\Gamma$ is supported in nonpositive degrees. 

The way we chose to orient the edges of $\Gamma$ has only a minor
effect on $\mathscr{G}_\Gamma$. Namely, different choices change the signs in
the formula for the differential. Our choice is to ensure the consistency with
the choice of orientations of the Lagrangians $L_\nu$ as we shall see in the
next section. In particular, let $\Gamma^{op}$ be the quiver obtained from $\Gamma$  by reversing
the orientation of all edges of $\Gamma$. Then the associated Ginzburg algebra gives
$\mathscr{G}_\Gamma^{op}$, the opposite of the Ginzburg algebra $\mathscr{G}_\Gamma$ associated to
the original quiver $\Gamma$. In other words,
\[ \mathscr{G}_{\Gamma^{op}} = \mathscr{G}_\Gamma^{op} \]

\begin{definition} \label{preprojective} The cohomology in total degree 0 of $\mathscr{G}_\Gamma$ is
    called the \emph{preprojective algebra} $\Pi_\Gamma := H^0(\mathscr{G}_\Gamma)$. It is the
    quotient of the path algebra $\mathbb{K}\mathrm{D}\Gamma$ by the ideal generated by 
\[ \sum_{g\in \Gamma_1} g^*g - gg^* \ , \]
    where $\mathrm{D}\Gamma$ denotes the double of $\Gamma$ obtained by adding the opposite arrow $g^*$ for every $g \in \Gamma_1$. 
\end{definition} 
 
It turns out that the nature of the DG-algebra $\mathscr{G}_\Gamma$ depends on whether $\Gamma$ is
of Dynkin type or not as shown in the following theorem. It was first proven by Hermes
\cite{hermes} under the assumption that $\mathbb{K}$ is algebraically closed and characteristic
0. In Cor.~(\ref{rehermes}), we give a proof of the first part of the theorem over an arbitrary field.

\begin{theorem}(Hermes \cite{hermes}, and also Cor.~(\ref{rehermes})) \label{hermes} 
	\begin{enumerate}
	\item Suppose $\Gamma$ is non-Dynkin. Then $H^*(\mathscr{G}_\Gamma) = \Pi_\Gamma$ is supported in degree 0 and is quasi-isomorphic to $\mathscr{G}_\Gamma$. In other words, $\mathscr{G}_\Gamma$ is formal.
    \item Suppose $\Gamma$ is Dynkin and $\mathbb{K}$ is characteristic 0 and algebraically
        closed. Then \[  H^*(\mathscr{G}_\Gamma) \cong \Pi_\Gamma \rtimes_\nu \mathrm{k}[u] \ \ \ \ \ , \ \ \ \ \ |u| = -1 
	\] as a $\mathrm{k}$-algebra, where the multiplication is twisted by the Nakayama automorphism $\nu$ on $\Pi_\Gamma$.
	Furthermore, $\mathscr{G}_\Gamma$ is not formal and there is an $A_\infty$-structure $(\mu^n)_{n \geq 2}$ on the
	twisted polynomial algebra $\Pi_\Gamma \rtimes_\nu \mathrm{k}[u]$ making it a
	minimal model of $\mathscr{G}_\Gamma$. Moreover, this $A_\infty$-structure is $u$-equivariant and $\mu^n = 0$
	for $n \neq 2, 3.$ \end{enumerate} \end{theorem} 
	
The \emph{Nakayama automorphism} $\nu:\Pi_\Gamma \to \Pi_\Gamma$ in the above theorem refers to the automorphism defined by

\begin{equation*} \nu (g_{wv}) = \begin{cases} g_{\rho(w)\rho(v)} & \ \ \ \text{if} \ \ \ g_{wv} \in \Gamma  \ \ \text{or $ g_{\rho(w)\rho(v)}\in \Gamma$,}  \\ -g_{\rho(w)\rho(v)} & \ \ \ \text{if} \ \ \ g_{vw} , g_{\rho(v)\rho(w)}\in \Gamma\end{cases} \end{equation*}

where $g_{wv}$ denotes the arrow from the vertex $v$ to $w$ in $\Pi_\Gamma$, and $\rho$ denotes either the natural involution of the Dynkin graph (precisely when $\Gamma$ is of type $A_n, D_{2n+1}$ or $E_6$)  or the identity.  
We will abuse the notation and always denote the arrow from $v$ to $w$ by $g_{wv}$ regardless of
    where it is considered, in the quiver $\Gamma$, its double $D\Gamma$, the extended quiver
    $\widehat{\Gamma}$ or in the algebras $\mathscr{G}_\Gamma$ and $\Pi_\Gamma$ for that matter. In particular, $g_{vw}=g^*_{wv}$ if $g_{wv}$ belongs to $\Gamma$.
Note that $\nu$ has order at most $2$ and it is the identity if and only if $\Gamma$ is of type $A_1$ or it is of type $D_{2n}, E_7$ or $E_8$ and the base field $\f{K}$ is of characteristic $2$ (see \cite[Def. 4.6]{almostkoszul}).

\subsection{Legendrian cohomology DG-algebra of $\Lambda_\Gamma$} 

	We recall the definition of the $\f{Z}$-graded Chekanov-Eliashberg DG-algebra of the Legendrian link $\Lambda_\Gamma = \bigcup \Lambda_v $ following  \cite[Sec. 4]{BEE},  where it is denoted as $LHA(\Lambda_\Gamma)$. It was originally introduced in \cite{SFT}, \cite{chekanov}. 

Let $\mathcal{R}$ denote the finite set of Reeb chords from $\Lambda_\Gamma$ to
itself. Recall from Sec.~(\ref{plumbing}) that $\mathcal{R}$ is in bijection
with the set of crossings in the Lagrangian projection of
$\Lambda_\Gamma$ (Fig.~(\ref{fig3})). We endow the vector space
$\mathbb{K}\langle \mathcal{R} \rangle$ with a $\mathrm{k}$-bimodule structure by declaring 
	\[ e_w \mathcal{R}  e_v \] to be the set of Reeb
chords from $\Lambda_w$ to $\Lambda_v$. As a $\mathrm{k}$-module,
$LHA(\Lambda)$ is the tensor algebra over the semisimple ring $\mathrm{k}$
given by: \[ LHA_* (\Lambda_\Gamma) :=\bigoplus_{i=0}^\infty \f{K}\langle \mathcal{R}
\rangle^{\otimes_\mathrm{k} i}. \] 
After decorating $\Lambda_\Gamma$ with extra data by orienting each component and picking a base point at each component as in Fig.~(\ref{fig3}), the chords $c \in \mathcal{R}$ acquire a kind of Conley-Zehnder grading by $\f{Z}$ which we denote by $|c|$. 
The subscript in the notation of $LHA_*(\Lambda_\Gamma)$ denotes the induced grading on the tensor algebra. 
Elements $e_v \in \mathrm{k}$ have degree 0, however in general there may also be Reeb chords which have degree 0.
The differential $D : LHA_* (\Lambda_\Gamma) \to LHA_{*-1} (\Lambda_\Gamma)$ is defined as a map  $ D: \f{K}\langle \mathcal{R} \rangle_* \to LHA_{*-1}(\Lambda_\Gamma)$ and extended by the graded Leibniz rule to $LHA_* (\Lambda)$. 

Note that in general the differential is not compatible with the path-length grading corresponding
    to the index $i$ in the definition of $LHA_*(\Lambda)$. 

As we follow the cohomological convention to be consistent with the literature on Fukaya categories,  instead of $LHA_*(\Lambda)$ we will use the cohomologically graded DG-algebra $LCA^*(\Lambda)$. As a $\mathrm{k}$-module, it is given by:
\[ LCA^*(\Lambda_\Gamma) := LHA_{-*} (\Lambda_\Gamma) \]

The differential $D: LCA^*(\Lambda_\Gamma) \to LCA^{*+1}(\Lambda_\Gamma)$ is just carried over from the one on $LHA_* (\Lambda_\Gamma)$. 

Let us describe the Legendrian cohomology DG-algebra of $\Lambda_\Gamma$ more explicitly. The underlying algebra of $LCA^*(\Lambda_\Gamma)$ is the tensor algebra of the $\mathrm{k}$-bimodule $\mathbb{K}\langle \mathcal{R} \rangle$ generated by the Reeb chords (i.e. crossings in Fig.~(\ref{fig3})): \[ \mathcal{R} = \{ c_{wv} , c_{vw} : g_{wv} \in \Gamma_1 \} \cup \{ c_v : v \in \Gamma_0 \} \] 
where $c_v$ is the Reeb chord at the unique self-crossing of the component $\Lambda_v$, and for every two adjacent vertices $v$ and $w$ of the tree $\Gamma$, $c_{wv}$ corresponds to the unique Reeb chord from $\Lambda_w$ to $\Lambda_v$, i.e. the chord at the unique crossing between $\Lambda_v$ and $\Lambda_w$ where $\Lambda_w$ is the undercrossing component. 

Notice the remarkable coincidence of the $\mathrm{k}$-bimodule structure on $LCA^* (\Lambda_\Gamma)$ and the $\mathrm{k}$-bimodule structure on $\mathscr{G}_\Gamma$ from Def.~(\ref{ginzdga}). 
Next, we will see that the differentials do not agree in general. Nonetheless the Legendrian cohomology DG-algebra is isomorphic to a deformation of the Ginzburg algebra.

\begin{theorem} \label{balik} If $\Lambda_\Gamma$ is the Legendrian link in the standard form
    associated to the tree $\Gamma$ with Lagrangian projection in Fig.~(\ref{fig3}) with the grading decoration as
    indicated, then there is an isomorphism between $(LCA^*(\Lambda_\Gamma),D)$ and a deformation of
    $(\mathscr{G}_\Gamma,d)$ as DG-algebras. More precisely, there is a graded derivation
    $\mathfrak{d}: \mathscr{G}_\Gamma \to \mathscr{G}_\Gamma$ with homogeneous components
    $\mathfrak{d} = d_3 + d_5 + \ldots +
    d_{2m-1}$ for some $m\geq 1$, $d_{2i-1}$ having bidegree $(2i-1, 2-2i)$, and there is an isomorphism of DG-algebras
    \[ (LCA^*(\Lambda_\Gamma), D) \simeq (\mathscr{G}_\Gamma, d+ \mathfrak{d})  \]
such that the Conley-Zehnder degree on the left-hand-side agrees with the total degree on the right-hand-side. \end{theorem} 
\begin{proof}

 \emph{Generators:} The natural one-to-one correspondence, i.e., $g_{wv} \leftrightarrow c_{wv}$, $h_v \leftrightarrow c_v$, between the arrow set $\widehat{\Gamma}_1$ of the extended quiver $\widehat{\Gamma}$ and the set $\mathcal{R}$ of Reeb chords provides the isomorphism of the underlying $\mathrm{k}$-algebras, the path algebra $\mathbb{K} \widehat{\Gamma}$ and the tensor algebra of $\mathbb{K}\langle \mathcal{R} \rangle$. 
Note that the Reeb orientation of the chord $c_{wv}$ is from $\Lambda_w$ to $\Lambda_v$ whereas the arrow $g_{wv}$ goes from the vertex $v$ to $w$.

 \emph{Gradings:} It suffices to identify the gradings of the generators.  We first recall the definition for an arbitrary Legendrian link $\Lambda \subset (S^3 , \xi_{std})$.
  
According to the original combinatorial description \cite{chekanov}, $LCA$ has a
$\mathbb{Z}/r\mathbb{Z}$-grading, where $r$ is the $gcd$ of the rotation numbers of the components.
In our case, each component of $\Lambda_\Gamma$ is an unknot with rotation number $0$ providing a
$\mathbb{Z}$-grading on $LCA^*(\Lambda_\Gamma)$. 

Let $z_\pm$ be the endpoints of a Reeb chord $c$ of an oriented Legendrian link $\Lambda$ equipped with basepoints on every component, $z_+$ being the one with the greater $z$-coordinate. 
Let $\gamma_\pm$ be the shortest paths in $\Lambda$, from $z_\pm$ to the basepoint of the corresponding component, in the direction of the orientation of $\Lambda$.
The grading of $c$ in $LCA$ is defined to be $2r_--2r_++1/2$, where $r_{\pm} \in \mathbb{Q}$ is the number of counterclockwise rotations the tangent vector of $\gamma_\pm$ makes (in the $xy$-plane). 
It is straightforward to verify that the grading of  every generator of the form $c_v$  of $LCA(\Lambda_\Gamma)$ is $-1$ and that of the form $c_{wv}$ is $0$. 

\emph{Differential:}  We briefly recall the definition of the differential of $LCA$ for any Legendrian link in the standard contact $S^3$, and then compute the differentials on the set $\mathcal{R}$ of generators of $LCA^*(\Lambda_\Gamma)$. The rest will be determined by the Leibniz rule.

To simplify the definition, we arrange that at every crossing of the Lagrangian projection, the understrand and the overstrand have slopes $+1$ and $-1$, respectively. 
We also use the same notation for a crossing in the Lagrangian projection as the corresponding Reeb chord.

First of all, each quadrant around a crossing in the Lagrangian projection is decorated with a Reeb sign. 
The right and left quadrants at a crossing have positive signs whereas the top and bottom quadrants have negative signs. 

There is also a second set of signs, orientation signs, for these quadrants.  Every quadrant has
    orientation sign $+1$ except for the bottom and right quadrants at an even-graded crossing which
    are decorated with $-1$ as in Fig.~(\ref{fig4}). In fact, the choice of orientation signs for a
    given diagram depends on an isotopy of the diagram near the crossing so that the strand with
    a positive slope goes under the strand with a negative slope as in Fig.~(\ref{fig4}). We indicated our
    choice in the upper left diagram of Fig.~(\ref{fig5}). This affects the signs but
    different choices give isomorphic DG-algebras (see \cite[pg. 80]{EkNg}).

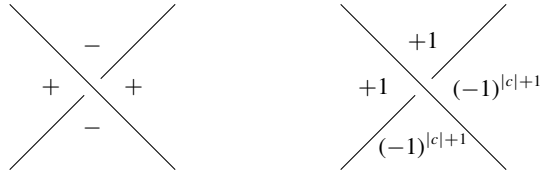
\begin{figure}[htb!]
\centering
\begin{tikzpicture} [scale=1.1]
\draw (0,0) -- (0.9,0.9);
\draw (1.1, 1.1) -- (2,2);
\draw (0,2) -- (2,0);
\draw (4,0) -- (4.9,0.9);
\draw (5.1,1.1) -- (6,2);
\draw (4,2) -- (6,0);
\node at (0.5,1)   {\footnotesize $+$};
\node at (1.5,1)   {\footnotesize $+$};
\node at (1,1.5)   {\footnotesize $-$};
\node at (1,0.5)   {\footnotesize $-$};
\node at (4.4,1)   {\footnotesize $+1$};
\node at (5.9,1)   {\footnotesize $(-1)^{|c|+1}$};
\node at (5,1.55)   {\footnotesize $+1$};
\node at (5,0.3)   {\footnotesize $(-1)^{|c|+1}$};
\end{tikzpicture}
\caption{Reeb signs (left) and orientation signs (right) at a crossing $c$}
\label{fig4}
\end{figure}

On a generator, the differential is given by a count of immersed polygons and it is extended by the graded Leibniz rule. The polygons taken into account are in the $xy$-plane with boundary on the Lagrangian projection of the link and vertices at the crossings. It is also required that at all but one vertex of the polygon, the quadrant included in the polygon should have a negative \emph{Reeb} sign.
Suppose that $\Delta$ is such an immersed polygon whose positive vertex is at $c$ and the negative vertices $c_1, c_2 , \dots , c_m$ are in order as we traverse the boundary of $\Delta$ counterclockwise starting at $c$. 
Note that $m$ may be $0$ and the $c_i$ are not necessarily distinct. 
If $b$ is the total number of times the boundary of $\Delta$ passes through basepoints of the Legendrian link, the orientation sign $\epsilon_{_{\Delta}}$ is defined to be $(-1)^b$ times the product of the \emph{orientation} signs at the vertices.

With this setup, we have
$$ d c = \sum_{\Delta} \epsilon_{_{\Delta}} c_m c_{m-1} \cdots c_1 $$
for any generator $c$. Observe that the differential of a generator of the form $c_{wv}$ vanishes since it has grading $0$ and $LCA^*(\Lambda_\Gamma)$ is nonpositively graded.
Again for grading reasons, any negative vertex of an immersed polygon which contributes to the differential of a generator $c_v$ is of type $c_{uw}$. 

In the rest of the proof we will show that 

$$D (c_v) = - \sum_{u:  g_{vu} \in \Gamma_1} c_{vu} c_{uv} + \sum_{i\geq 1} \sum_{ \scriptsize \begin{array}{c} w_1, \dots , w_i \\  g_{w_jv} \in \Gamma_1 \\ w_1 < \cdots < w_i   \end{array}} c_{vw_1} c_{w_1v} \dots  c_{vw_i} c_{w_iv} , $$

where the ordering in the last summation refers to the clockwise ordering of the components of
$\Lambda_{\Gamma}$ which are linked to $v$ from the right in the Lagrangian projection in
Fig.~(\ref{fig3}), e.g. the natural ordering of the integers associated to components in
Fig.~(\ref{fig3}). Note that the second sum not only corresponds to higher order terms in the
length filtration, it also contributes terms of wordlength 2 of
the form $c_{vw_1} c_{w_1 v}$. Indeed, all the terms of wordlength 2 that appear in the image of
$D(c_v)$ precisely correspond to $d(c_v)$ in $\mathscr{G}_\Gamma$. In particular, the first sum has at most one term as long as our Legendrian link is associated to a tree in the standard form.

\begin{figure}[!h]
\centering
\begin{tikzpicture}
	\draw (0,0) node[inner sep=0] {\includegraphics[scale=0.95]{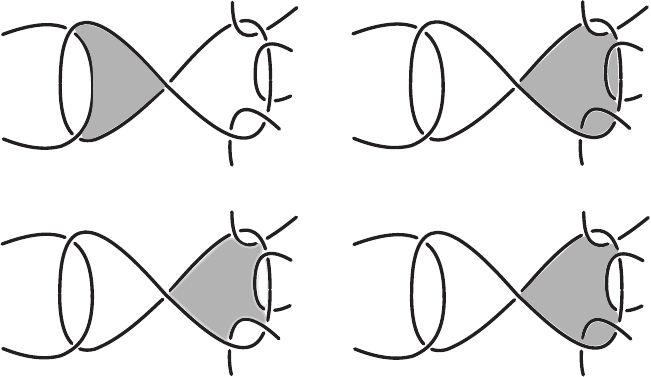}} ;
	\node[] at (-4.15,2.68) {\tiny $-$};
    \node[] at (-3.87,2.52) {\tiny  $-$};
    \node[] at (-4.2,0.80) {\tiny $-$};
    \node[] at (-3.95,0.65) {\tiny  $-$};
    \node[] at (-1.52,2.52) {\tiny $-$};
    \node[] at (-1.30,2.60) {\tiny  $-$};
    \node[] at (-1.32+0.45,2.64-0.15) {\tiny  $-$};
    \node[] at (-1.345+0.27,2.61-0.2) {\tiny  $-$};
    \node[] at (-1.54+0.75,2.52-0.3) {\tiny $-$};
    \node[] at (-1.32+0.5,2.63-0.25) {\tiny  $-$};
    \node[] at (-1.29+0.47,1.33-0.17) {\tiny  $-$};
    \node[] at (-1.32+0.43,1.33-0.35) {\tiny  $-$};
    \node[] at (-1.32+0.55,1.33) {\tiny  $-$};
    \node[] at (-1.32+0.3,1.33) {\tiny  $-$};
    \node[] at (-1.37,0.9) {\tiny  $-$};
    \node[] at (-1.37,0.75) {\tiny  $-$};
    \node[] at (-4.253,1.6)  {{\tiny $\bigstar$}};
    \node[] at (-1.127,1.8)  {{\tiny $\bigstar$}};
    \node[] at (-1.49,3.0)  {{\tiny $\bigstar$}};
    \node[] at (-1.53,0.55)  {{\tiny $\bigstar$}};

\end{tikzpicture} \caption{The polygons which correspond to the words in the differential $D(c_v)$:
    (from top left in clockwise order) a triangle (with a negative orientation sign), a triangle, a
    pentagon, and a heptagon (all with positive orientation signs). The quadrants with negative
    orientation signs and the basepoints are indicated in the upper left diagram.} \label{fig5} \end{figure}
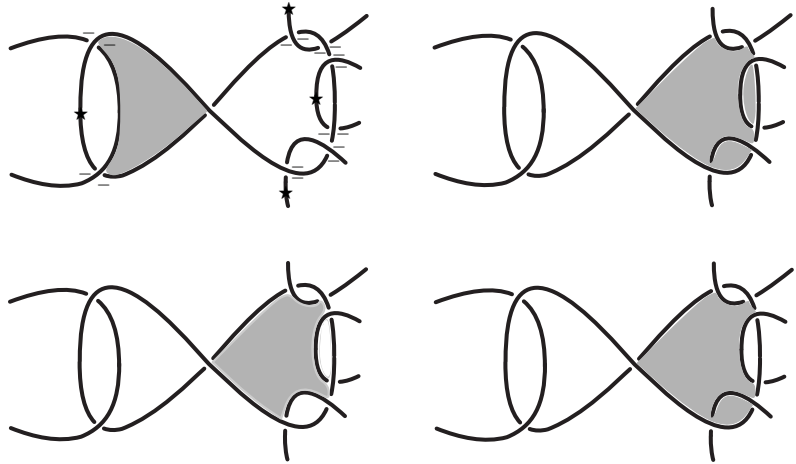

We will prove that all the terms in the above differential are induced by \emph{embedded} polygons as indicated in
Fig.~(\ref{fig5}), the relevant piece of the Lagrangian projection given in Fig.~(\ref{fig3}), together with the orientation signs at the crossings. 
There are also two unigons with a unique vertex at $c_v$, one to the left and the other to the
right with canceling contributions to the differential $D(c_v)$ since they come with opposite signs. 

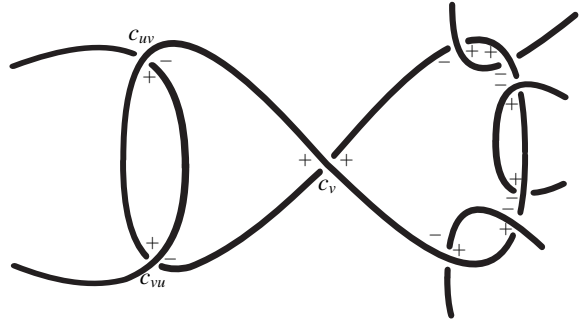
\begin{figure}[!h]
\centering
\begin{tikzpicture}
	\draw (0,0) node[inner sep=0] {\includegraphics[scale=1.5]{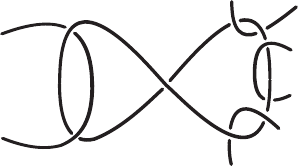}} ;
	
	\node[] at (0.15,0) {\tiny $+$};
        \node[] at (0.7,0) {\tiny  $+$};
        \node[] at (0.45,-0.35) {\footnotesize  $c_v$};
        
        \node[] at (-1.92,1.1) {\tiny $+$};
        \node[] at (-1.92+0.22,1.1+0.22) {\tiny  $-$};
          \node[] at (-2.02,1.65) {\footnotesize  $c_{uv}$};
        
	\node[] at (-1.88,-1.1) {\tiny  $+$};
	\node[] at (-1.88+0.24,-1.1-0.22) {\tiny  $-$};
	 \node[] at (-1.87,-1.6) {\footnotesize  $c_{vu}$};
	
	\node[] at (1.88,-1.) {\tiny  $-$};
	\node[] at (2.2,-1.2) {\tiny  $+$};
	
	\node[] at (2,1.3) {\tiny  $-$};
	\node[] at (2.37,1.45) {\tiny  $+$};
	
		\node[] at (2+0.75,1.15) {\tiny $-$};
	\node[] at (2.37+0.25,1.44) {\tiny  $+$};
	
		\node[] at (2+0.75,1.) {\tiny  $-$};
	\node[] at (2.37+0.53,0.75) {\tiny  $+$};
	
			\node[] at (2.9,-0.5) {\tiny  $-$};
	\node[] at (2.95,-0.25) {\tiny  $+$};
	
	\node[] at (2.85,-0.65) {\tiny  $-$};
	\node[] at (2.82,-0.92) {\tiny $+$};

\end{tikzpicture}
\caption{Reeb signs}
\label{fig6}
\end{figure}

We now prove that there are no other immersed polygons which contribute to the differential $D(c_v)$. To begin with,
any such polygon has a (Reeb-) positive vertex at $c_v$ (see Fig.~(\ref{fig6}) for the Reeb signs at
the relevant crossings). Start traversing its boundary in the counterclockwise direction assuming
that the polygon includes the left quadrant at $c_v$. 
If it has a vertex other than $c_v$, i.e. if it is not the unigon canceled by a similar unigon to the right, then the only option for an initial negative vertex is at $c_{uv}$ because of the configuration of the Reeb signs. 
Moreover, this vertex has to be followed (as we continue traversing the boundary) by a vertex at $c_{vu}$ since otherwise the polygon would intersect the region outside the Lagrangian projection which is prohibited. 
Similar considerations imply that a polygon which includes the right quadrant at $c_v$ can only have vertices at the crossings of $\Lambda_v$ with other components of $\Lambda$ as shown in Fig.~(\ref{fig5}) above so as not to intersect the noncompact region.  
\end{proof}

\begin{remark}  A relation between Ginzburg's construction of CY3 DG-algebras associated with quivers (with
potentials) and Fukaya categories of certain quasi-projective 3-folds also appears in the work of Smith \cite{smith}. 
\end{remark}

\begin{remark}
\label{sftshit} 
Recall that $LCA^*(\Lambda_\Gamma)$ is associated to the Legendrian attaching
spheres $\Lambda_v$ of Weinstein 2-handles. Stated results of \cite{BEE}
provide a dual picture given in terms of the wrapped Floer cohomology of the
cocores $L_v$ of these handles induced by cobordism maps associated to the
handle attachments. Namely, there is a grading preserving quasi-isomorphism of
	$A_\infty$-algebras: \[ LCA^*(\Lambda_\Gamma) \simeq \bigoplus_{v,w}
	CW^*(L_v,L_w). \] 
A rigorous justification of the equivalence of these two dual pictures is not fully established at
	this time. However, a detailed sketch of proof based on the results of \cite{BEE} has recently appeared in \cite[Thm. 2]{EL}. We must emphasize that we do not make use of this correspondence anywhere in our computations. Rather, this appealing geometric picture serves us as a guide to find the correct algebraic statement to be proven rigorously. 
\end{remark} 

\subsubsection{Recourse to deformation theory of DG-algebras} 

As a consequence of the explicit computation given above we can see the Legendrian cohomology DG-algebra $LCA^*(\Lambda_\Gamma)$ as a deformation of the Ginzburg DG-algebra $\mathscr{G}_\Gamma$. Therefore, it is natural to check whether
this deformation is trivial or not (up to equivalence). We recall here the basics of deformation theory of DG-algebras and exploit it to
determine the relationship between our computation of $LCA^*(\Lambda_\Gamma)$ and the Ginzburg
DG-algebra $\mathscr{G}_\Gamma$. A classical reference for this material is
\cite{GersWilker}. A recent exposition close to our purpose appears in
\cite[Appendix A]{seidel2half}. 

Unfortunately, these methods do not help directly as they apply in the setting of formal
deformations (such as a deformation over $\mathrm{k}[[t]]$) whereas here we have that
$LCA^*(\Lambda_\Gamma)$ is a global deformation of $\mathscr{G}_\Gamma$ (over $\mathrm{k}[t]$).
Nonetheless, it is helpful to start at the formal level and observe that we can arrange for a
globalisation in certain cases.

There is a decreasing, exhaustive, bounded above filtration on the complex $LCA^*(\Lambda_\Gamma)$ 
\[ \mathcal{F}^{0} := LCA^*(\Lambda_\Gamma) \supset \mathcal{F}^1 := \bigoplus_{i=1}^\infty \mathbb{K} \langle \mathcal{R}
    \rangle^{\otimes_\mathrm{k}^i} \supset \ldots \supset  \mathcal{F}^p := \bigoplus_{i=p}^\infty \mathbb{K} \langle \mathcal{R}
            \rangle^{\otimes_\mathrm{k}^i} \supset \ldots 
\]
Let us write $(LCA^*(\Lambda_\Gamma), D) = (\mathscr{G}_\Gamma, d_1+ d_2 + \ldots d_m)$, for some finite $m$,
where $d_i : \mathcal{F}^p \to \mathcal{F}^{p+i}$ is the $i^{th}$ homogeneous piece of the
differential.  
Observe that $d_1= d$ can be identified as the differential in the Ginzburg DG-algebra.  
It follows from $\mathrm{k}$-linearity of the differential that in fact $d_i$ is identically zero for even $i$. 
Note also that since $\mathscr{G}_\Gamma$ is bigraded, this complex is doubly graded. 
Denoting the second grading by$s$, we have $s(d_{2i-1}) = 2-2i$.  

Now, the first non-trivial $d_i$ for $i>1$ is possibly
$d_3$. Because $D^2=0$, using the filtration, we deduce that
 \[ d_1 d_3 + d_3 d_1= 0 \]
 Recall that the reduced bar complex $(\hom_\mathrm{k} (\mathrm{T}\overline{\mathscr{G}_\Gamma},
 \mathscr{G}_\Gamma), \delta = \delta_0 + \delta_1) $ can be used to compute Hochschild cohomology of
 $\mathscr{G}_\Gamma$. Here, we only need the explicit form of the Hochschild differential for
 elements $\phi \in \hom_{\mathrm{k}}(\overline{\mathscr{G}_\Gamma},
 \mathscr{G}_\Gamma)$ (see formula in \cite[Eqn. 1.8]{thebook}, which we adapted using
 DG-algebra conventions given in the introduction). For such $\phi$, we have
 \[     (-1)^{|\phi|+|b|}(\delta_0 \phi) (a  \otimes_{\mathrm{k}} b) = a \phi(b) +
 (-1)^{|\phi||b|} \phi(a) b - \phi(ab)  \ \ , \] 
\[ (-1)^{|\phi|+|a|} (\delta_1 \phi) (a) = d\phi(a) - \phi(da) \ \ . \] 
 
By definition, $\mathscr{G}_\Gamma$ is bigraded and its differential has bidegree
$(1,0)$, so the Hochschild cochain complex $CC^*(\mathscr{G}_\Gamma,
\mathscr{G}_\Gamma) = \hom_\mathrm{k} (\mathrm{T}\overline{\mathscr{G}_\Gamma},
 \mathscr{G}_\Gamma)$ has three gradings: the cohomological degree, the degree induced by the
 total degree $r+s$ on $\mathscr{G}_\Gamma$ and the internal grading induced by the second grading
 $s$ on $\mathscr{G}_\Gamma$. However, the Hochschild differential $\delta = \delta_0 +
 \delta_1$ is homogeneous (of degree 1) with respect to the sum of the first two gradings and it also preserves the
 internal degree, hence we get a bigrading on 
 $HH^*(\mathscr{G}_\Gamma,\mathscr{G}_\Gamma)$ which we write as
 \begin{equation} \label{decomp} HH^*(\mathscr{G}_\Gamma, \mathscr{G}_\Gamma) \cong \bigoplus_{r,s} HH^r
 (\mathscr{G}_\Gamma,\mathscr{G_\Gamma}_\Gamma[s]),  \end{equation} 
 where $r$ is the total degree (the sum of the cohomological degree and the degree induced by the total degree
 on $\mathscr{G}_\Gamma$) and $s$ is the internal grading induced by the internal grading on
 $\mathscr{G}_\Gamma$. 

Now, the fact that $d_3$ is a degree 1 derivation which anti-commutes with $d_1$ means that the
 sign-modified map $\tilde{d}_3 \in
 \hom^1_{\mathrm{k}}(\overline{\mathscr{G}_\Gamma}, \mathscr{G}_\Gamma)$ defined by 
 \[ \tilde{d}_3 a  = (-1)^{|a|} d_3 a\]
is closed under the Hochschild
differential. This yields the first obstruction class of the deformation:
\[ [\tilde{d}_3] \in HH^2(\mathscr{G}_\Gamma, \mathscr{G}_\Gamma[-2]) \]
If this class is trivial, choosing a trivializing class  
$\phi_2 \in \hom^{0}_{\mathrm{k}}(\overline{\mathscr{G}_\Gamma}, \mathscr{G}_\Gamma [-2])$,
we get a map $\phi_2$ for which we have: 
\[ d_3 = d \phi_2 - \phi_2 d \] 
Note that $\phi_2$ is induced by a map $\mathbb{K}\langle \mathcal{R}
\rangle \to \mathbb{K}\langle \mathcal{R} \rangle^{\otimes_\mathrm{k} 3}$ . 
Therefore, we can consider an \emph{algebra} map  \[ \Phi_2 = \mathrm{Id} + \phi_2 :
\mathscr{G}_\Gamma \to \mathscr{G}_\Gamma \] defined initially as a map on $\mathbb{K} \langle
\mathcal{R} \rangle \to \mathscr{G}_\Gamma$ and then extended to an algebra map.

Then, we would like to define a new differential $D'$ on $\mathscr{G}_\Gamma$ of the form 
\[ D'= d+ d'_5+ \ldots  \]
so that $\Phi_2: (\mathscr{G}_\Gamma, D') \to (\mathscr{G}_\Gamma, D)$ is a chain map (in addition to
being an algebra map). The obvious candidate for $D'$ is given by:
\[ D' = (\mathrm{Id} - \phi_2 + \phi_2^2  - \ldots ) \circ D \circ ( \mathrm{Id} +\phi_2) \] 
However, the alternating sum $(\mathrm{Id} - \phi_2 + \phi_2^2  - \ldots )$ will in general be an
infinite series, therefore, to make sense of this we need to consider the completion of $\mathscr{G}_\Gamma$ with respect to the length
filtration $\mathcal{F}^\bullet$: 
\[ \widehat{\mathscr{G}_\Gamma} = \varprojlim_{p} \mathscr{G}_\Gamma /\mathcal{F}^p
\mathscr{G}_\Gamma \] 
The differential $D$ of $LCA^*(\Lambda_\Gamma)$ extends naturally to
$\widehat{\mathscr{G}_\Gamma}$. We write the resulting complex as:
\[ \widehat{LCA}(\Lambda_\Gamma)  = (\widehat{\mathscr{G}_\Gamma}, D) \] 
Concretely, we can write the underlying $\mathrm{k}$-bimodule as: $\widehat{LCA}(\Lambda_\Gamma) = \mathbb{K} \langle \mathcal{R}
\rangle [[t]]$, where $t$ is a formal parameter in degree 0. In other words, we now allow formal power series in Reeb chords. 

We can now proceed with the construction mentioned above. Notice that since $\phi_2$ increases the
length by $2$, there is no convergence issue for the series $(\mathrm{Id} - \phi_2 + \phi_2^2  -
\ldots )$ on $\widehat{\mathscr{G}_\Gamma}$. Therefore, we have a filtered DG-algebra map
\[ \Phi_2 : (\widehat{\mathscr{G}_\Gamma}, D') \to (\widehat{\mathscr{G}_\Gamma}, D) \] 
which by construction is a chain map with an inverse, hence is in particular a quasi-isomorphism. 

We can then focus on the complex $(\widehat{\mathscr{G}_\Gamma}, D'=d +d_5'+ \ldots)$. As
before, we have that $d'_5$ is a derivation which anti-commutes with $d$, hence
the sign-twisted map $\tilde{d}'_5$ leads to an obstruction class $[\tilde{d}'_5] \in
HH^2(\mathscr{G}_\Gamma, \mathscr{G}_\Gamma[-4])$. If this vanishes we can continue along and find a
quasi-isomorphism of the form $\mathrm{Id} + \phi_4$. Iterating
this argument infinitely many times (which we can do as each quasi-isomorphism increases the
length), we obtain the following lemma (cf. \cite[Lemma
A.5]{seidel2half}).  \begin{lemma} \label{deftheory} Suppose that $HH^2(\mathscr{G}_\Gamma, \mathscr{G}_\Gamma[s]) = 0$ for all $s <0$,
    then there exists a quasi-isomorphism of completed DG-algebras:
\[ (\widehat{\mathscr{G}_\Gamma}, d) \simeq (\widehat{LCA}(\Lambda_\Gamma),D) \] \end{lemma} We next apply these ideas to the case where $\Gamma = D_n$ and show that 
all the obstructions vanish in this case. Furthermore, we prove that one can truncate the above quasi-isomorphism eliminating the
need for completions. Here, we make use of the results of Sec.~(\ref{typeD}) where
$HH^*(\mathscr{G}_\Gamma, \mathscr{G}_\Gamma)$ is computed for 
$\Gamma = D_n$. We would like to point out that the computation given there is independent of the conclusions we are drawing here. 

The following lemma is the key technical result that we will use to truncate the
quasi-isomorphism given on completions by the above deformation theory argument.

\begin{lemma} \label{Hausdorff}  Let $\mathcal{F}^\bullet$ denote the length filtration on
    $LCA^*(\Lambda_{D_n})$. For each grading $k$, there exists a $p(k)$ such that for all $p \geq
    p(k)$  we have that
    \[ \mathcal{F}^p H^k(LCA(\Lambda_{D_n})) = \mathrm{Im}\left(H^k ( \mathcal{F}^p LCA(\Lambda_{D_n})) \to H^k(
	LCA(\Lambda_{D_n}))\right)  = 0 \]
    In particular, for all $k$, the filtration on $H^k( LCA(\Lambda_{D_n}))$ induced by
$\mathcal{F}^\bullet$ is complete and Hausdorff. 
\end{lemma} 

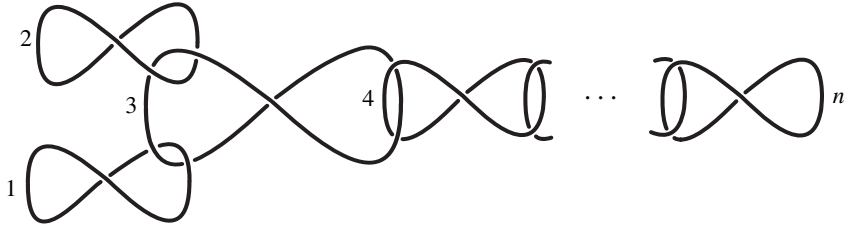
\begin{figure}[!h]
\centering
    \begin{tikzpicture}
\draw (0,0) node[inner sep=0] {\includegraphics[scale=0.95]{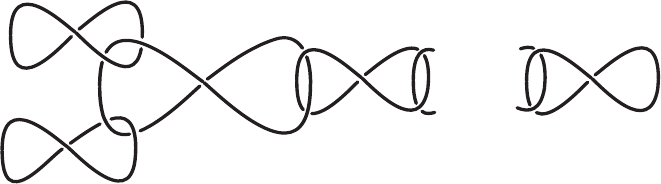}} ;
\node at (-5.3,1)   {\footnotesize $2$};
\node at (-5.5,-1)   {\footnotesize $1$};
\node at (-3.9,0.1)   {\footnotesize $3$};
\node at (-0.75,0.2)   {\footnotesize $4$};
\node at (5.5,0.2)   {\footnotesize $n$};
\path (-0.75,0.2)  to node {\dots}  (5.5,0.2) ;

\end{tikzpicture}
\caption{Lagrangian projection of a Legendrian link associated to the $D_n$ tree}
\label{fig7}
\end{figure}

\begin{proof} Consider the Lagrangian projection in Figure (\ref{fig7}). The proof of Thm.~(\ref{balik}) gives us the following description of the differential on $(LCA^*(\Lambda_{D_n}), D)$ 
\begin{align*}
    D c_1 & = c_{13} c_{31} \\
    D c_2 & = c_{23} c_{32} \\
    D c_3 &=  - c_{31} c_{13} - c_{32} c_{23} + c_{34} c_{43} - c_{31} c_{13} c_{32} c_{23} \\
       D c_4 & = - c_{43} c_{34} + c_{45} c_{54} \\
     & \cdots \\ 
    D c_{n-1} &= -c_{(n-1)(n-2)} c_{(n-2)(n-1)} + c_{(n-1)n} c_{n(n-1)}  \\
    D c_n &= - c_{n(n-1)} c_{(n-1)n} 
\end{align*}

where the gradings are given by $|c_i|=-1$ and $|c_{ij}| =0$. In particular,
$H^*(LCA(\Lambda_{D_n}))$ is supported in non-positive degrees.

Notice that $D = d_1 + d_3$, where $d_1$ is the differential on the Ginzburg DG-algebra 
$\mathscr{G}_{D_{n}}$ and $d_3$ is zero on all the generators except $c_3$, and we have
\[ d_3(c_3) = - c_{31} c_{13} c_{32} c_{23}. \]
We shall first establish the result for $H^0(LCA(\Lambda_{D_n}))$ by direct computation. The
        goal here is to take any word in $c_{ij}$ and prove that if the word is long enough, then it is actually null-homologous. 
        
Note that we have a decomposition
\[ H^0(LCA(\Lambda_{D_n})) \cong \bigoplus_{i,j=1}^n e_i H^0(LCA(\Lambda_{D_n}))e_j \]

Letting $x=c_{31}c_{13}$, $y=c_{32}c_{23}$ and $z=c_{34}c_{43}$ we obtain
\[ e_3 H^0(LCA(\Lambda_{D_n})) e_3 \cong \mathbb{K} \langle x,y,z \rangle /(x^2,y^2,z^{n-2}, x+y+xy-z)
\]
		(cf. \cite[Prop. 11.3.2 (i)]{schedler2}). Indeed, we have $$x^2= D (c_{31} c_1 c_{13}) \ \ , \ y^2 = D(
c_{32} c_2 c_{23})\ \ , \ x+y+xy-z = D(-c_3) \ \ . $$

Next, observe that for $4 \leq i \leq n-1$, we have $ c_{i (i-1)} c_{(i-1)i} = c_{i (i+1)} c_{(i+1)i} \in H^0(LCA(\Lambda_{D_n})) $ since their difference is precisely $Dc_i$. Consequently, 
we get 
\begin{align*}
z^{n-2} &= c_{34} (c_{43} c_{34})^{n-3} c_{43} = c_{34} (c_{45} c_{54})^{n-3} c_{43} = c_{34} c_{45} (c_{56} c_{65})^{n-4} c_{54} c_{43} =  \cdots \\ & = c_{34} c_{45} \ldots c_{(n-1)n}  c_{n(n-1)}c_{(n-1)n}  c_{n(n-1)} \ldots c_{54} c_{43} \\
&=  
D(- c_{34} c_{45} \ldots c_{(n-1)n} c_n c_{n(n-1)} \ldots c_{54} c_{43})
\end{align*}

Furthermore, any word in $e_3 H^0(LCA(\Lambda_{D_n}))e_3$ is cohomologous to a word in $x,y,z$ which
        is of the same length (note that the lengths of $x$, $y$ and $z$ are 2).
Namely, whenever a word $w$ has terms which goes along the long branch of the $D_n$ tree, it has to
return back at some point, hence it will include a subword of the form $c_{i(i+1)} c_{(i+1)i}$ which can be replaced with $c_{i(i-1)} c_{(i-1)i}$ applying the relation $Dc_i$. 
This can be repeated until we replace each subword that lies in the long branch by a power of $z$.

Arguing similarly, one can see why it suffices to consider $e_3 H^0(LCA(\Lambda_{D_n})) e_3$ to prove the statement in the lemma for the zeroth cohomology. 
Indeed, the relations given by $Dc_4, Dc_5,\ldots, Dc_n$ can be used to show that any sufficiently long word in
$LCA^0(\Lambda_{D_n})$ can be replaced by a word which contains a sufficiently long subword in $e_3
LCA^0(\Lambda_{D_n}) e_3$. More precisely, for any word $w \in \langle c_{ij} | i,j=1,n \rangle$ we can write 
\[ w = \alpha v \beta + \langle \mathrm{Im} D \rangle \]
		such that $v$ lies in $e_3 LCA^0(\Lambda_{D_n}) e_3$ and is sufficiently long. In fact, since we only use the preprojective relations, $Dc_i$ for $i\neq 3$, one can show that the analogue of \cite[Prop. 11.3.2 (ii)]{schedler2} holds in this case. 

We can simplify the presentation of $e_3 H^0(LCA(\Lambda_{D_n})) e_3$ further by eliminating the $z$ variable and write
\[ e_3 H^0(LCA(\Lambda_{D_n})) e_3 \cong \mathbb{K} \langle x,y \rangle / (x^2,y^2,(x+y+xy)^{n-2}). \]
Let us define two-sided ideals 
		\[ I_n=(x^2,y^2,(x+y+xy)^{n-2}) \text{\ \ \  and\ \ \ } J_n= (x^2,y^2,(x+y)^{n-2})\] in $\mathbb{K} \langle x,y \rangle$ and claim that they are equal for $n \geq 4$.
Note that in $\mathbb{K} \langle x,y \rangle / J_n$ any word that is long enough is trivial; in particular, this is a finite-dimensional vector space. This is because the only words that are not killed by the relations $x^2=y^2=0$ are words alternating in $x$ and $y$, and sufficiently long such words are killed by $x(x+y)^{n-2}y$ and $y(x+y)^{n-2}x$.
Therefore the result for $H^0(LCA(\Lambda_{D_n}))$ follows from the claim $I_n=J_n$.

To prove this claim, first observe that $A= x+y$ and $B=x+y+xy$ satisfy
$$  B^2 = (1+x)A^2(1+y)  \in \mathbb{K} \langle x,y \rangle /(x^2,y^2). $$
Moreover, since $(1+x)(1-x)=1=(1+y)(1-y)$ the above identity leads to $A^2 = (1-x)B^2(1-y)$ and together they show $I_4=J_4$. 
We similarly obtain $I_5=J_5$, using the observation
$$ B^3 = (1+x)A^3(1+x)(1+y) \in \mathbb{K} \langle x,y \rangle /(x^2,y^2).$$
The fact that $A^2$ is in the center of $\mathbb{K} \langle x,y \rangle /(x^2,y^2)$ implies
$$B^{2k}=(B^2)^k=(1+x)A^{2k}(1+y)(1+x)\cdots (1+y), $$
$$B^{2k+1}=B^3(B^2)^{k-1}=(1+x)A^{2k+1}(1+y)(1+x)\cdots (1+y) $$
proving $I_n=J_n$ for every $n \geq 4$.

Alternatively, one can check that a noncommutative Gr\"obner basis (with respect to the
lexicographical order) for both $I_n$ and $J_n$ is given
by the collection of the following three elements: 
\[ \{ x^2, y^2, xyxy\ldots + yxyx\ldots \} \]
where the lengths of the words in the last element is $n-2$.

This completes the proof of the lemma for $H^0(LCA(\Lambda_{D_n}))$. It is much harder to
directly compute $H^i(LCA(\Lambda_{D_n}))$ for $i<0$ and verify Hausdorffness of the length
filtration. Fortunately,
there is an alternative way to go about this making use of a recent result of Dimitroglou Rizell
\cite{rizell} which in turn exploits the weak division algorithm in free noncommutative algebras due
to P. Cohn \cite{cohn}. This is a general result about Legendrian
cohomology DG-algebras which states that the natural algebra homomorphism \[ H^*(LCA(\Lambda_\Gamma))
\to LCA^*(\Lambda_\Gamma) / \langle \mathrm{Im} D \rangle \] induced by inclusion is injective, where $\langle
\mathrm{Im} D \rangle $ denotes the two-sided ideal in the tensor algebra $LCA^*(\Lambda_\Gamma)$
generated by the image of the differential. 
In view of this, it suffices to show that for each $k$ there exists a $p(k)$ such that if $w$ is a
word in $c_{ij}$ of length greater than $p(k)$ and contains exactly $k$ instances of $c_i$, then $w$ is in $\langle \mathrm{Im} D \rangle$. 

This is, however, quite straightforward given what we have already proven. Namely, in any such word,
since the number of degree $-1$ generators, $c_i$, is precisely $k$ as soon as the length is
sufficiently large, we can find a sufficiently long subword consisting of degree 0 generators
$c_{ij}$ only. Now, we proved above that any sufficiently long word in the degree 0 generators
$c_{ij}$ is in the image of $D$. Thus, the result follows. \end{proof} 

Note that the corresponding result also holds true for $\mathscr{G}_{D_n}$ but this is much simpler. The cohomology $H^*(\mathscr{G}_{D_n})$ is a graded filtered algebra, where the filtered
subalgebras $\mathcal{F}^p H^*(\mathscr{G}_{D_n})$ for $p \geq 0$ are induced by the length
filtration on $\mathscr{G}_{D_n}$. We claim that this filtration on
$H^*(\mathscr{G}_{D_n})$ is complete and Hausdorff. To see this, observe the image of the
differential of $\mathscr{G}_{D_n}$ consists of homogeneous
terms (with respect to length filtration), hence the filtration is Hausdorff. The filtration is
complete because $H^*(\mathscr{G}_{D_n})$ is finite-dimensional at each degree. To see this, when $\mathbb{K}$ is algebraically closed and of characteristic $0$, one
can use the result by Hermes (see Thm.~(\ref{hermes})) that $H^i(\mathscr{G}_{D_n}) \cong \Pi_{D_n}$ for every $i \geq 0$,
and the well-known fact that the
preprojective algebra of a Dynkin quiver is finite-dimensional. Alternatively, for any field, $H^0(\mathscr{G}_{D_n}) = \Pi_{D_n}$ by definition, hence we can appeal to the argument given in the last part of the above lemma
to conclude. (Note that the result of \cite{rizell} requires an action filtration on the chain
complex respected by the differential. This is automatic for $LCA^*(\Lambda_\Gamma)$ as the relevant filtration is given by the
geometric action functional. On the other hand, if the complex is supported in nonpositive (or
nonnegative) degrees, then one can easily construct an action filtration of the required type
inductively, hence the main result of \cite{rizell} is applicable to $\mathscr{G}_\Gamma$ as well for any $\Gamma$.) 

We are now ready to prove the main result of this section:

\begin{theorem} \label{chainmap} Let $\Gamma = A_n$ or $D_n$, and assume that
    $\mathrm{char}(\mathbb{K}) \neq 2$
    if $\Gamma=D_n$. Then there exists a quasi-isomorphism    \[ LCA^*(\Lambda_\Gamma) \simeq \mathscr{G}_\Gamma \]
Furthermore, if $\mathrm{char}(\mathbb{K})=2$ and $\Gamma=D_n$, then $LCA^*(\Lambda)$ and
    $\mathscr{G}_\Gamma$ are not quasi-isomorphic.
\end{theorem}

(We conjecture that $LCA^*(\Lambda) \simeq \mathscr{G}_\Gamma$ for $\Gamma= E_6, E_7$ if
$\mathrm{char}(\mathbb{K}) \neq 2,3$ and for $\Gamma = E_8$ if
$\mathrm{char}(\mathbb{K}) \neq 2,3, 5$.)

\begin{proof} The case when $\Gamma = A_n$ is immediate since $LCA^*(\Lambda_\Gamma)$ and
    $\mathscr{G}_\Gamma$ are identical in this case. So, we will focus on the case $\Gamma=D_n$. 
    
When $\mathrm{char}(\mathbb{K}) \neq 2$, we will construct a chain map $\Phi : \mathscr{G}_\Gamma \to LCA^*(\Lambda_\Gamma)$ which is of the
    form:
    \[ \Phi = \mathrm{Id} + \mathrm{h.o.t.} \]
    where $\mathrm{h.o.t.}$ stands for higher order terms in terms of the length filtration
$\mathcal{F}^\bullet$ on $LCA^*(\Lambda_{\Gamma})$. 

In Sec.(\ref{typeD}), we computed 
\[ HH^*(\mathscr{G}_\Gamma, \mathscr{G}_\Gamma) \cong HH^*(A_\Gamma, A_\Gamma) \]
where $A_\Gamma$ is the Koszul dual to $\mathscr{G}_\Gamma$ as proven in
Thm.~(\ref{koszulness}). Note that the isomorphism between the Hochschild cohomologies of $\mathscr{G}_\Gamma$ and $A_\Gamma$ is a consequence of the Koszul duality given by Thm.~(\ref{koszulness}) which also states that the Koszul duality functor sends the internal grading of $\mathscr{G}_\Gamma$ to those of $A_\Gamma$ implying that the internal gradings on their Hochschild cohomologies match as well. 
In particular, we have:
\[ HH^2(\mathscr{G}_\Gamma, \mathscr{G}_\Gamma[s]) \cong HH^{2-s}(A_\Gamma, A_\Gamma[s]) \]

Let us warn the reader of a potentially confusing point in our notation. On the right hand side,
$r=2-s$ refers to the length grading in Hochschild cohomology, and $s$ refers to the internal grading induced from the
internal grading of the algebra $A_\Gamma$. This group is a summand of $HH^2(A_\Gamma,
A_\Gamma)$ where $2=r+s$ is the total degree. On the other hand,
$HH^2(\mathscr{G}_\Gamma, \mathscr{G}_\Gamma[s])$ is a summand of
$HH^2(\mathscr{G}_\Gamma, \mathscr{G}_\Gamma)$ where $s$ refers to the second grading on
$\mathscr{G}_\Gamma$ (as was explained after Equation (\ref{decomp})). 

The computation given in Sec.~(\ref{typeD}) implies that for $\Gamma =D_n$ and when $\mathrm{char}(\mathbb{K}) \neq 2$,we have:
\[ HH^2(\mathscr{G}_\Gamma, \mathscr{G}_\Gamma[s]) =0 \ \ \ \text{for} \ \ \ s < 0 \]
Therefore, from Lem.~(\ref{deftheory}), we deduce that there exists a
quasi-isomorphism:
\[ \Phi: \widehat{\mathscr{G}_\Gamma} \to \widehat{LCA}^*(\Lambda_\Gamma) \]

Now, let $N$ be an integer large enough that
$\mathcal{F}^N H^0(LCA(\Lambda_\Gamma))= 0 $; such an $N$ exists as we proved above in Lem.~(\ref{Hausdorff}). 
We then consider the truncation of $\Phi$ at length $N$ to define an algebra map between
\emph{uncompleted} algebras:
\[ \Phi^N : \mathscr{G}_\Gamma \to LCA^*(\Lambda_{\Gamma}) \]
The apparent problem with $\Phi^N$ is that it is not a chain map, though it fails to be a chain
map only at large length. So, we can correct it as
follows. For each vertex $v$, let us find a chain $\alpha_v$ such that
\[ D \Phi^N (h_v) - \Phi^N (dh_v) = D \alpha_v. \] 
Note that the left-hand-side is automatically $D$-closed since it lies in
    $LCA^0(\Lambda_\Gamma)$.

We then define a new algebra map by setting:
\[ \Psi(h_v) := \Phi^N(h_v) + \alpha_v, \ \ \ \Psi(g_{vw}) := \Phi^N(g_{vw}) \] 

We now have a filtered chain map $\mathscr{G}_\Gamma \to LCA^*(\Lambda_\Gamma)$ which
respects the length filtrations on each side. Note that the $E_2$-pages of the associated
spectral sequences are identical:
\[ E_2^{p,q} \cong \mathcal{F}^{p} \mathscr{G}_\Gamma / \mathcal{F}^{p+2}
    \mathscr{G}_\Gamma \]
with the differential induced from the differential on the Ginzburg DG-algebra. Furthermore,
the length filtration is not only complete and
Hausdorff on both sides by Lem.~(\ref{Hausdorff}) 
and the discussion following its proof, but also easily seen to be weakly convergent. Therefore the spectral sequences converge \emph{strongly} to
$H^*(\mathscr{G}_{D_n})$ and $H^*(LCA(\Lambda_{D_n}))$, respectively.
Moreover, since, 
\[ \Psi =  \mathrm{Id} + \mathrm{h.o.t.} \]
where $\mathrm{h.o.t.}$ refers to a higher order term that sends $\mathcal{F}^\bullet$ to $\mathcal{F}^{\bullet+2}$, it induces an isomorphism on the $E_2$-page therefore we conclude that it induces a
quasi-isomorphism of chain complexes by \cite[Thm. 2.6]{boardman}. This completes the proof that
        $LCA^*(\Lambda_{D_n})$ and $\mathscr{G}_{D_n}$ are quasi-isomorphic over a field of
        characteristic $\neq 2$.
        
Next suppose that $\mathbb{K}$ is a field of characteristic $2$. Let us write $D= d + d_3$ for the differential on $LCA^*(\Lambda_{D_n})$ where, in the notation of Lem.~(\ref{Hausdorff}),
we have
\[ d_3(c_3) = -c_{31} c_{13} c_{32} c_{23} \] 
We want to show that there is no degree 0 derivation $\phi_2$ which increases
length by 2 and solves $d_3 = d \phi_2 - \phi_2 d$. For $\Gamma=D_4$, this is equivalent to the following set of \emph{linear} equations:
\begin{align*}
    0 &= d \phi_{2} (c_1) - \phi_2 (c_{13}) c_{31} - c_{13} \phi_2 (c_{31}) \\
    0 &= d \phi_{2} (c_2) - \phi_2 (c_{23}) c_{32} - c_{23} \phi_2 (c_{32})\\
    -c_{31} c_{13} c_{32} c_{23} &= d \phi_{2} (c_3) + \phi_2 (c_{31} c_{13} + c_{32} c_{23} -
    c_{34}
    c_{43} ) \\
        0 &= d \phi_{2} (c_4) + \phi_2 (c_{43}) c_{34} + c_{43} \phi_2 (c_{34}) \\
\end{align*} 
(Although we are working over characteristic 2 here, we have kept the signs in their general form for reference.)

Since $\phi_2$ is supposed to preserve the degree and increase the length by 2, there are only a few
possibilities. The general form of the possibilities is as follows:
\begin{align*}
\phi_2(c_1) &\in \mathbb{K} c_1 c_{13} c_{31} \oplus \mathbb{K} c_{13} c_{31} c_{1} \oplus \mathbb{K} c_{13} c_3 c_{31} \\  
\phi_2(c_2) &\in \mathbb{K} c_2 c_{23} c_{32} \oplus \mathbb{K} c_{23} c_{32} c_{2} \oplus \mathbb{K} c_{23} c_3 c_{32} \\  
 \phi_2(c_3) &\in \mathbb{K} c_3 c_{31} c_{13} \oplus \mathbb{K} c_{31} c_{13} c_{3} \oplus
\mathbb{K} c_3 c_{32} c_{23} \oplus \mathbb{K} c_{32} c_{23} c_{3} \oplus \mathbb{K} c_3 c_{34}
    c_{43} \oplus \mathbb{K} c_{34} c_{43} c_{3} \\ &  \oplus
     \mathbb{K} c_{31} c_1 c_{13} \oplus \mathbb{K} c_{32} c_2 c_{23} \oplus \mathbb{K} c_{34} c_4 c_{43} \\  \phi_2(c_4) &\in \mathbb{K} c_4 c_{43} c_{34} \oplus \mathbb{K} c_{43} c_{34} c_{4} \oplus
    \mathbb{K} c_{43} c_3 c_{34} \\ 
\phi_2(c_{13}) &\in \mathbb{K} c_{13} c_{31} c_{13} \oplus \mathbb{K} c_{13} c_{32} c_{23} \oplus \mathbb{K} c_{13} c_{34} c_{43} \\
\phi_2(c_{31}) &\in \mathbb{K} c_{31} c_{13} c_{31} \oplus \mathbb{K} c_{32} c_{23} c_{31} \oplus \mathbb{K} c_{34} c_{43} c_{31} \\
\phi_2(c_{23}) &\in \mathbb{K} c_{23} c_{32} c_{23} \oplus \mathbb{K} c_{23} c_{31} c_{13} \oplus \mathbb{K} c_{23} c_{34} c_{43} \\
\phi_2(c_{32}) &\in \mathbb{K} c_{32} c_{23} c_{32} \oplus \mathbb{K} c_{31} c_{13} c_{32} \oplus \mathbb{K} c_{34} c_{43} c_{32} \\
\phi_2(c_{43}) &\in \mathbb{K} c_{43} c_{34} c_{43} \oplus \mathbb{K} c_{43} c_{31} c_{13} \oplus \mathbb{K} c_{43} c_{32} c_{23} \\
\phi_2(c_{34}) &\in \mathbb{K} c_{34} c_{43} c_{34} \oplus \mathbb{K} c_{31} c_{13} c_{34} \oplus \mathbb{K} c_{32} c_{23} c_{34}  
\end{align*}

This leads to a system of 18 linear equations of 36 variables. It is straightforward, if tedious, to
verify directly (or with the help of a computer) that none of the possibilities gives a solution
when $\mathbb{K} = \mathbb{Z}_2$. This, in turn, implies that the class of $[\tilde{d}_3]$ is
non-trivial over any field $\mathbb{K}$ of characteristic 2 by the universal coefficient theorem. 

This implies that there is a non-vanishing obstruction for
constructing a chain map between $\mathscr{G}_{D_4}$ and $LCA^*(\Lambda)$ over a field of
characteristic 2 for $D_4$. In other words, the class $[\tilde{d}_3] \in
HH^2(\mathscr{G}_{D_4} ,
\mathscr{G}_{D_4} [-2])$ is non-trivial. (Compare this with our computation of
$HH^2(\mathscr{G}_{D_4}, \mathscr{G}_{D_4} [-2])$ given later on in Fig. (\ref{d4table}); where
this group is shown to be non-trivial only in characteristic 2.) Now, the class of $[\tilde{d}_3]$ for $\Gamma = D_n$ restricts to the class of $\Gamma=D_4$ under the restriction map. (Note that in general Hochschild cohomology does
not have good functoriality properties however there is a full and faithful
inclusion of the $\mathscr{G}_{D_4}$ to $\mathscr{G}_{D_n}$, and there is a restriction map on
Hochschild cohomology in this case.) Hence, it cannot vanish for $\Gamma= D_n$ either. 
\end{proof}

\begin{remark}
Over a field of characteristic $\neq 2$, and for $\Gamma=D_4$, we constructed an explicit chain map between
    $\mathscr{G}_{D_4}$ and $LCA^*(\Lambda_{D_4})$ as a check on our arguments above. The complication in this also displays the effectiveness of the deformation
    theory argument given above. (Notice the factors of $1/2$ which are indeed necessary.) The map is given as follows:
\begin{align*} 
    h_1 &\mapsto c_1 - (1/2) ( c_{13} c_{31} c_1 + c_{13} c_3 c_{31} + c_1 c_{13} c_{32} c_{23} c_{31}) \\
    h_2 &\mapsto c_2 - (1/2) (c_{23} c_{32} c_2 + c_{23} c_3 c_{32} + c_{23} c_{31} c_{13}  c_{32} c_2) \\ 
    &\ \ \ \ \ \ \ \ \    + (1/4) ( c_{23} c_{34} c_{43} c_3 c_{32} + c_{23} c_{34} c_4 c_{43} c_{32} + c_{23} c_{34} c_{43} c_{32}c_2 + c_{23} c_{34} c_{43} c_{31} c_{13}c_{32}c_2 ) \\
    h_3 &\mapsto c_3 - (1/4)( c_{31} c_{13} c_3 c_{34} c_{43} + c_{31} c_1 c_{13} c_{34} c_{43} +
    c_{31} c_{13} c_{34} c_4 c_{43} + c_{31} c_1 c_{13} c_{32} c_{23} c_{34} c_{43}) \\ 
    h_4 &\mapsto c_4 - (1/2) ( c_4 c_{43} c_{34}  + c_{43} c_3 c_{34} - c_{43} c_3 c_{32} c_{23} c_{34} -
    c_{43} c_{32} c_2 c_{23} c_{34} - c_4 c_{43} c_{32} c_{23} c_{34}\\  & \ \ \ \ \ \ \ \ \ \ \ \ \ \ \ \ \ \ \ \ \ \ \ \  -
    c_{43} c_{31} c_{13} c_{32}    c_2 c_{23} c_{34}) \\
g_{13} &\mapsto c_{13} + (1/2) (c_{13} c_{32} c_{23} - c_{13} c_{34} c_{43}) \\
    g_{31} &\mapsto c_{31} \\
    g_{23} &\mapsto c_{23} - (1/2) c_{23} c_{34} c_{43} \\
    g_{32} &\mapsto c_{32} + (1/2) c_{31} c_{13} c_{32} \\
    g_{34} &\mapsto c_{34} - (1/2)(c_{32} c_{23} c_{34}+c_{31} c_{13} c_{34} )   \\
    g_{43} &\mapsto c_{43} 
\end{align*}
\end{remark}

\begin{remark} \label{subtree} One can deduce from the argument given in the last part of the proof of Thm.
    (\ref{chainmap}) that for any tree $\Gamma$ which is not of type $A_n$, we have that
    $\mathscr{B}_\Gamma := LCA^*(\Lambda_\Gamma)$ is a non-trivial deformation of
    $\mathscr{G}_\Gamma$ over a field of characteristic $2$ since any such tree has a subtree of the
    form $D_4$ (see also Rmk. (\ref{BVremark})).
\end{remark}

\vspace{-.1in}
\section{Floer cohomology algebra of the spheres in $X_\Gamma$} 
\label{spheres}
\vspace{-.1in}

We next consider the $A_\infty$-algebra over $\mathrm{k}$ given by the Floer cochain complexes: 
\[ \mathscr{A}_\Gamma := \bigoplus_{v,w} CF^*(S_v,S_w) \]

Recall that the Lagrangian 2-spheres $S_v$ and $S_w$ intersect only if the vertices $v$ and $w$ are connected by an edge, in which case $S_v \cap S_w$ is a unique point. 
Recall also that we made choices of grading structures on the sphere $S_v$ in Sec.~(\ref{plumbing}) so that $CF^*(S_v,S_w)$ is concentrated in degree 1, if $v,w$ are adjacent vertices. 
On the other hand, the self-Floer cochain complex $CF^*(S_v,S_v)$ is quasi-isomorphic to the singular chain complex $C^*(S_v)$ since $S_v$ is an exact Lagrangian sphere in $X_\Gamma$. 
Therefore, we can take a model for $\mathscr{A}_\Gamma$ such that the differential on $\mathscr{A}_\Gamma$ necessarily vanishes for degree reasons.

Let us put $A_\Gamma = H^*(\mathscr{A}_\Gamma)$ for the corresponding associative algebra. We can
think of $\mathscr{A}_\Gamma$ as a minimal $A_\infty$-structure $(\mu^n)_{n \geq2}$ on the
associative algebra $A_\Gamma$. As before, by choosing a root, we make $\Gamma$ into a directed
graph such that oriented edges point away from the root. Let $\mathrm{D}\Gamma$ denote the double of the quiver $\Gamma$, formed by introducing a new oriented edge $a_{vw}$ from $w$ to $v$ for every oriented edge $a_{wv}$ from $v$ to $w$.

\begin{proposition} \label{compacts} Suppose $\Gamma \neq A_1$. The graded $\mathrm{k}$-algebra
    $A_\Gamma$ is isomorphic to the \emph{zigzag algebra} of $\Gamma$ given by the path algebra
    $\mathbb{K}\mathrm{D}\Gamma$ equipped with the path-length grading modulo the homogeneous ideal generated by the following elements:
	\begin{itemize}
	\item $a_{u v} a_{v w}$ such that $u \neq w$, where $v$ is adjacent to both $u,w$.
	\item $a_{v w} a_{w v} - a_{vu} a_{uv}$ where $v$ is adjacent to both $u,w$.
	\end{itemize}
If $\Gamma = A_1$, then $A_\Gamma \cong H^*(S^2) = \mathbb{K}[x]/(x^2)$ with $|x|=2$. 
\end{proposition}
\begin{proof} Note that $S_v$ intersects $S_w$ for $w\neq v$ if and only if $v$ and $w$ are adjacent vertices in which case the intersection is transverse at a unique point. Furthermore, we have chosen the grading structures on the Lagrangians $S_v$ so as to ensure that for $v,w$ adjacent $CF^*(S_v,S_w)$ is of rank 1 and concentrated in degree 1. We let $a_{vw}$ be a generator for this 1-dimensional vector space. Finally, the algebra structure is determined by the general Poincar\'e duality property of Floer cohomology (see \cite[Sec. 12e]{thebook}). 
\end{proof}

The algebra $A_\Gamma$ only depends on the underlying tree
$\Gamma$; different ways of orienting its edges results in the same algebra.
We call the algebra $A_\Gamma$ the $\emph{zigzag algebra}$ of $\Gamma$
following Khovanov and Huerfano \cite{huekho} who studied
properties of this algebra and its appearances in a variety of areas related to
representation theory and categorification. On the other hand, the case where
$\Gamma$ is the $A_n$ quiver appeared in an earlier paper of Seidel and Thomas
\cite{seidelthomas} in the context of Floer cohomology (as it does here) and mirror
symmetry. In the context of Koszul duality (see \cite{priddy}, \cite{bgs}), the algebras $A_\Gamma$ were
studied much earlier by Martinez-Villa in \cite{MV}. This remarkable work is
the first paper, as far as we know, which draws attention to the fact that
$A_\Gamma$ is a Koszul algebra if and only if $\Gamma$ is not Dynkin or
$\Gamma=A_1$.  

We will next discuss formality of  $\mathscr{A}_\Gamma$ , i.e., the question of whether there is a
quasi-isomorphism between $\mathscr{A}_\Gamma$ and $A_\Gamma =
H^*(\mathscr{A}_\Gamma)$. In the case when $\Gamma$ is the $A_n$ quiver, the
formality was proven by Seidel and Thomas in \cite[Lem. 4.21]{seidelthomas}
based on the notion of \emph{intrinsic formality}.

\begin{definition} A graded algebra $A$ is called intrinsically formal if any $A_\infty$-algebra $\mathscr{A}$ with $H^*(\mathscr{A}) \cong A$ is quasi-isomorphic to $A$. 
\end{definition} 
 
Furthermore, Seidel and Thomas give a useful method to recognize intrinsically formal algebras. Recall that for a graded algebra $A$, $HH^*(A)$ has two gradings: the cohomological grading $r$ and the grading $s$ coming from the grading of the algebra $A$. To specify the decomposition into graded pieces, we write: 
\[ HH^*(A) = \bigoplus_{*= r+s} HH^{r}(A, A[s]) \]

Notice that the superscript denotes the diagonal grading, as usual. It is also
the grading that survives, if $A$ is more generally a DG-algebra or an
$A_\infty$-algebra.

\begin{theorem}\label{intrinsic} (Kadeishvili \cite{kadeishvili}, see also Seidel-Thomas \cite{seidelthomas}) Let $A$ be an augmented graded algebra.  If 
	\[ HH^{2-s}(A,A[s]) = 0 \text{ \  \ for all $s<0$ } \] then, $A$ is intrinsically formal.
\end{theorem} 

As mentioned above, Seidel-Thomas proved intrinsic formality of $A_\Gamma$ where $\Gamma$ is the $A_n$ quiver by showing the vanishing of $HH^{2-s}(A_\Gamma, A_\Gamma[s])$ for $s<0$. 
In a similar vein, we prove in Thm.~(\ref{hhd}) that $A_\Gamma$ is intrinsically formal if $\Gamma$ is the $D_n$ quiver and the characteristic of the ground field is not $2$. 

We have the following conjecture for the remaining Dynkin types. 

\begin{conjecture} \label{formality} Working over a ground field $\f{K}$ of characteristic 0, 
let $\Gamma$ be a tree of type $E_6,E_7$ or $E_8$. Then, the corresponding zigzag algebra $A_\Gamma$ is intrinsically formal. 
\end{conjecture}

Unlike the $A_n$ case, some restriction on characteristic of $\f{K}$ is
necessary as we have checked that the zigzag algebras are not intrinsically
formal in type $D_n$, $n \geq4$, over characteristic 2, in type $E_6$ and $E_7$
over characteristic 2 or 3, and in the type $E_8$, over characteristic 2, 3 or
5.  It is very likely that these are the only ``bad'' characteristics (cf.
\cite{schedler}).

\section{Koszul duality} 
\label{seckoszul}
\vspace{-.1in}

By combining the work of Bourgeois, Ekholm and Eliashberg \cite{BEE} with
Abouzaid's generation criteria \cite{abouzgen}, one might suspect that the
Lagrangians $L_v$ split-generate the wrapped Fukaya category
$\mathcal{W}(X_\Gamma)$. Now, there exists a full and faithful embedding  \[
\mathcal{F}(X_\Gamma) \to \mathcal{W}(X_\Gamma)\] of the exact Fukaya category
of compact Lagrangians. Therefore, in view of Rem.~(\ref{sftshit}), we would conclude that there is a quasi-isomorphism of DG-algebras: \begin{align} \label{rhom}
\operatorname{RHom}_{\mathscr{B}_\Gamma} (\mathrm{k},\mathrm{k}) \simeq
\mathscr{A}_\Gamma \end{align} The right-hand-side is in turn quasi-isomorphic to
$A_\Gamma$ if one checks that $\mathscr{A}_\Gamma$ is formal (for example this is known if $\Gamma$ is
of type $A_n$ \cite{seidelthomas} and we prove it in Thm.~(\ref{hhd}) for type $D_n$ over a
field of characteristic $\neq 2$). We will provide
an alternative independent approach via a purely algebraic argument based on
Koszul duality theory for DG- or $A_\infty$-algebras (see \cite{LPWZ}) to stay
within the algebraic framework of this paper (and avoid the technicalities that
go into the discussion in Rem.~(\ref{sftshit})). 

In fact, as we shall see below, Koszul duality theory allows us to work directly with $A_\Gamma = H^*(\mathscr{A}_\Gamma)$, hence in this way we bypass formality questions for $\mathscr{A}_\Gamma$. 

We now give a brief review of Koszul duality, first in the case of associative algebras and then for $A_\infty$-algebras.
\vspace{-0.2in}
\subsection{Quadratic duality and Koszul algebras} \label{quadraticduality} To begin with, we review quadratic duality for
associative algebras following \cite[Sec. 2.1]{seidelflux} which has an
explicit discussion of signs in the context relevant here. Original reference is \cite{priddy}, and see also the excellent exposition in \cite{bgs}.  

Let $\mathrm{k} = \bigoplus_v \f{K}e_v$ be the commutative semi-simple ring of orthogonal primitive idempotents over the
base-field $\f{K}$, as before. Let $V$ be a finite-dimensional graded
$\f{K}$-vector space with a $\mathrm{k}$-bimodule structure. We write 
\[ T_\mathrm{k} V := \bigoplus_{i=0}^{\infty} V^{\otimes_{\mathrm{k}} i} \]
for the tensor algebra over $\mathrm{k}$. A quadratic graded algebra $A$ is an associative unital graded $\mathrm{k}$-algebra that is a quotient
\[ A :=  T_\mathrm{k} V/  J  \]
of $T_\mathrm{k} V$ by the two-sided ideal generated by a graded $\mathrm{k}$-submodule $J \subset V \otimes_\mathrm{k} V$. 
In fact, this makes $A$ into a bigraded algebra: it has an internal grading coming from the graded vector space $V$, denoted by $s$ or $|x|$ if for a specific element , and a length grading coming from the tensor algebra, denoted by $r$. The reference \cite{LPWZ} refers to $s$ as Adams grading. 

Let $V^{\vee} = Hom_{\f{K}} (V, \f{K})$ be the linear dual of $V$ viewed naturally as a $\mathrm{k}$-bimodule, i.e. $e_i V^{\vee} e_j$ is the dual of $e_jVe_i$.
Next, we consider the orthogonal dual $J^{\perp} \subset V^{\vee} \otimes_{\mathrm{k}} V^{\vee}$ with respect to the canonical pairing given by:
\begin{align*} 
	V^{\vee} \otimes_{\mathrm{k}} V^{\vee} \otimes_{\mathrm{k}} V \otimes_{\mathrm{k}} V &\to \mathrm{k}  \\
	v_2^{\vee} \otimes_{\mathrm{k}} v_1^{\vee} \otimes_{\mathrm{k}} v_1 \otimes_{\mathrm{k}} v_2 &\mapsto (-1)^{|v_2|} v_2^{\vee} (v_2) v_1^{\vee}(v_1)  \end{align*}
The quadratic dual to $A$ is defined as:
\[ A^{!} = T_\mathrm{k} \left(V^{\vee}[-1]\right) / J^{\perp}[-2] \]
As does $A$, the graded quadratic algebra $A^{!}$ has two natural gradings: one internal grading coming from the internal grading of the vector space $V^{\vee}[-1]$, denoted by $s$ or $|x^!|$ for a specific element, and the length grading coming from the tensor algebra, denoted by $r$.

The Koszul complex of a quadratic algebra is the graded right $A$-module $A^{!}
\otimes_{\mathrm{k}} A$ with the differential \footnote{\cite{bgs}
prefers to use the graded left module $A
\otimes_\mathrm{k}^{\phantom{1}\vee\!}\!(A^{!})$, the two graded modules are
related by the right module isomorphism $A^{!} \otimes A \simeq Hom_A(A
\otimes_\mathrm{k}^{\phantom{1}\vee\!}\!(A^{!}), A)$ and the sign $(-1)^{|x|}$ coming from
this dualization.}: 
\begin{equation}\label{koszulcomplex}  x^{!} \otimes_{\mathrm{k}} x \to \sum_{i} (-1)^{|x|} x^{!} a_i^{\vee} \otimes_{\mathrm{k}} a_i x \end{equation}
where the sum is over a basis of $\{a_i \}$ of $V$, and $\{a_i^\vee \}$ is the dual basis in $V^\vee [-1]$. This should be thought of as an $(r,s)$-bigraded complex, where the grading $r$ is the path-length grading in the $A^{!}$ factor and the total grading $r+s$ corresponds to the natural grading $|x^!|+|x|$. In particular, one has $|a_i^\vee| + |a_i| =1$ for all $i$, hence the $s$ grading is preserved by the differential. 

A Koszul algebra $A$ is a quadratic algebra for which the Koszul complex is
acyclic (i.e. homology is isomorphic to $\mathrm{k}[0]$). Taking the dual by applying the left exact functor $Hom_{A}( \cdot ,  A)$, we
get a resolution of $\mathrm{k}$ as a graded right $A^{op}$-module (see \cite[Sec. 2]{bgs} for more details).

In fact, if $A$ is Koszul, considering $\mathrm{k}$ as a simple module in the abelian category of graded \emph{right} $A^{op}$-modules, one has a canonical isomorphism of bigraded rings:
\[ A^{!} \cong \mathrm{Ext}^*_{A^{op}} (\mathrm{k}, \mathrm{k}) \]

Since $A$ is bigraded, a priori $\mathrm{Ext}^*_{A^{op}} (\mathrm{k},\mathrm{k})$ is triply graded (by the cohomological degree and by the length and internal gradings, derived from the corresponding ones in $A$). One characterization of Koszulity is that the cohomological degree, which we denote by $r$, agrees with the grading induced by length. Finally, we denote the internal grading by $s$. With this understood, we have the graded identifications:
\[ A^{!}_{r,s}  \cong \mathrm{Ext}^{r}_{A^{op}} (\mathrm{k}, \mathrm{k}[s]) \] 
If $A$ is Koszul, then its Koszul dual $A^{!}$ is also Koszul and $(A^{!})^{!} \cong A$.

Finally, for a Koszul algebra $A$, the Hochschild cohomology can be computed via the Koszul bimodule resolution of $A$. The resulting complex which computes Hochschild cohomology is 
\begin{equation} \label{koszulbi} (A^{!} \otimes_{\mathrm{k}} A)_{diag} = \bigoplus_{v} e_v A^{!} \otimes_{\mathrm{k}} A e_v \end{equation} with the differential: 
\[  x^{!} \otimes_{\mathrm{k}} x \to \sum_{i} (-1)^{|x|} x^{!} a_i^{\vee} \otimes_{\mathrm{k}} a_i x  - (-1)^{(|a_i|+1)(|x|+|x^!|)}a_i^{\vee} x^! \otimes_k x a_i \]

It is often the case, as in this paper, that $V$ is generated either by odd elements or even elements; this simplifies the signs in the above formula. For Koszul algebras, the homology of this complex coincides with the bigraded Hochschild cohomology groups $HH^r(A,A[s])$ where $r+s$ corresponds to the natural grading on $(A^! \otimes A)_{diag}$, that is, an element $x^! \otimes_{\mathrm{k}} x$ has grading $|x^!|+|x|$. The length grading $r$ corresponds to the path-length grading in the $A^{!}$ factor.

\begin{example} \label{singlevertex} Let $A_\Gamma = \f{K}[x]/(x^2)$ with $|x|=2$ be the zigzag
	algebra associated with a single vertex, i.e. $\Gamma$ is of type
	$A_1$. It is easy to see that this is a Koszul algebra and we have
	$A^{!}_\Gamma = \f{K}[x^\vee]$, the free algebra with $|x^\vee|=-1$. One can
	compute Hochschild cohomology using the Koszul bimodule complex. This
	has generators $(x^\vee)^i \otimes 1$ and $(x^\vee)^i \otimes x$ for $i \geq 0$. The
	differential can be computed as: 
\begin{align*}
      d( (x^\vee)^i \otimes 1 ) &= (1 + (-1)^{i+1}) (x^\vee)^{i+1} \otimes x \\ 
      d( (x^\vee)^i \otimes x ) &= 0 
\end{align*}
Therefore, whenever  $\mathrm{char}(\f{K}) = 2$ the differential vanishes and as a consequence $HH^* (A_\Gamma)$ has a basis $(x^\vee)^i \otimes 1$, for $i \geq 0$, in bigrading $(r,s)= (i, -2i)$ and $(x^\vee)^i \otimes x$, for $i \geq 0$, in bigrading $(r,s) = (i, 2-2i)$. 

If $\mathrm{char}(\f{K}) \neq 2$, then $HH^*(A_\Gamma)$ has a basis $(x^\vee)^{2i} \otimes 1$, for $i \geq 0$, in bigrading $(r,s) = (2i, -4i)$ and $(x^\vee)^{2i+1} \otimes x$, for $i \geq 0$, in bigrading $(r,s) = (2i+1, -4i)$ and $1 \otimes x$ in bigrading $(0,2)$. 

In view of the discussion given in the introduction, this result computes $SH^*(T^* S^2)$ for $* = r+s$. For convenient access, we record a finite portion of this calculation in Table~(\ref{loophomology}). 

By Viterbo's isomorphism (\cite{viterbo},\cite{abouzmanuz}), this computation also gives $H_{2-*}(\mathcal{L}S^2)$, where $\mathcal{L}S^2$ is the free loop space of $S^2$.
This was previously computed as a ring by Cohen, Jones and Yan \cite{CJY} over $\mathbb{Z}$ to be: 
\[ H_{2-*}(\mathcal{L}S^2 ; \mathbb{Z}) \cong (\Lambda b \otimes
\mathbb{Z}[a,v])/(a^2,ab,2av), \ \ |a|=2, |b|=1, |v|=-2   \]
using the fibration $\Omega_x S^2 \to \mathcal{L}S^2 \to S^2$. From
this, one can deduce that:
\[ H_{2-*}(\mathcal{L}S^2 ; \mathbb{K}) \cong \Lambda a \otimes \mathbb{K}[u], \ \ |a|=2, |u|=-1\]
if $\mathrm{char}\ \mathbb{K}=2$ and
\[ H_{2-*}(\mathcal{L}S^2 ; \mathbb{K}) \cong (\Lambda b \otimes \mathbb{K}[a,v])/(a^2,ab,av), \ \
    |a|=2, |b|=1, |v|=-2\]
if $\mathrm{char}\ \mathbb{K} \neq 2$,
in agreement with our computation. 
\end{example}

\begin{table}[htbp!]
\centering
\begin{tikzpicture}
\matrix (mymatrix) [matrix of nodes, nodes in empty cells, text height=1.5ex, text width=2.5ex, align=right]
{
\begin{scope} \tikz\node[overlay] at (-2.2ex,-0.6ex){\footnotesize r+s};\tikz\node[overlay] at (-1ex,0.8ex){\footnotesize s}; \end{scope} 
   & 2 & 1 & 0 & $-1$ & $-2$ & $-3$ & $-4$ & $-5$ & $-6$ & $-7$ & $-8$ \\ 
 2 & 1 & 0 & 0 &  0 &  0 &  0 &  0 &  0 &  0 &  0 &  0 \\ 
 1 & 0 & 0 & 1 &  0 &  0 &  0 &  0 &  0 &  0 &  0 &  0 \\
 0 & 0 & 0 & 1 &  0 &  x &  0 &  0 &  0 &  0 &  0 &  0 \\
-1 & 0 & 0 & 0 &  0 &  x &  0 &  1 &  0 &  0 &  0 &  0 \\
-2 & 0 & 0 & 0 &  0 &  0 &  0 &  1 &  0 &  x &  0 &  0 \\
};
\draw (mymatrix-1-1.south west) ++ (-0.2cm,0) -- (mymatrix-1-12.south east);
\draw (mymatrix-1-2.north west) -- (mymatrix-6-2.south west);
\draw (mymatrix-1-1.north west) -- (mymatrix-1-1.south east);%n k diagonal line
\end{tikzpicture}
\caption{$\Gamma= A_1$.  $x$ is 1 if $\text{char} \f{K} = 2$, $0$ otherwise.} 
\label{loophomology}
\end{table}

\subsection{Koszul duality for $A_\infty$-algebras}

We now review Koszul duality for $A_\infty$-algebras. Our primary reference
for this material is \cite{LPWZ}. The discussion in $\cite{LPWZ}$ is about
$A_\infty$-algebras over a field $\f{K}$, but as in classical Koszul duality,
the proofs extend readily to $A_\infty$-algebras over a semisimple ring
$\mathrm{k}$ (see also \cite{segal}). The extension of Koszul duality theory to
DG- or $A_\infty$-algebras has appeared earlier (see eg. \cite{kellerderiving}). 

Suppose $A = \bigoplus_{i\geq0} A_i$ is a positively graded associative algebra
over $A_0=\mathrm{k}$. Then, as before, the complex \[ \mathrm{RHom}_{A^{op}}
(\mathrm{k},\mathrm{k}) \] inherits a bigrading by cohomological and length
gradings. However, it usually happens that at the level of homology these two
gradings do not agree, that is, $A$ is not Koszul as an associative algebra,
and passing to the homology of this complex yields an associative algebra
$\mathrm{Ext}^*_{A^{op}} (\mathrm{k},\mathrm{k})$ from which one cannot recover
$A$.  In this case, the idea is that rather than passing to homology, one
thinks of the DG-algebra $\mathrm{RHom}_{A^{op}} (\mathrm{k},\mathrm{k})$ as
the DG-Koszul dual of $A$. To be able to carry this out, one is led to work
with DG- or $A_\infty$-algebras from the beginning. So, let $\mathscr{A}$ be
a $\f{Z}$-graded $A_\infty$-algebra over $\mathrm{k}$ together with an
augmentation $\epsilon : \mathscr{A} \to \mathrm{k}$, making $\mathrm{k}$ into
a right $A_\infty$-module over $\mathscr{A}^{op}$. One defines \[
\mathscr{A}^{!} = \mathrm{RHom}_{\mathscr{A}^{op}} (\mathrm{k},\mathrm{k}). \]
Note that the Yoneda image of $\mathrm{k}$ given by
$\mathrm{RHom}_{\mathscr{A}^{op}}(\mathscr{A}^{op},\mathrm{k})$ makes $\mathrm{k}$ into a
right $(\mathscr{A}^{!})^{op}$-module. Now, the obvious concern is whether
$(\mathscr{A}^{!})^{!}$ gets back to $\mathscr{A}$ (up to quasi-isomorphism).
This is not quite the case in general; one recovers a certain completion of
$\mathscr{A}$ (see \cite{segal} for a beautiful geometric description of this
construction). However, suppose that $\mathscr{A}$ has an additional $s$
grading (called Adams grading in \cite{LPWZ}) which is required to be preserved
by the $A_\infty$ operations. Furthermore, assume that $\mathscr{A}$ is
connected and locally finite with respect to this grading; this means that $\mathscr{A}$ is either
non-negatively or non-positively graded and the $s$-homogeneous subspace of
$\mathscr{A}$ is of finite dimension for each $s$ (see \cite[Def. 2.1]{LPWZ}).
Then, it is true that $(\mathscr{A}^{!})^{!}$ is quasi-isomorphic to
$\mathscr{A}$. We state this as: \begin{theorem} (Lu-Palmieri-Wu-Zhang
	\cite[Thm. 2.4 \footnote{The proof of \cite[Thm. 2.4]{LPWZ} uses \cite[Lem. 1.15]{LPWZ} which omits a necessary hypothesis. Namely, in the notation of \cite[Lem. 1.15]{LPWZ}, one should further assume $B_{aug}^\infty R$ is locally finite. By \cite[Lem. 2.2]{LPWZ}, this requirement holds under our hypothesis.}]{LPWZ} ) \label{dgkoszul} Suppose $\mathscr{A}$ is an augmented
	$A_\infty$-algebra over the semisimple ring $\mathrm{k}$ with a
	bigrading for which $\mu^k$ has degree $(2-k,0)$ and suppose
	$\mathscr{A}$ is connected and locally finite with respect to the second grading. Let \[ \mathscr{A}^! =
	\mathrm{RHom}_{\mathscr{A}^{op}} (\mathrm{k},\mathrm{k}) \] be its
	Koszul dual as an $A_\infty$-algebra. Then, there is a quasi-isomorphism
	of $A_\infty$-algebras: \[ \mathscr{A} \simeq
	\mathrm{RHom}_{(\mathscr{A}^{!})^{op}}(\mathrm{k},\mathrm{k}) \]
\end{theorem}

Below, we will apply this result for $\mathscr{A}= A_\Gamma$ viewed as a formal $A_\infty$-algebra.

\begin{example} To see the importance of the connectedness and finiteness assumptions, let us consider $A=\f{K}[x,x^{-1}]$ with $x$ in bigrading
	$(0,0)$, the (trivially graded) algebra of Laurent polynomials.
	Consider the augmentation $\epsilon: A^{op} \to \f{K}$ given by mapping $x$ to $1 \in \f{K}$, which makes $\f{K}$ into a right $A$-module. Then, one can check
	that $A^{!} = \mathrm{RHom}_{A^{op}}(\f{K},\f{K})$ is quasi-isomorphic
	to the exterior algebra $\f{K}[x^{!}]/((x^{!})^2)$ with $x^{!}$ in
	bigrading $(0,1)$. However, $\mathrm{RHom}_{(A^!)^{op}}(\f{K},\f{K})
	\cong \f{K}[[y]]$ gives the power series ring with $y$ in bigrading
	$(0,0)$. Hence, dualizing twice does not get us back in this case.

\end{example}

\subsection{Koszul dual of $\mathscr{G}_\Gamma$}

We next prove that the DG-algebra $\mathscr{G}_\Gamma$ and $A_\Gamma$ (viewed as
a formal $A_\infty$-algebra) are related by Koszul duality. We remind the
reader that we always work with right modules (as we follow the sign
conventions from \cite{thebook}). 

We have the following analogue of  \cite[Prop. 2.9.5]{ginzburg}  in our setting:

\begin{theorem} \label{koszulness} Consider $\mathrm{k} = A^{op}_\Gamma/ (A^{op}_\Gamma)_{>0}$ as a
    right $A_\Gamma^{op}$-module. There is a DG-algebra isomorphism: \[
    \operatorname{RHom}_{A_\Gamma^{op}} (\mathrm{k},\mathrm{k}) \simeq \mathscr{G}_{\Gamma^{op}} \]
such that cohomological (resp. internal) grading on the left-hand-side agrees with the path-length (resp. internal) grading on the right-hand-side. 
\end{theorem} 
\begin{proof} First, let us clarify the multiplication on $A_\Gamma^{op}$, which we view as a formal $A_\infty$-algebra. We identify the elements of $A_\Gamma^{op}$ with the elements of $A_\Gamma$ which are given by the symbols $a_{vw}$, and $a_{vw}a_{wv}$ as before. 
	Since $|a_{wv}|=1$ for all $w$ adjacent to $v$, the product is given by:  
\[ \mu^2_{A_\Gamma^{op}} (a_{wv},a_{vw}) = (-1)^{|a_{wv}|+|a_{vw}|} \mu^2_{A_\Gamma}(a_{vw},a_{wv}) = (-1)^{|a_{vw}|} a_{vw} a_{wv} = -a_{vw} a_{wv}\]
for $w$ adjacent to $v$ (see \cite[Sec. 1a]{thebook} for signs used in defining the opposite of an $A_\infty$-algebra).  	

We use the reduced bar resolution of $\mathrm{k}$ as a right $A_\Gamma^{op}$-module to calculate $\operatorname{RHom}_{A_\Gamma^{op}} (\mathrm{k},\mathrm{k})$ which takes the form
$$ \operatorname{RHom}_{{A}^{op}}  (\mathrm{k},\mathrm{k}) \simeq 
\operatorname{hom}_{{A}^{op}} ( (A \otimes_{\mathrm{k}} T\bar{A})^{op}, \mathrm{k}) \ , $$
where $A=A_\Gamma$, $\bar{A}= A_\Gamma /\mathrm{k}$, and $T\bar{A}$ is the
tensor algebra of $\bar{A}_\Gamma$ over $\mathrm{k}$.  

The fact that $k=A_0$ allows us to identify $\bar{A}$ with the positive graded subalgebra $A_1
\oplus A_2$ of $A$. We follow the conventions in \cite[Sec. 1j]{thebook} for the DG-algebra
structure of $ \operatorname{hom}_{A^{op}} ((A \otimes_{\mathrm{k}} T\bar{A})^{op} , \mathrm{k})$. However, we view $ \operatorname{hom}_{A^{op}} ((A\otimes_{\mathrm{k}} T\bar{A})^{op}, \mathrm{k})$ as a DG-algebra rather than an $A_\infty$-algebra with $\mu^k =0$ for $k>2$ since $\mathscr{G}_\Gamma$ is always viewed as a DG-algebra. The difference is in the signs and this was explained in the introduction (see Eqn.~(\ref{dgsigns})). 

More precisely, a generator $t \in  \operatorname{hom}_{A^{op}} ( (A\otimes_{\mathrm{k}}
T\bar{A})^{op}, \mathrm{k})$ of bidegree $(r,s)$ is an $A^{op}$-module homomorphism $ t :
A\otimes_{\mathrm{k}} \bar{A}^{\otimes r} \to \mathrm{k}$ of internal degree $|t|=s$. Observe that,
any such $t$ maps an element $(a_{r+1}, a_{r}, \dots , a_1)$ to $0$ unless $a_{r+1} \in A_0$ because
of the $A^{op}$-module structure of $\mathrm{k}$.

The differential and the product on the DG-algebra $ \operatorname{hom}_{A^{op}} (
(A\otimes_{\mathrm{k}} T\bar{A})^{op}, \mathrm{k})$ are defined by
$$ (d t )(e_v, a_{r+1}, \dots , a_1) = \sum_{n=1}^{r} (-1)^{\dagger+|t|} t (e_v, a_{r+1} , \dots a_{n+2}, \mu^2_{A^{op}} (a_{n+1} , a_{n}) , a_{n-1} , \dots , a_1) \ , $$
and if $t_1$ and $t_2$ are two generators of length $r_1,r_2$ then,
$$ (t_2 \cdot  t_1) (e_v, a_{r_2+r_1} , \dots , a_1) = (-1)^{\ddagger +|t_1|} t_2(t_1 (e_v, a_{r_2+r_1}, \dots , a_{r_2+1}) , a_{r_2} , \dots, a_1) \ , $$
where $\dagger = \sum_{i=n}^{r+1} (|a_i|- 1)$ and $\ddagger = \sum_{i=r_2+1}^{r_2+r_1} (|a_i|  - 1)$.  

We now construct a chain map 
\[ \Phi : \mathscr{G}_{\Gamma^{op}} \to \operatorname{hom}_{A^{op}} ( (A\otimes_{\mathrm{k}}
T\bar{A})^{op}, \mathrm{k}) \]  
that respects the cohomological and internal gradings, first by defining it on the generators $g_{wv}$ and $h_v$ of the underlying
tensor algebra of $\mathscr{G}_{\Gamma^{op}}$, and then extending by mapping the product $p_2 p_1$ of two
elements $p_2$ and $p_1$ in $\mathscr{G}_{\Gamma^{op}}$ to the homomorphism $\Phi (p_2) \cdot \Phi
(p_1) \in \operatorname{hom}_{A^{op}} ( (A\otimes_{\mathrm{k}} T\bar{A})^{op}, \mathrm{k})$. 

Indeed, let us define $\Phi (g_{wv})$ and $\Phi (h_v)$ to be $A$-module homomorphisms each of which is nonzero only on a 1-dimensional subspace of $A\otimes_{\mathrm{k}} T\bar{A}$ given by 
$$\Phi (g_{wv}):  (e_v, a_{wv}) \mapsto \epsilon_{wv} e_w \ \mbox{ and } \  
\Phi (h_v) : (e_v, a_{vw} a_{wv}) \mapsto \epsilon_v e_v \ , $$
for any vertex $w$ adjacent to $v$ in $\Gamma$. Here the signs $\epsilon_{wv}$, $\epsilon_v$ are determined as follows.
For a vertex $v \in \Gamma_0 $, we set $\epsilon_v=(-1)^{\delta_v}$, where $\delta_v$ is the distance from the root of $\Gamma$ to the vertex $v$.
If $g_{wv}$ is an arrow in the quiver $\Gamma^{op}$, then define $\epsilon_{wv} = \epsilon_v$ and $\epsilon_{vw} = +1$. 
Note that $\frac{\epsilon_{wv} \epsilon_{vw}}{\epsilon_v}$ is $+1$ if and only if $g_{wv}$ is an arrow in the quiver $\Gamma^{op}$.

Observe that the internal gradings are 
	\[ |\Phi(g_{wv})| = - |a_{wv}|= -1 \text{\ \ \ and\ \ \ } |\Phi(h_v)| = - |a_{vw} a_{wv}| = -2,\] respectively. Note also that $\Phi$ takes the path-length grading on $\mathscr{G}_\Gamma$ to the cohomological grading
on $\operatorname{hom}_{A^{op}}((A\otimes_{\mathrm{k}}T\bar{A})^{op},\mathrm{k})$, hence $\Phi$ respects the bigraded structure of both sides.

The differentials on the DG-algebras $\mathscr{G}_{\Gamma^{op}}$ and
$\operatorname{hom}_{A^{op}} ( (A\otimes_{\mathrm{k}} T\bar{A})^{op}, \mathrm{k})$ obey the graded Leibniz rule, hence it suffices to check that
$$ d (\Phi (g_{wv})) = \Phi (d g_{wv}) = 0 \ \text{\ \ \  and\ \ \ } d (\Phi (h_v) ) = \Phi (d h_v) $$
to verify that $\Phi$ is a DG-algebra homomorphism.

The first identity follows immediately since both $g_{wv}$ and $\Phi (g_{wv})$ are in total degree $0$ and the differential vanishes here. To check the second identity, observe that $d (\Phi (h_v) )$ is nonzero only on the subspace of $A \otimes_{\mathrm{k}} T\bar{A} $ spanned by $$\{ (e_v, a_{wv} , a_{vw}): w \mbox{ is adjacent to } v \} \ , $$
and for every $w$ adjacent to $v$,
$$(d(\Phi (h_v) ))( e_v, a_{wv} , a_{vw}) = (-1)^{|\Phi(h_v)|+(|a_{wv}|-1)+(|a_{vw}|-1)} \Phi(h_v)(e_v, -a_{vw} a_{wv}) = -\epsilon_v e_v  \ . $$
Note that the appearance of the extra sign here is precisely the point where the use of $A_\Gamma^{op}$ rather than $A_\Gamma$ takes effect.  

On the other hand, 
$$\Phi (d h_v) = \Phi \left(\sum_w \frac{\epsilon_{wv} \epsilon_{vw}}{\epsilon_v} \ g_{vw} g_{wv} \right) 
= \sum_w  \frac{\epsilon_{wv} \epsilon_{vw}}{\epsilon_v} \ \Phi(g_{vw}) \cdot \Phi(g_{wv}) $$

For each  $w$ adjacent to $v$, $\Phi(g_{vw}) \cdot \Phi(g_{wv})$ is nonzero only on the subspace spanned by $ (e_v, a_{wv} , a_{vw})$ and
$$ ( \Phi(g_{vw}) \cdot \Phi(g_{wv})) (e_v, a_{wv} , a_{vw}) = (-1)^{|\Phi(g_{wv})| + (|a_{wv} |-1)} \Phi (g_{vw} )((\Phi(g_{wv})( e_v, a_{wv})),a_{vw}) = - \epsilon_{wv}\epsilon_{vw} e_v \ .$$

Indeed, we also have an extra sign here, and hence the second identity holds.

To prove the bijectivity of $\Phi$, consider a generator $(e_v, a_r, \dots , a_1)$ of $A \otimes_{\mathrm{k}} \bar{A}^{\otimes r}$. 
Note that such a generator is uniquely determined by the initial and terminal points of $a_i$ considered as paths in $A_\Gamma$ which in turn determine a unique path $g_r \cdots g_1$ of length $r$ in ${\mathscr{G}_\Gamma}$, so that the initial and terminal points of each arrow $g_i$ in the extended quiver $\widehat{\Gamma}$ match those of $a_{r+1-i}$. It is straightforward to check that 
$$ (\Phi(g_r \cdots g_1) ) (e_v, a_r, \dots , a_1) = \pm e_w , $$
where $w$ is the terminal point of $a_1$. This proves that $\Phi$ is injective since the algebra underlying ${\mathscr{G}_\Gamma}$ is the path algebra generated by the arrows in $\widehat{\Gamma}$. Moreover, the observation that $\Phi(g_r \cdots g_1) $ is nonzero only on the subspace of $A \otimes_{\mathrm{k}} T\bar{A}$ spanned by $(e_v, a_r, \dots , a_1)$ shows that $\Phi$ is surjective as well.
\end{proof} 

\begin{remark} As can be seen from the proof of Thm.~(\ref{koszulness}), we could arrange the definition of the DG-algebra isomorphism $\Phi$ so as to obtain an isomorphism: 
\[ \operatorname{RHom}_{A_\Gamma} (\mathrm{k},\mathrm{k}) \simeq \mathscr{G}_\Gamma \]
where $\mathrm{k} = A_\Gamma/(A_\Gamma)_{>0}$ is viewed as a right $A_\Gamma$-module. 
This is because there happens to be an isomorphism of algebras between $A_\Gamma$ and $A_\Gamma^{op}$. We have opted to use $A_\Gamma^{op}$ to be consistent with the general framework of Koszul duality (see \cite[Thm. 2.10.1]{bgs}). 
\end{remark}

The following corollary is immediate from Thm.~(\ref{koszulness}) and Thm.~(\ref{dgkoszul}): 

\begin{corollary} \label{koszulness2} Consider $\mathrm{k} = \mathscr{G}_\Gamma / (\mathscr{G}_\Gamma)_{r>0}$ as a right $\mathscr{G}_\Gamma$-module, and $A_\Gamma$ as a DG-algebra with trivial differential. 
There is a quasi-isomorphism of DG-algebras: 
\[ \operatorname{RHom}_{\mathscr{G}_\Gamma} (\mathrm{k},\mathrm{k}) \simeq A_\Gamma \] 
such that the cohomological and internal gradings on the left-hand-side coincide with each other and they agree with the path-length grading on the right-hand-side. 
\end{corollary} 
\begin{proof} In view of Thm.~(\ref{koszulness}) and Thm.~(\ref{dgkoszul}), we only need to check
    the hypothesis in Thm.~(\ref{dgkoszul}) but this is straightforward. Certainly, $A_\Gamma$ is
    positively graded and the local finiteness condition holds since $A_\Gamma$ is
    finite-dimensional (see \cite[Def. 2.1]{LPWZ}).  
\end{proof}

Since $A_\Gamma$ is known to be Koszul in the classical sense for non-Dynkin $\Gamma$, we easily get an alternative proof of formality result mentioned in part (1) of Thm.~(\ref{hermes}).

\begin{corollary} \label{rehermes} For $\Gamma$ non-Dynkin, $\mathscr{G}_\Gamma$ is formal, that is, it is quasi-isomorphic to the preprojective algebra $\Pi_\Gamma = H^0(\mathscr{G}_\Gamma)$. 
\end{corollary}
\begin{proof} Recall that the differential on the complex $\operatorname{RHom}_{A_\Gamma^{op}} (\mathrm{k},\mathrm{k})$ has bidegree $(1,0)$. Therefore, after applying homological perturbation lemma, we obtain a minimal $A_\infty$-structure on $Ext^*_{A^{op}}(\mathrm{k},\mathrm{k})$ such that $\mu^d$ has bidegree $(2-d,0)$. On the other hand, Koszulity of $A_\Gamma$ means that the two gradings agree at the level of cohomology. Therefore, it is impossible to have a non-trivial $\mu^d$ for $d \neq 2$. 
\end{proof}

Note that if $\Gamma$ is a Dynkin type graph, $\mathscr{G}_\Gamma$ is not
quasi-isomorphic to the preprojective algebra $\Pi_\Gamma$. Our result above
can be described as stating that $\mathscr{G}_\Gamma$ and $A_\Gamma$ are
$A_\infty$-Koszul dual. This should be seen as the natural extension to all
$\Gamma$ of the classical Koszul duality between $\Pi_\Gamma$ and $A_\Gamma$
which only worked when $\Gamma$ is non-Dynkin. 

Finally, in view of the Thm.~(\ref{koszulness}) and Cor.~(\ref{koszulness2}), we conclude from Keller's theorem \cite{keller} that there
is an isomorphism of Hochschild cohomologies as Gerstenhaber algebras. Besides this isomorphism, the following theorem also uses the fact that $HH_{2-*}(\mathscr{G}_\Gamma) \cong
HH^*(\mathscr{G}_\Gamma)$ by the duality result for smooth Calabi-Yau algebras \cite{VdB}, together
with \cite[Cor. 5.7]{BEE} which for technical reasons applies over $\mathbb{K}$ of characteristic 0, and Thm (\ref{chainmap}).
\begin{theorem} \label{voila} For any tree $\Gamma$, there is an isomorphism of Gerstenhaber
    algebras over $\mathbb{K}$: 
    \[ HH^*(\mathscr{G}_\Gamma) \cong HH^*(A_\Gamma). \]
    If $\Gamma$ is Dynkin type $A_n$ or $D_n$ (and conjecturally also for $E_6,E_7,E_8$) and $\mathbb{K}$
is of characteristic $0$, then we have:
\[ SH^*(X_\Gamma) \cong HH^*(\mathscr{G}_\Gamma) \cong HH^*(A_\Gamma). \]
\end{theorem} 

\begin{remark} Note that all of the Gerstenhaber algebras appearing in the
	above theorem are induced from a natural underlying 
	Batalin-Vilkovisky (BV) algebra structure. In the case of symplectic
	cohomology, BV-algebra structure is given by a geometric construction
	reminiscent of the loop rotation in string topology and in the cases of
	$\mathscr{G}_\Gamma$ and $A_\Gamma$, it is induced by the underlying Calabi-Yau
	structure on these DG-algebras, which allows one to dualize the Connes differential $B$ on
	Hochschild homology. However, the above theorem does not claim an
	isomorphism of the underlying Batalin-Vilkovisky structures. We believe that this can be
    achieved, however, it requires a finer investigation of Calabi-Yau structures.  
	On the other hand, we explain in Rem.~(\ref{BVremark}) that for $\Gamma$ non-Dynkin and
    non-extended Dynkin, we have an isomorphism of
Batalin-Vilkovisky algebras between $HH^*(\mathscr{G}_\Gamma)$ and $HH^*(A_\Gamma)$ as it turns out
that there is a unique way of equipping this Gerstenhaber algebra with a BV-algebra structure.   \end{remark}

\begin{remark} It is well-known that in the case when $\Gamma$ is Dynkin, the exact
	Lagrangian spheres $S_v$ split-generate the Fukaya category
	$\mathcal{F}(X_\Gamma)$ of compact exact Lagrangians - this follows for example by combining
	\cite[Lem. 4.15]{seidelgraded}  and \cite[Cor. 5.8]{thebook}. Furthermore,
	as mentioned in the beginning of Sec.~(\ref{seckoszul}), one expects that
	the non-compact Lagrangians $L_v$ split-generate the wrapped Fukaya category.  Hence, one could interpret the above result as showing that: \[
	HH^*(\mathcal{F}(X_\Gamma)) \cong HH^*(\mathcal{W}(X_\Gamma)) \] On the
	other hand, it is by no means the case that  
	$D^\pi \mathcal{F}(X_\Gamma)$ and $D^\pi \mathcal{W}(X_\Gamma)$ are equivalent as triangulated categories. (Here, we mean an equivalence between the Karoubi-completed triangulated closures of $\mathcal{F}(X_\Gamma)$ and $\mathcal{W}(X_\Gamma)$). This is clear from the fact that the latter category has objects with infinite-dimensional endomorphisms (over $\f{K}$) but every object in the former
    has finite-dimensional endomorphisms. More strikingly, the monotone Lagrangian tori studied in
    \cite{lekilimaydanskiy} give objects in $D^\pi \mathcal{W}(X_\Gamma)$ for $\Gamma = A_{n}$ with
    finite-dimensional endomorphisms and yet these do not belong to the category $D^\pi
    \mathcal{F}(X_\Gamma)$. One has to collapse the grading to $\mathbb{Z}_2$
    in order to admit these objects in $\mathcal{F}(X_\Gamma)$.
\end{remark}

In the next section, we give computations of $HH^*(A_\Gamma)$ for all trees $\Gamma$except $E_6,E_7,E_8$.

\section{Hochschild cohomology computations}
\label{weshallcompute}
\subsection{Non-Dynkin case} 
\label{nondynkin}
In this section we assume that $\Gamma$ is a non-Dynkin tree and describe the Hochschild cohomology
$HH^*(\mathscr{G}_\Gamma)$ of the associated Ginzburg DG-algebra. Note, however, that as explained in the introduction, when $\Gamma$ is non-Dynkin, $\mathscr{B}_\Gamma$ is a non-trivial deformation of $\mathscr{G}_\Gamma$, and so this computation does not directly give enough information to compute $HH^*(\mathscr{B}_\Gamma)$, and thus $SH^*(X_\Gamma)$. However, at least away from characteristic 0, the computation of $HH^*(\mathscr{G}_\Gamma) \cong HH^*(A_\Gamma)$ is still of geometric significance as it controls the deformations of the compact Fukaya category $\mathcal{F}(X_\Gamma)$. 

Recall that for non-Dynkin $\Gamma$, the cohomology $H^*(\mathscr{G}_\Gamma) \cong\Pi_\Gamma$ is supported in total degree $0$ and moreover $\mathscr{G}_\Gamma$ is formal, i.e. it is quasi-isomorphic to the preprojective algebra $\Pi_\Gamma$.  
Therefore we have an isomorphism of Gerstenhaber algebras \[ HH^*(\mathscr{G}_\Gamma)\cong HH^*(\Pi_\Gamma) \] where $\Pi_\Gamma$ is to be considered as a trivially graded algebra.
For any non-Dynkin quiver $\Gamma$, the Gerstenhaber structure of the Hochschild cohomology of $\Pi=\Pi_\Gamma$ is described in \cite{schedler} (and previously in \cite{BVEG} when $\mathrm{char}(\f{K})=0$).
We do not have anything new to say here, we simply review some of the results of \cite{BVEG} and \cite{schedler} to give a flavor of what's known. For an impressive amount of further information see the comprehensive work of Schedler \cite{schedler}.

The Hochschild cohomology $HH^*(\Pi_\Gamma)$ turns out to be  trivial in every grading except for $0,1$ and $2$. A way to see this is to use the Koszul bimodule resolution given in Eqn.~(\ref{koszulbi}). Recall that for $\Gamma$ non-Dynkin, $\Pi_\Gamma$ is Koszul in the classical sense with Koszul dual $A= A_\Gamma$. The latter has a decomposition into its graded pieces as $A= A_0 \oplus A_1 \oplus A_2$. Hence, the Koszul bimodule resolution takes the form:
\[ 0 \to \bigoplus_v e_v \Pi e_v  \to \bigoplus_v e_v A_1 \otimes_{\mathrm{k}} \Pi e_v \to \bigoplus_v e_v A_2 \otimes_\mathrm{k} \Pi e_v \to 0 \]
Moreover, it is well-known that $\Pi$ is Calabi-Yau of dimension $2$ (see \cite[Def. 3.2.3]{ginzburg}), hence a duality result of Van den Bergh \cite{VdB} applies and we have a canonical isomorphism 
 $$ HH^* (\Pi) \cong HH_{2-*} (\Pi) \ . $$ For the $\f{K}$-vector space structure of the Hochschild cohomology let us recall some general facts (see e.g. \cite{loday}) which apply to any algebra (with trivial grading and differential). The zeroth cohomology $HH^0(\Pi)$ is given by the center $Z(\Pi)$, and $HH^1(\Pi)$ is given by \emph{outer derivations} $Der (\Pi) / Inn (\Pi)$. Recall that a derivation is a linear map $D: \Pi \to \Pi$ satisfying the Leibniz rule, and each $a \in \Pi$ defines an inner derivation by $D_a (b )= ab-ba$. The zeroth homology $ HH_0(\Pi)$ is isomorphic to $\Pi_{cyc} := \Pi / [\Pi, \Pi]$, where $[\Pi, \Pi] \subset \Pi$ is the $\f{K}$-linear subspace spanned by the commutators.

 \begin{theorem}[\cite{schedler}-v1, Cor. 10.1.2, cf. \cite{BVEG}, Thm. 8.4.1]
\label{sched}
The $\f{K}$-vector space structure of the Hochschild cohomology $HH^*(\Pi)$ of the preprojective algebra associated to a non-Dynkin quiver is as follows.
\begin{enumerate}
\item If $\Gamma$ is extended Dynkin, then $HH^0 (\Pi) \cong Z(\Pi) \cong e_{v_0} \Pi e_{v_0}$, where $v_0$ is a vertex in $\Gamma$ whose complement is Dynkin. 
Otherwise the center $Z(\Pi)$ is isomorphic to $\f{K}$.
\item $HH^1(\Pi) \cong Der (\Pi) / Inn (\Pi) \cong Z(\Pi) \oplus (F \otimes_{\mathbb{Z}} \f{K} ) \oplus \left(T \otimes_{\mathbb{Z}} \bigoplus_{p} \operatorname{Hom}_{\mathbb{Z}} (\mathbb{F}_p, \f{K})\right) $, where $F$ and $T$ are the free and torsion parts of $\Pi^{\mathbb{Z}}_{cyc}$, respectively, and  $\Pi^{\mathbb{Z}}$ is the preprojective $\mathbb{Z}$-algebra associated to $\Gamma$. 
\item $HH^2 (\Pi) \cong HH_0 (\Pi) \cong \Pi_{cyc}$                     
\end{enumerate}
\end{theorem}

\begin{remark}
	In the extended Dynkin case, by the McKay correspondence $Z(\Pi)$ is isomorphic to the ring of invariant polynomials in $\f{K}[x,y]$ under the action  of the corresponding finite subgroup
    $G\subset SL_2(\f{K})$ as long as $\f{K}$ has $|G|$-{th} roots of unity (see \cite[Thm.
    9.1.1]{schedler}-v1). Furthermore, in this case $T$ is trivial and hence $HH^*(\Pi)$ is determined by $Z(\Pi)$ and $\Pi_{cyc}$, unless the characteristic of $\f{K}$ is a ``bad prime'' for $\Gamma$, i.e. $2$ for $\tilde{D}_n$, $2$ or $3$ for $\tilde{E}_6$ and $\tilde{E}_7$, and $2, 3$ or $5$ for $\tilde{E}_8$ \cite{schedler}. Note that the Hilbert series of $Z(\Pi)$ and $\Pi_{cyc}$, as algebras graded by path-length, are given in  \cite{EG} and \cite{schedler}. 
\end{remark}

The quotient $\Pi_{cyc}$ can be considered as a graded Lie algebra with  the path-length grading  and the Lie bracket induced by the \emph{necklace Lie bracket} $\{ \cdot , \cdot \}$ on $\Pi$, given by 
$$\{ p,q\} = \sum_{g_{wv} \in \Gamma_1} (\partial_{vw} q)(\partial_{wv} p) - (\partial_{wv} q)(\partial_{vw} p)  \ .$$
Here, for any path $p \in \Pi$ and adjoint pair $(v,w)$ in $\Gamma$,  $\partial_{wv} p$ is given as the sum
$$\sum_i g_{i-1} \cdots g_1 g_l \cdots g_{i+1}$$ 
taken over all $i$ for which the $i^{th}$ arrow $g_i$ in the path $p=g_l \cdots g_1$ is $g_{wv}$. 

Note that the Lie bracket $[D, D'] = D \circ D' - D' \circ D$ on $Der (\Pi) / Inn (\Pi)$ coincides with the Gerstenhaber bracket on $HH^1(\Pi)$ in favorable cases, e.g. if $\mathrm{char}(\f{K})=0$ and $\Gamma$ is not extended Dynkin.

The Lie brackets above are used to describe the (cup) product as well as the Gerstenhaber bracket on
$HH^*(\Pi)$ in \cite{BVEG}, when $\mathrm{char}( \f{K}) = 0$. We now recall the description of the
Gerstenhaber algebra structure of $HH^*(\Pi)$ in \cite{schedler}, for arbitrary
$\mathrm{char}(\f{K})$, using the BV operator $\Delta$ dual to the Connes
differential (see, e.g. \cite{loday}) on $HH_*(\Pi)$. The \emph{Euler derivation} $\mathrm{eu}$ on $\Pi_{cyc}$ is defined as
multiplication by $l$ on each path of length $l$, and the derivation $u$, called \emph{half
Euler derivation} in \cite{schedler}, multiplies each path by the number of edges from $\Gamma$ that
it contains.  Note that we have $\mathrm{eu} = 2u$ as elements of $HH^1(\Pi)$. In other words, their difference is an inner derivation. The first summand of
$HH^1(\Pi)$ in Thm.(\ref{sched}) consists of multiples of $u$ by $Z(\Pi)$.

\begin{theorem}[\cite{schedler}-v1, Thm. 10.3.1] As a BV-algebra, $HH^*(\Pi)$ is determined by the following properties.
\begin{enumerate} 
    \item The graded-commutative product $$\cup : HH^i(\Pi) \otimes HH^j(\Pi) \to HH^{i+j}(\Pi)$$ is given as follows.
\begin{enumerate}
\item If $\theta, \theta' \in Der (\Pi) / Inn (\Pi) \cong HH^1(\Pi)$ and $\theta'$ belongs to the $F \otimes_{\mathbb{Z}} \f{K}$ summand of $HH^1(\Pi)$, then $\theta \cup \theta'$ is obtained by considering $\theta'$ as an element of $\Pi_{cyc}$ and applying the derivation $\theta$ to it.
\item If none of $\theta, \theta' \in  HH^1(\Pi)$ belongs to  the $F \otimes_{\mathbb{Z}} \f{K}$ summand, then $\theta \cup \theta' =0$.
\item If $ij=0$, then $\cup$ is given by multiplication in $\Pi$. 
\end{enumerate}
\item The BV-operator $$\Delta: HH^i(\Pi) \to HH^{i-1}(\Pi)$$ dual to the Connes differential is given as follows.
\begin{enumerate} 
\item We have
$$\Delta (u) =1,\ \ \  \Delta( z \cup \theta ) = \theta (z) + z \Delta(\theta) $$
for every $z \in HH^0(\Pi) \cong Z(\Pi) $, $\theta \in Der (\Pi) / Inn (\Pi) \cong HH^1(\Pi)$.  
The BV-operator vanishes on the $\left(T \otimes_{\mathbb{Z}} \bigoplus_p \operatorname{Hom}_{\mathbb{Z}} (\mathbb{F}_p, \f{K})\right)$ summand of $HH^1(\Pi)$.
\item The operator $\Delta: HH^2(\Pi)\cong \Pi_{cyc} \to  Der (\Pi) / Inn (\Pi) \cong HH^1(\Pi) $ maps to the $F \otimes_{\mathbb{Z}} \f{K}$ summand and it is  given by
$$ \Delta (g_l \cdots g_1) = \sum_{i=1}^l \pm  \partial_{g_i^*} (\cdot) g_{i-1} \cdots g_1 g_l \cdots g_{i+1} \ , $$
where each $g_i$ is an arrow in the double of the quiver $\Gamma$ and the sign is positive if and
        only if $g_i \in \Gamma$.
\end{enumerate}

\end{enumerate}
\end{theorem}

\begin{remark} \label{BVremark} A word of caution is in order. For $\Gamma$ non-Dynkin, the
    BV-algebra structure on $HH^*(\Pi_\Gamma)$ is induced by the 2-Calabi-Yau structure (in the sense
    of Ginzburg \cite{ginzburg}, also known as smooth Calabi-Yau structure) on the homologically smooth algebra $\Pi_\Gamma$. This means that we have an isomorphism of
    $\Pi_\Gamma$-bimodules: 
    \[ \Pi_\Gamma \simeq \mathrm{RHom}_{\Pi_\Gamma-\Pi_\Gamma} (\Pi_\Gamma, \Pi_\Gamma \otimes \Pi_\Gamma)[2], \]
    where the bimodule structure on the right is with respect to the inner bimodule structure on
    $\Pi_\Gamma \otimes \Pi_\Gamma$ and $\mathrm{RHom}$ is taken with respect to the outer bimodule structure on $\Pi_\Gamma
    \otimes \Pi_\Gamma$. Two such 2-Calabi-Yau structures differ by an invertible element in
    $HH^0(\Pi_\Gamma)$. The effect by such an invertible $\phi$ is to replace $\Delta$ by $\phi^{-1} \Delta \phi$ (\cite[Rmk. 4.8]{seidelPicard}).

    We can consider the Koszul dual notion. Namely, by Koszul duality, for $\Gamma$ non-Dynkin, we have 
    $HH^*(\Pi_\Gamma) \cong HH^*(A_\Gamma)$ and then the BV-algebra structure can be seen as naturally
    arising, from a weak Calabi-Yau structure on $A_\Gamma$. Recall that, a weak Calabi-Yau
    structure (also known as Frobenius structure or compact Calabi-Yau structure) of dimension 2 on the finite-dimensional algebra $A_\Gamma$ is a quasi-isomorphism of $A_\Gamma$-bimodules:
    \[ A_\Gamma \simeq A_\Gamma^{\vee}[-2], \]
    where $A_\Gamma^{\vee}$ is the $\mathbb{K}$-linear dual of $A_\Gamma$. Two such Calabi-Yau structures again differ by an invertible element in
    $HH^0(A_\Gamma)$. 
    
    In any case, if $\Gamma$ is non-Dynkin and non-extended Dynkin, then by Thm.~(\ref{sched}), $HH^0(\Pi_\Gamma) \cong
    HH^0(A_\Gamma) \cong \mathbb{K}$ is rank 1 generated by the identity, hence there exists (up to scaling) at most one (Ginzburg) Calabi-Yau structure
on $\Pi_\Gamma$ and at most one (weak) Calabi-Yau structure on
    $A_\Gamma$. These Calabi-Yau structures can either be constructed algebraically as in
 \cite{ginzburg} or symplectically as a manifestation of Poincar\'e duality for the Fukaya
category of compact Lagrangians or the open Calabi-Yau property of the wrapped Fukaya category. 

Now, {\bf suppose} $\mathscr{B}_\Gamma \simeq \mathscr{G}_\Gamma$. Then, since
$\mathscr{G}_\Gamma$ is formal, we would have an isomorphism $SH^*(X_\Gamma) \cong
HH^*(\mathscr{B}_\Gamma) \cong HH^*(\Pi_\Gamma)$. Under this isomorphism, the natural $BV$-algebra structure on $SH^*(X_\Gamma)$ given by the loop rotation operator $\Delta:
  SH^*(X_\Gamma) \to SH^{*-1}(X_\Gamma)$ has to coincide with the algebraically constructed
BV-algebra structure on $HH^*(\Pi_\Gamma)$ in the case that $\Gamma$ is non-Dynkin and
non-extended Dynkin. 
    
	On the other hand, combining the results from \cite[Thm. 24]{menichi} and
	\cite[Ch. 12-14]{abouzmanuz} one deduces that 
	\[ SH^*(T^*S^2) \cong HH^*(C_{2-*}(\Omega S^2)) \cong HH^*(C^*(S^2)) \] does not admit a
    \emph{dilation} over a field of characteristic 2 \footnote{An independent verification of
    this fact based on a Morse-Bott computation of BV-operator on $SH^*(T^*S^2)$ was
communicated to us by P. Seidel.}.  Recall that a dilation is an element $b \in
	SH^1(X_\Gamma)$ such that \[ \Delta b =1 \] where $\Delta : SH^*(X_\Gamma)
	\to SH^{*-1}(X_\Gamma)$ is the BV-operator in symplectic cohomology. Furthermore,
    since $T^*S^2$ can be embedded as a Liouville subdomain of $X_\Gamma$, one has a restriction
    map, $SH^*(X_\Gamma) \to SH^*(T^*S^2)$ which is a map of BV-algebras. Therefore, a
    dilation on $X_\Gamma$ can be restricted to a dilation on $T^*S^2$. 
    On
	the other hand, we see from the above theorem that there is a class $u \in HH^1(\Pi_\Gamma)$ that is sent to the identity by
	the BV-operator induced from the Calabi-Yau structure on $\Pi_\Gamma$. Hence, we arrive at a contradiction. 

    This is in agreement with Rmk. (\ref{subtree}) where we have seen that $\mathscr{B}_\Gamma$ is a non-trivial deformation of $\mathscr{G}_\Gamma$ over a field of characteristic 2.
\end{remark}

\subsection{Dynkin case} 

In this section we compute the Hochschild cohomology of the zigzag algebra
$A_\Gamma$ associated with a \emph{Dynkin} tree. If the underlying tree
$\Gamma$ is of type $A_1$, i.e. a single vertex, then $A_\Gamma =
\f{K}[x]/(x^2)$ with $|x|=2$ and it is a Koszul algebra. Its Hochschild
cohomology was computed in Ex.~(\ref{singlevertex}) above. Thus, hereafter we
assume $\Gamma \neq A_1$. It turns out that if the underlying tree $\Gamma$ is
of Dynkin type but not a single vertex, then $A_\Gamma$ is an almost-Koszul
algebra (in the sense of \cite{almostkoszul}). In this situation, Koszul
complex leads to a construction of a minimal \emph{periodic} resolution. We
first review the basics of quadratic algebras and the associated Koszul
complexes. 

\subsubsection{Zigzag algebra $A_\Gamma$ as a trivial extension}
\label{sec-trivial}

Recall that for any $\Gamma$, the zigzag algebra $A_\Gamma$ is defined as the quotient of the path
algebra $\f{K}\mathrm{D}\Gamma$ of the double quiver $\mathrm{D}\Gamma$ by the ideal $J$ generated by the elements 	\begin{itemize}
	\item $a_{u v} a_{v w}$ such that $u \neq w$, where $v$ is adjacent to both $u,w$, and
	\item $a_{v w} a_{w v} - a_{vu} a_{uv}$ where $v$ is adjacent to both $u,w$.
	\end{itemize}
	Clearly, this is an example of a quadratic algebra over $\mathrm{k}$ where $V$ is the
    $\f{K}$-vector space generated by the edges $a_{wv}$ of $\mathrm{D}\Gamma$ and supported in
    grading 1. The path-length grading on $\f{K}\mathrm{D}\Gamma$  descends to $A_\Gamma$ where it is supported in degrees $0,1$ and $2$.  It is straightforward to verify that:
	\begin{proposition} For any tree $\Gamma$ the quadratic dual $A^{!}_\Gamma$ of the zigzag algebra $A_\Gamma$ is the preprojective algebra $\Pi_\Gamma$, when both are equipped with path-length grading. 
	\end{proposition} \QED 

As mentioned before, when $\Gamma$ is a single vertex, or not a Dynkin type
tree, $A_\Gamma$ is a Koszul algebra. For these cases, we have already computed
$HH^*(A_\Gamma)$ above (see Sec.~(\ref{nondynkin}) and Ex.~(\ref{singlevertex})). Henceforth, we will assume that $\Gamma$ is
\emph{Dynkin}, but not a single vertex. These cases are the only cases when  $A^{!}_\Gamma = \Pi_\Gamma$ is finite-dimensional. 

Let us drop $\Gamma$ from the notation for the moment and write \[ A = A_0
\oplus A_1 \oplus A_2  \text{\ \ \ and \ \ \ } \Pi= \Pi_0 \oplus \Pi_1 \oplus
\ldots \oplus \Pi_{h-2} \] for the graded pieces of $A$ and $\Pi$. 
Here $h$ stands for the Coxeter number of the Dynkin tree and it is equal to $n+1$, $2n-2$, $12$, $18$, and $30$, for $A_n$, $D_n$, $E_6$, $E_7$, and $E_8$, respectively (\cite{almostkoszul}). 

It turns out that, in this case, $A_\Gamma$ is not Koszul and its Koszul complex
(\ref{koszulcomplex}) is not acyclic. Indeed, the Koszul complex is given by 
\begin{align}\label{koszul} 
0\to A_\Gamma \to \Pi_1 \otimes_{\mathrm{k}} A_\Gamma \to \cdots \to \Pi_{h-2} \otimes_{\mathrm{k}} A_\Gamma  \rightarrow 0
\end{align} 
and it fails to be exact at the right end but only there (\cite{almostkoszul}). Nonetheless, in \cite{almostkoszul} the authors are able to modify the Koszul bimodule complex to obtain a $(2h-2)$-periodic complex that computes Hochschild cohomology of $A_\Gamma$. Indeed, the algebras $A_\Gamma$ belong to a class of periodic algebras which are \emph{almost Koszul}. 

We will, however, now turn to a slightly different approach, which makes use of the fact that $A_\Gamma$ is isomorphic to a \emph{trivial extension algebra}. 

\begin{definition} Let $B$ be a finite-dimensional algebra over the field $\f{K}$. Let $B^{\vee}:=Hom_{\f{K}}(B, \f{K})$ be the linear dual of $B$, viewed naturally as a $B$-bimodule. The trivial extension algebra of $B$, denoted by $\mathcal{T}(B)$, is the vector space $B \oplus B^{\vee}$ equipped with the multiplication:	
	\[ (x,f)\cdot (y,g) = (xy, xg+fy) \]
	If $B$ is graded, to get a $CY2$ algebra, we grade $\mathcal{T}(B)$ so that $\mathcal{T}(B) = B \oplus B^{\vee}[-2]$.
\end{definition} 

Let $A^{\to} = \f{K}\Gamma / J$ be the quotient of the path algebra of a quiver with respect to an arbitrary orientation of the edges modulo the ideal generated by paths of length 2. The following proposition appears in \cite[Prop. 9]{huekho} and results from an easy computation.

\begin{proposition} \label{trivial} $A_\Gamma$ is isomorphic to the trivial extension algebra $\mathcal{T}(A^{\to})$. 
\end{proposition} \QED

In particular, if we orient $\Gamma$ so that each vertex is either a sink or a source, then there are no paths of length 2, hence $A_\Gamma$ is a trivial extension algebra of the path algebra $\f{K}\Gamma$ in the bipartite orientation.

\begin{remark} There is a way to understand the above proposition in terms of
	symplectic topology. Namely, one can consider a Lefschetz fibration $f:
	\f{C}^3 \to \f{C}$, $(x,y,z) \mapsto f(x,y,z)$ given by perturbing the simple singularities
	\begin{align*}
		A_n &:x^2+y^2+z^{n+1} \text{ for } n \geq 1 , \\
		D_n &:x^2+zy^2+z^{n-1} \text{ for } n \geq 4, \\
		E_6 &: x^2+y^3+z^4, \\
		E_7 &: x^2+y^3+yz^3, \\
		E_8 &:x^2+y^3+z^5. 
	\end{align*} 
	One can then identify the surface $X_\Gamma$ with a regular fiber of
	these fibrations, i.e., the Milnor fibre of the singularity. The spheres
	$S_v$ can be identified with the vanishing spheres and the corresponding thimbles
	generate the Fukaya-Seidel category of $f$ by a famous result of Seidel \cite[Thm. 18.24]{thebook}. For a suitable choice of grading structures and ordering of objects, the Floer endomorphism algebra
	$A^{\to}$ of these thimbles in the Fukaya-Seidel category of $f$ coincides
	with the path algebra of $\f{K}\Gamma$ modulo the ideal generated by
	length 2 paths. The algebra isomorphism \[ A_\Gamma = A^{\to} \oplus A^{\to}[-2]
\] follows from general relationship between Fukaya-Seidel category of a
	Lefschetz fibration and the Fukaya category of its fiber (see \cite[Sec. 4]{seidelshhh}). 
\end{remark}

We next recall the following theorem about trivial extension algebras, which we will apply to path
algebras of quivers whose underlying graph is a tree. Note that by a well-known result of
Bernstein-Gelfand-Ponomarev \cite{bgp}, the path algebras $\f{K}Q$ of quivers $Q$ obtained by
orienting edges of the same \emph{tree} in different ways are derived equivalent algebras.

\begin{theorem} \label{ricky}(Rickard \cite{rickard}) Suppose $C$ and $D$ are derived equivalent algebras. Then their trivial extensions $\mathcal{T}(C)$ and $\mathcal{T}(D)$ are also derived equivalent. In particular, $HH^*(\mathcal{T}(C))$ and $HH^*(\mathcal{T}(D))$ are isomorphic as Gerstenhaber algebras.
\end{theorem} 

Our strategy will be to apply the above theorem to $\mathcal{T}(A^{\to})=A_\Gamma$ to
pass to another algebra whose Hochschild cohomology is previously computed.
However, it is important to note that the above theorem is for trivially graded
algebras. On the other hand, we need to compute $HH^{*}(A_\Gamma)$ as a
bigraded algebra. What's worse, since $A_\Gamma$ has elements in both even and
odd degrees, we cannot simply forget about the grading and reinstate it
afterwards, as in a graded resolution, odd elements affects the signs. 

We next explain how to deal with this tricky point. Namely, recall from Prop.~(\ref{compacts}) that $A_\Gamma$ is the graded algebra obtained as 
\[ A_\Gamma = \bigoplus_{v,w} HF^*(S_v,S_w) \] 

On the other hand, given integers $\sigma_v \in \f{Z}$ for every vertex $v$, we can
define another graded algebra: \[ \tilde{A}_\Gamma = \bigoplus_{v,w} Hom(S_v
	[\sigma_v] , S_w[\sigma_w]) = \bigoplus_{v,w} HF^*(S_v,S_w)[\sigma_w-\sigma_v] \] where
	$S_v[n_v]$ denotes a graded object whose grading is shifted down by
	$n_v$.  Clearly, $A_\Gamma$ and $\tilde{A}_\Gamma$ are graded Morita
	equivalent (in particular, derived equivalent). 
	Therefore, the (graded) Hochschild cohomologies of $A_\Gamma$ and $\tilde{A}_\Gamma$ are canonically
    isomorphic (see for ex. \cite[Sec. (1c)]{seidelflux}). Hence, for the purpose of computing
    Hochschild cohomology of $A_\Gamma$, we can choose the shifts $\sigma_v$ so that the shifted
    algebra is supported in even degrees. In fact, using the standard tree form of $\Gamma$ as in
    Fig.~(\ref{fig1}), we simply shift the object $S_v$ up $S_v[-\delta_v]$, where $\delta_v$ is the
    distance from the root to the vertex $v$. In this way, any arrow in the double $\mathrm{D}\Gamma$ is in degree $0$ or $2$ according to whether it points towards or away from the root.

{\bf Summary} To compute $HH^*(A_\Gamma)$ as a graded Gerstenhaber algebra, we follow this procedure:
\begin{itemize}
	\item First check that it is possible to shift gradings so that $A_\Gamma$ is supported in even degrees. 
      \item Forget the grading all together, and treat $A_\Gamma$ as an
ungraded algebra. 
	\item Compute the algebra structure of Hochschild cohomology of the
ungraded algebra by relating it to previous computations using derived equivalences of ungraded algebras in Thm.~(\ref{ricky}). This algebra will have only the cohomological grading $r$.
	\item Finally, reinstate the $s$-grading on $HH^*(A_\Gamma)$ by finding explicit (graded)
cocycles for the generators of Hochschild cohomology as an algebra. 
\end{itemize}

\subsubsection{Type $A$} 
\label{typeA}

Throughout this section, $\Gamma$ is the Dynkin tree $A_n$, $n>1$. We describe
the Hochschild cohomology ring of the zigzag algebra $A_\Gamma$ in detail. We
follow the strategy outlined in the previous section. Namely, we first
determine the Hochschild cohomology of $A_\Gamma$ as an ungraded algebra. The
result will be singly graded with the cohomological grading $r$. We then
reinstate the $s$ grading by explicitly identifying generators.

As was mentioned in Prop.~(\ref{trivial}), $A_\Gamma$ is isomorphic to the
trivial extension algebra of the path algebra $\f{K}Q$ of the quiver $Q$ with
the underlying tree $\Gamma=A_n$ and oriented with the bipartite orientation
(see Fig.~(\ref{bipar})). Furthermore, as explained above, the derived equivalence class of a path algebra of quiver, and hence by Thm.~(\ref{ricky}), the derived equivalence
class of trivial extensions of $\f{K}Q$ does not depend on the choice of the orientation of the
edges of the underlying tree. 

\begin{figure}[htb!]
\centering
\begin{tikzpicture}
	\tikzset{vertex/.style = {style=circle,draw, fill,  minimum size = 2pt,inner sep=1pt}}
\tikzset{edge/.style = {->,>=stealth',shorten >=8pt, shorten <=8pt  }}

\node[vertex] (a) at  (0,0) {};
\node[vertex] (c) at  (8,0) {};
\node[vertex] (a1) at (1.5,0) {};
\node[vertex] (a2) at (3,0) {};

\draw[edge] (a)  to (a1);
\draw[edge] (a2) to (a1);

\path (a2) to node {\dots} (c);
\node [shape=circle,minimum size=2pt, inner sep=1pt] (a3) at (4.5,0) {};
\draw[edge] (a2) to (a3);

\node [shape=circle,minimum size=2pt, inner sep=1pt] (c1) at (6.5,0) {};
\draw[edge] (c) to (c1);

\end{tikzpicture}
\caption{$A_n$ quiver in bipartite orientation}
\label{bipar}
\end{figure}
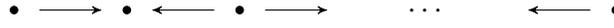

Let $B_\Gamma$ be the trivial extension algebra of the path algebra of $\Gamma=A_n$ where the underlying quiver is now oriented in the linear orientation (see Fig.~(\ref{linear})). 

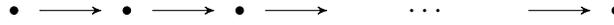
\begin{figure}[htb!]
\centering

\begin{tikzpicture}
	\tikzset{vertex/.style = {style=circle,draw, fill,  minimum size = 2pt,inner sep=1pt}}
\tikzset{edge/.style = {->,>=stealth',shorten >=8pt, shorten <=8pt  }}

% vertices
\node[vertex] (a) at  (0,0) {};
\node[vertex] (c) at  (8,0) {};
\node[vertex] (a1) at (1.5,0) {};
\node[vertex] (a2) at (3,0) {};

% edges

\draw[edge] (a)  to (a1);
\draw[edge] (a1) to (a2);

\path (a2) to node {\dots} (c);
\node [shape=circle,minimum size=2pt, inner sep=1pt] (a3) at (4.5,0) {};
\draw[edge] (a2) to (a3);

\node [shape=circle,minimum size=2pt, inner sep=1pt] (c1) at (6.5,0) {};
\draw[edge] (c1) to (c);

\end{tikzpicture}

\caption{$A_n$ quiver in linear orientation}
\label{linear}
\end{figure}

Let $\tilde{A}_{n-1}$ be the extended Dynkin quiver of type $A_{n-1}$, namely the quiver with cyclic orientation whose underlying graph is a
simple cycle with $n$ vertices and $n$ edges (see Fig.~(\ref{cyclicorientation})), and let us denote the ideal generated by paths of
length $\geq n+1$ by $J_{n+1}$.

\begin{figure}[htb!]
\centering
\begin{tikzpicture}
\tikzset{vertex/.style = {style=circle,draw, fill,  minimum size = 2pt,inner sep=1pt}}
\tikzset{edge/.style = {->,>=stealth',shorten >=8pt, shorten <=8pt  }}
\def \n {5}
\def \radius {1.5cm}
\def \margin {8} % margin in angles, depends on the radius

\foreach \s in {1,...,\n}
{
\node[vertex] at ({360/\n * (\s - 1)}:\radius) {} ;
  \ifthenelse{\s=1}{  
  \draw[edge, dashed] ({360/\n * (\s - 1)+\margin}:\radius) 	
  arc ({360/\n * (\s - 1)+\margin}:{360/\n * (\s)-\margin}:\radius);}
   \draw[edge] ({360/\n * (\s - 1)+\margin}:\radius) 	
   arc ({360/\n * (\s - 1)+\margin}:{360/\n * (\s)-\margin}:\radius); 
}

\end{tikzpicture}
\caption{Cyclic quiver $\tilde{A}_{n-1}$ }
\label{cyclicorientation}
\end{figure}
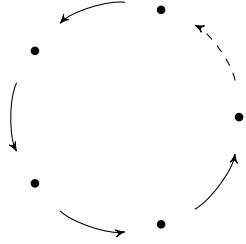

The following well-known fact (cf. \cite{almostkoszul}) can be verified by identifying $\f{K}\Gamma$ with its image under the natural inclusion $\f{K}\Gamma \to \f{K}\tilde{A}_{n-1}/J_{n+1}$, and observing that the subspace of $\f{K}\tilde{A}_{n-1}/J_{n+1}$ spanned by paths containing the unique arrow in the complement of $\Gamma$ in $\tilde{A}_{n-1}$ is canonically isomorphic to the linear dual of $\f{K}\Gamma$ as a $\f{K}\Gamma$-bimodule. 

\begin{lemma} 
$B_\Gamma$ is isomorphic to the truncated algebra
$\f{K}\tilde{A}_{n-1}/J_{n+1}$.
\end{lemma}

The derived equivalence between $A_\Gamma$ and $B_\Gamma$ implies an
isomorphism between the Hochschild cohomology rings. On the other hand, the
Hochschild cohomology of the (trivially graded) algebra $B_\Gamma$ is studied
in \cite{holm, ErHo, BLM}. In particular, the algebra structure of
$HH^*(B_\Gamma)$ over a field of arbitrary characteristic was already known.
Our contribution is to determine the internal $s$ grading coming from the grading
of $A_\Gamma$. We have the following result:

\begin{theorem} 
	\label{hha} 

As a (graded) commutative $\f{K}$-algebra the $(r,s)$-bigraded the Hochschild cohomology algebra \[ HH^*(A_\Gamma) = \bigoplus_{r+s=*} HH^r(A_\Gamma, A_\Gamma[s]),  \ \  \]
of the graded $\mathrm{k}$-algebra $A_\Gamma$ is given by the following generators and relations. (The subscripts of the generators, except for
     $s_i$, refer to total degrees.)  

\begin{itemize}
	\item Suppose $\mathrm{char} \f{K} \nmid n+1$. We have generators labeled along with their bidegrees $(r,s)$ given by:
		\begin{align*} s_1, \dots, s_n &\text{\ \ \ \ \ } (0, 2) \\
			t_1 &\text{\ \ \ \ \ } (1, 0) \\
			t_0 &\text{\ \ \ \ \ } (2,-2)  \\
			t_{\text{\scriptsize{-}}2} &\text{\ \ \ \ \ } (2n,-2n-2)  
		\end{align*}
and relations 
 \[ s_i s_j=s_i t_j=t_1^2=t_0^n=0 \ . \]
\item Suppose $\mathrm{char} \f{K} \mid n+1$. We have generators labeled along with their bidegrees $(r,s)$ given by:
	\begin{align*}  s_1, \dots, s_n, &\text{\ \ \ \ \ } (0, 2) \\
			t_1 &\text{\ \ \ \ \ } (1, 0) \\
			t_0 &\text{\ \ \ \ \ } (2,-2)  \\
			u_{\text{\scriptsize{-}}1} &\text{\ \ \ \ \ } (2n-1,-2n) \\  
                        t_{\text{\scriptsize{-}}2} &\text{\ \ \ \ \ } (2n,-2n-2)  
		\end{align*}
and relations 
\begin{align*} s_i s_j= s_i t_1 = s_i t_0 & =t_1^2=0  \\
	s_i u_{\text{\scriptsize{-}}1} &= t_1t_0^{n-1} \\ 
	s_i t_{\text{\scriptsize{-}}2} & =  t_0^n \\     
	t_0 u_{\text{\scriptsize{-}}1} &= t_1t_{\text{\scriptsize{-}}2} \\   t_1u_{\text{\scriptsize{-}}1} &= \alpha\ t_0^n \\  u_{\text{\scriptsize{-}}1}^2   &= \beta\ t^{n-1}_0 t_{\text{\scriptsize{-}}2}  
\end{align*}
where $\alpha=\beta=1$ if $\mathrm{char} (\f{K}) = 2$ and $4 \nmid n+1 $, otherwise $\alpha = \beta =0$. 
\end{itemize}
\end{theorem}

\begin{proof} The presentation of $HH^*(A_\Gamma)$ given above is adapted from
	the presentation of $HH^*(B_\Gamma)$ as a $\f{K}$-algebra graded by the
	cohomological grading, which was calculated in \cite[Thm. 8.1 and
	8.2]{holm} and \cite[Thm 5.19]{ErHo}. 
	In view of the isomorphism between
	$HH^*(A_\Gamma)$ and $HH^*(B_\Gamma)$ as $\f{K}$-algebras graded with
	respect to the cohomological $r$ gradings, it remains to determine the
	$s$ gradings. 
	In particular, the rank of $ HH^r(B_\Gamma) \cong \bigoplus_{s} HH^r(A_\Gamma, A_\Gamma[s])$ is given explicitly in \cite{holm, ErHo} for each $r$ and it can be recovered from the presentations in the statement. 
	We will make extensive use of this information in the following arguments. 
    
In what follows, we describe generators as elements of the reduced bar-resolution:
\begin{equation} \label{reducedbar} CC^* (A, A) := \mathrm{hom}_{\mathrm{k}}(T\bar{A},A) \end{equation}
where $A= A_\Gamma$ and $\bar{A}= A/\mathrm{k}$. The grading on $A$ gives a decomposition:
\[ CC^{*} (A,A) = \bigoplus_{*= r+s} CC^r(A,A[s]) \]
where the Hochschild differential $\delta$ is of bidegree $(1,0)$. We find explicit cocycles for $r=0,1,2$ and show that the $s$ gradings of other
generators are determined by the relations given above.

	As a graded algebra $A_\Gamma = A_0 \oplus A_1 \oplus A_2$, with components given by: 
\begin{align*}  
A_0 = \bigoplus_{i=1}^n \f{K}e_i, \ \ \ A_1  = \bigoplus_{i=1}^{n-1} \f{K}a_i \oplus  \bigoplus_{i=1}^{n-1} \f{K}b_i, \ \ \ A_2 &= \bigoplus_{i=1}^n \f{K}s_i 
\end{align*}
where $e_{i+1}a_i e_i = a_i$, $e_i b_i e_{i+1} = b_i$, and $s_{i+1} = a_i b_i = b_{i+1} a_{i+1}$.

The Hochschild differential $\delta$ in the complex (\ref{reducedbar}) is given by the formula in \cite[Eqn. 1.8]{thebook} (recall also the convention in Eqn.~(\ref{dgsigns})). We will only need the differentials on $CC^r(A, A[s])$ for $r=0,1,2$. These are given by:
\begin{align*}
\delta (c) (x_1) &= \mu^2(x_1,c) + (-1)^{(s-1)(|x_1|-1)} \mu^2(c, x_1)  \ ,\\
    \delta (c) (x_2, x_1) &= \mu^2(x_2,c(x_1)) + (-1)^{(s-1)(|x_1|-1)} \mu^2(c(x_2), x_1) +(-1)^{s} c(\mu^2(x_2,x_1))   \ , \\
	\delta (c) (x_3, x_2, x_1) &= \mu^2(x_3, c(x_2,x_1)) + (-1)^{(s-1)(|x_1|-1)} \mu^2(c(x_3, x_2),x_1) \\ & \omit\hfill $+ (-1)^{s} c(x_3, \mu^2(x_2,x_1))+ (-1)^{s+|x_1|-1} c(\mu^2(x_3,x_2), x_1) \ , $ 
\end{align*}
for $c\in CC^0(A, A[s])$, for $c\in CC^1(A, A[s])$, and for $c\in CC^2(A, A[s])$, respectively.
	
	{\bf r=0}: The $0$-cocycles are given by central elements. The identity element 
\[ \sum_j e_j \in CC^{0}(A,A[0]) \]
and the elements 
\[ s_i  \in CC^{0}(A,A[2]) \ \ \ \text{for} \ \ \ i=1, \ldots, n \]
give a basis of the center of $A$ over $\f{K}$.

{\bf r=1}: The $1$-cocycles are given by derivations. We define a $1$-cocycle $\tau_1 \in CC^1(A,A[0])$ by
\begin{align*} 
\tau_1 (a_i) = -a_i \ \ , \ \  \tau_1 (b_i) = 0 \ \ , \ \ \tau_1 (s_i) = s_i
\end{align*} 
for all $i=1,\ldots n$. 
It is straightforward to check that $\tau_1$ is a derivation but not an inner derivation so it is a non-trivial element of  $\bigoplus_s HH^{1}(A, A[s])$ which is $1$-dimensional for any $\f{K}$. Therefore, any generator of this group, in particular $t_1$, must have the same $s$ grading as $\tau_1$.

{\bf r=2}: We define a $2$-cocycle $\tau_0 \in CC^2(A,A[-2])$ given by:
\begin{align*} 
\tau_0 (a_i,b_i) &= (-1)^i e_{i+1}\\ 
\tau_0 (a_i, s_i) &=  (-1)^{i+1} a_i \\
\tau_0 (s_{i}, b_i) &= (-1)^{i} b_i \\ 
\tau_0 (s_{i}, s_i) & = (-1)^{i+1} s_i
\end{align*} 
for all $i=1,\ldots n$. Applying the Hochschild differential we get:
\[ (\delta (\tau_0)) (x_3,x_2,x_1) = (-1)^{|x_1|+|x_2|} x_3 \tau_0(x_2,x_1) - \tau_0(x_3,x_2)x_1 + (-1)^{|x_1|} \tau_0(x_3, x_2 x_1) - (-1)^{|x_1|+|x_2|} \tau_0(x_3x_2,x_1) \]
It is straightforward (if tedious) to check that this expression vanishes identically on $\bar{A}^{\otimes 3}$. On the other hand, $\tau_0$ cannot be a coboundary, since any $\kappa \in CC^1(A,A[-2])$ has to be of the form:
\[ \kappa(s_i) = m_i e_i , \text{\ for some\ } m_i \in \f{K} \] 
and the Hochschild differential takes the form:
\[ (-1)^{|x_1|} (\delta( \kappa)) (x_2,x_1) = x_2 \kappa(x_1) + \kappa(x_2 x_1) - (-1)^{|x_1|} \kappa(x_2)x_1  \]
which gives, in particular, that $\delta(\kappa)(s_i,s_i) = 0$ and $\delta(\kappa)(a_i,s_i) = m_i a_i$.

Hence, $\tau_0$ cannot be of the form $\delta(\kappa)$ and therefore it represents a non-trivial element of the group $\bigoplus_s HH^2(A, A[s])$. 
But we know that this group is $1$-dimensional over any field $\f{K}$, consequently any generator of this group over arbitrary field $\f{K}$ must have the same $s$ grading as $\tau_0$.

It is harder to find explicit cocycles representing the elements $u_{\text{\scriptsize{-}}1}$ and $t_{\text{\scriptsize{-}}2}$ given in the statement of the theorem. 
Fortunately, for the purpose of determining the $s$ gradings we do not need explicit cocycles for these. 

The element $u_{\text{\scriptsize{-}}1}$ appears only if $\mathrm{char} \f{K}  \mid n+1$, and it satisfies the equation: 
\[ s_i u_{\text{\scriptsize{-}}1} = t_1 t_0^{n-1} \]
Since the $s$ gradings of $s_i$, $t_1$ and $t_0$ are $2,0$ and $-2$, respectively, it follows that the projection $u'_{\text{\scriptsize{-}}1}$ of $u_{\text{\scriptsize{-}}1}$ to $HH^{2n-1}(A, A[-2n])$ must be nonzero.
\emph{A priori} $u_{\text{\scriptsize{-}}1}$ is not necessarily homogeneous with respect to the $s$ grading, but it has $r$ grading $2n-1$, and $\bigoplus_s HH^{2n-1}(A, A[s])$ is 2-dimensional with generators $u_{\text{\scriptsize{-}}1}$ and $t_1 t_0^{n-1}$. 
Therefore, $u_{\text{\scriptsize{-}}1}$ has a decomposition $u'_{\text{\scriptsize{-}}1} + \lambda t_1 t_0^{n-1}$ into $(r,s)$-homogeneous elements  for some $\lambda \in \f{K}$.
On the other hand, the relations in the statement of the theorem which involve $u_{\text{\scriptsize{-}}1}$ are satisfied by $u_{\text{\scriptsize{-}}1}$ if and only if they are satisfied by $u'_{\text{\scriptsize{-}}1} =u_{\text{\scriptsize{-}}1} - \lambda t_1 t_0^{n-1}$. 
Therefore, we may freely replace $u_{\text{\scriptsize{-}}1}$ by $u'_{\text{\scriptsize{-}}1}$ and hence assume that it is homogeneous with $s$ grading $-2n$.

Similarly, if $\mathrm{char} \f{K}  \mid  n+1$, then $t_{\text{\scriptsize{-}}2} \in \bigoplus_s HH^{2n}(A, A[s])$ appears in the relation \[
s_i t_{\text{\scriptsize{-}}2} =  t_0^{n} \]
and $\bigoplus_s HH^{2n}(A, A[s])$ is 2-dimensional with generators $t_{\text{\scriptsize{-}}2} $ and $t_0^n$. 
As a consequence, $t_{\text{\scriptsize{-}}2}$ has a decomposition $t_{\text{\scriptsize{-}}2} = t'_{\text{\scriptsize{-}}2} + \lambda t_0^n$ into $(r,s)$-homogeneous elements for some $\lambda \in \f{K}$ and  $t'_{\text{\scriptsize{-}}2} \neq 0$. 
The argument we used for $u_{\text{\scriptsize{-}}1}$ applies here as well and we may assume that $t_{\text{\scriptsize{-}}2}$ is homogeneous with $s$ grading $-2n-2$.

Finally, we need to determine the $s$ grading of $t_{\text{\scriptsize{-}}2}$ over a field $\f{K}$ for which $\mathrm{char} \f{K}  \nmid n+1$. Since $A$ can be defined over $\f{Z}$, its Hochschild cohomology groups can also be defined over $\f{Z}$. Furthermore, since $A$ has finite rank as a $\f{Z}$-module, the bar-complex over $\f{Z}$ is just a chain complex of finitely generated free abelian groups. So we can apply the universal coefficient theorem 
\begin{equation}\label{uct}
0 \to \bigoplus_s HH_{\f{Z}}^{r}(A, A[s]) \otimes \f{K} \to \bigoplus_s HH_{\f{K}}^{r}(A \otimes \f{K}, A[s] \otimes \f{K} ) \to \mathrm{Tor} \left(\bigoplus_s HH_{\f{Z}}^{r+1}(A, A[s]), \f{K}\right) \to 0 
\end{equation}
Now, it follows from the presentation given in the statement that the middle group for $r=2n+1$ has rank 1 for any field $\f{K}$ and we know that it is supported in internal degree $s=-2n-2$ if the field $\mathrm{char} \f{K} \mid n+1$. Therefore, we deduce from the universal coefficient theorem (by testing $\f{K}=\f{F}_p$ for infinitely many primes $p$) that: 
\[ \bigoplus_s HH_\f{Z}^{2n+1}(A,A[s])  = \f{Z}[2n+2] \]
hence, in particular 
\[ \bigoplus_s HH_\f{K}^{2n+1}(A,A[s])  = \f{K}[2n+2] \]
Finally, observe that the element
\[ t_1 t_{\text{\scriptsize{-}}2} \in \bigoplus_{s} HH^{2n+1}(A,A[s]) = \f{K}[2n+2]\]
 is a generator of the Hochschild cohomology group in grading $r=2n+1$ over an arbitrary field $\f{K}$, and hence $t_{\text{\scriptsize{-}}2}$ must have $s$ grading $-2n-2$ over arbitrary field $\f{K}$. 
\end{proof}

\begin{remark} Over the finite field $\f{F}_3$ of characteristic $3$, the group
	algebra $\f{F}_3 \mathfrak{S}_3$ of the symmetric group in 3 letters
	is isomorphic to the algebra $A_\Gamma$ for $\Gamma=A_2$. A
	presentation for the Hochschild cohomology ring of this group algebra
	was given in \cite[Thm. 7.1] {siewither}. This agrees with the presentation given above.
\end{remark}

As a consequence of Thm.~(\ref{hha}) we conclude that the group $\bigoplus_{r+s=*} HH^r(A_\Gamma,A_\Gamma[s])$ is nontrivial if and only if $*\leq 2$. 
If $\mathrm{char} \f{K}  \nmid n+1$, the rank is $n$ at each $*\leq 2$, otherwise the rank is $n$ for $*=2,1$ and $n+1$ for $*\leq 0$.

Recall that we have proved in Thm.~(\ref{voila}) that there is an isomorphism of Gerstenhaber algebras:
$$SH^*(X_\Gamma) \cong HH^*(A_\Gamma)$$
over a field $\f{K}$ of characteristic zero, where the Conley-Zehnder grading on the left corresponds to the total grading $r+s$ on the right. Having computed $HH^*(A_\Gamma)$ as a bigraded algebra, we immediately get a description of the algebra structure of symplectic cohomology. Let us also record its rank.

\begin{corollary} \label{sha}
The symplectic cohomology group $SH^{*}(X_\Gamma)$ over a field $\f{K}$ of characteristic zero is of rank $n$ if $* \leq 2$ and it is trivial otherwise.
\end{corollary} 

We have also performed computer-aided checks on our calculations. Tables~(\ref{a2}), and ~(\ref{a3}) list the ranks (of a finite portion) for the cases $A_2$ and $A_3$.

\begin{table}[htb!]
\centering
\begin{tikzpicture}
\matrix (mymatrix) [matrix of nodes, nodes in empty cells, text height=1.5ex, text width=2.5ex, align=right]
{
\begin{scope} \tikz\node[overlay] at (-2.2ex,-0.6ex){\footnotesize r+s};\tikz\node[overlay] at (-1ex,0.8ex){\footnotesize s}; \end{scope} 
   & 2 & 1 & 0 & $-1$ & $-2$ & $-3$ & $-4$ & $-5$ & $-6$ & $-7$ & $-8$  \\ 
 2 & 2 & 0 & 0 &  0 &  0 &  0 &  0 &  0 &  0 &  0 &  0  \\ 
 1 & 0 & 0 & 1 &  0 &  1 &  0 &  0 &  0 &  0 &  0 &  0  \\
 0 & 0 & 0 & 1 &  0 &  1 &  0 &  x &  0 &  0 &  0 &  0  \\
-1 & 0 & 0 & 0 &  0 &  0 &  0 &  x &  0 &  1 &  0 &  1  \\
-2 & 0 & 0 & 0 &  0 &  0 &  0 &  0 &  0 &  1 &  0 &  1  \\
};
\draw (mymatrix-1-1.south west) ++ (-0.2cm,0) -- (mymatrix-1-12.south east);
\draw (mymatrix-1-2.north west) -- (mymatrix-6-2.south west);
\draw (mymatrix-1-1.north west) -- (mymatrix-1-1.south east);%n k diagonal line
\end{tikzpicture}
\caption{$\Gamma= A_2$.  $x$ is 1 if $\text{char} \f{K} = 3$, $0$ otherwise.} 
\label{a2}
\end{table}

\begin{table}[htb!]
\centering
\begin{tikzpicture}
\matrix (mymatrix) [matrix of nodes, nodes in empty cells, text height=1.5ex, text width=2.5ex, align=right]
{
\begin{scope} \tikz\node[overlay] at (-2.2ex,-0.6ex){\footnotesize r+s};\tikz\node[overlay] at (-1ex,0.8ex){\footnotesize s}; \end{scope} 
   & 2 & 1 & 0 & $-1$ & $-2$ & $-3$ & $-4$ & $-5$ & $-6$ & $-7$ & $-8$  \\ 
 2 & 3 & 0 & 0 &  0 &  0 &  0 &  0 &  0 &  0 &  0 &  0  \\ 
 1 & 0 & 0 & 1 &  0 &  1 &  0 &  1 &  0 &  0 &  0 &  0  \\
 0 & 0 & 0 & 1 &  0 &  1 &  0 &  1 &  0 &  x &  0 &  0  \\
-1 & 0 & 0 & 0 &  0 &  0 &  0 &  0 &  0 &  x &  0 &  1  \\
-2 & 0 & 0 & 0 &  0 &  0 &  0 &  0 &  0 &  0 &  0 &  1  \\
};
\draw (mymatrix-1-1.south west) ++ (-0.2cm,0) -- (mymatrix-1-12.south east);
\draw (mymatrix-1-2.north west) -- (mymatrix-6-2.south west);
\draw (mymatrix-1-1.north west) -- (mymatrix-1-1.south east);%n k diagonal line
\end{tikzpicture}
\caption{$\Gamma= A_3$.  $x$ is 1 if $\text{char} \f{K} = 2$, $0$ otherwise.} 
\label{a3}
\end{table}

%\newpage

\subsubsection{Type $D$} 
\label{typeD}

In this section we consider the case where $\Gamma$ is the Dynkin tree $D_n$, $n\geq 4$. Most of the arguments in the previous section apply verbatim or with minor modifications. So we will focus on the differences and provide details as necessary.

Considering the quiver based on $\Gamma$ with the orientation of the arrows given by Fig.~(\ref{D_n quiver}), we obtain the following result.
\begin{figure}[htb!]
\centering

\begin{tikzpicture}
	\tikzset{vertex/.style = {style=circle,draw, fill,  minimum size = 2pt,inner sep=1pt}}
\tikzset{edge/.style = {->,>=stealth',shorten >=8pt, shorten <=8pt  }}

% vertices
\node[vertex] (a) at  (0,0) {};
\node[vertex] (a1) at (1.5,0) {};
\node[vertex] (a2) at (3,0) {};
\node[vertex] (c) at  (7,0) {};
\node[vertex] (c1) at  (8.25,0.5) {};
\node[vertex] (c2) at  (8.25,-0.5) {};

% edges

\draw[edge] (a)  to (a1);
\draw[edge] (a1) to (a2);

\path (a2) to node {\dots} (c);
\node [shape=circle,minimum size=2pt, inner sep=1pt] (a3) at (4.5,0) {};
\draw[edge] (a2) to (a3);

\node [shape=circle,minimum size=2pt, inner sep=1pt] (c3) at (5.5,0) {};
\draw[edge] (c3) to (c);

\draw[edge] (c)  to (c1);
\draw[edge] (c) to (c2);

\node[] at (0.7,0.1) {\tiny{$a_1$}};
\node[] at (2.2,0.1) {\tiny{$a_2$}};
\node[] at (3.7,0.1) {\tiny{$a_3$}};
\node[] at (6.2,0.1) {\tiny{$a_{n-3}$}};

\node[] at (7.5,0.4) {\tiny{$a_{n-2}$}};
\node[] at (7.5,-0.45) {\tiny{$a_{n-1}$}};

\end{tikzpicture}

\caption{$D_n$ quiver}
\label{D_n quiver}
\end{figure}
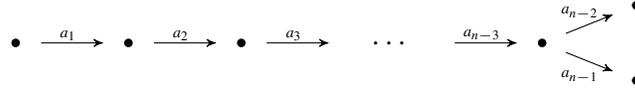

\begin{lemma}
The trivial extension algebra $B_\Gamma$ of the path algebra $\f{K} \Gamma$ is isomorphic to the quotient $\f{K} Q/ I$, where $Q$ is the quiver given in Fig.~(\ref{Q}) and $I$ is the ideal generated by the elements
\begin{align*}
& \beta_{n-1}\gamma_{n-1} - \beta_n \gamma_n \\
& \alpha_i \cdots \alpha_1 \beta_n \gamma_n \alpha_{n-3} \cdots \alpha_i \\
& \gamma_n \alpha_{n-3} \cdots \alpha_1 \beta_{n-1} \\
& \gamma_{n-1} \alpha_{n-3} \cdots \alpha_1 \beta_n
\end{align*}
\end{lemma}

\begin{figure}[htb!]
\centering

\begin{tikzpicture}
	\tikzset{vertex/.style = {style=circle,draw, fill,  minimum size = 2pt,inner sep=1pt}}
\tikzset{edge/.style = {->,>=stealth',shorten >=8pt, shorten <=8pt  }}

% vertices
\node[vertex] (a) at  (0,0) {};
\node[vertex] (a1) at (1.5,0) {};
\node[vertex] (a2) at (3,0) {};
\node[vertex] (c) at  (7,0) {};
\node[vertex] (c1) at  (8.25,1.5) {};
\node[vertex] (c2) at  (8.25,-1.5) {};

% edges

\draw[edge] (a)  to (a1);
\draw[edge] (a1) to (a2);

\path (a2) to node {\dots} (c);
\node [shape=circle,minimum size=2pt, inner sep=1pt] (a3) at (4.5,0) {};
\draw[edge] (a2) to (a3);

\node [shape=circle,minimum size=2pt, inner sep=1pt] (c3) at (5.5,0) {};
\draw[edge] (c3) to (c);

\draw[edge] (c)  to (c1);
\draw[edge] (c) to (c2);

\node [shape=circle,minimum size=2pt, inner sep=1pt] (a0) at (-0.05,0.07) {};
\draw[edge] (c1)  to (a0);
\node [shape=circle,minimum size=2pt, inner sep=1pt] (a0) at (-0.05,-0.07) {};
\draw[edge] (c2)  to (a0);

\node[] at (0.8,0.1) {\tiny{$\alpha_1$}};
\node[] at (2.2,0.1) {\tiny{$\alpha_2$}};
\node[] at (3.7,0.1) {\tiny{$\alpha_3$}};
\node[] at (6.2,0.1) {\tiny{$\alpha_{n-3}$}};

\node[] at (4.5,1.1) {\tiny{$\beta_{n-1}$}};
\node[] at (4.5,-1.1) {\tiny{$\beta_n$}};

\node[] at (7.8,0.32) {\tiny{$\gamma_{n-1}$}};
\node[] at (7.65,-0.38) {\tiny{$\gamma_n$}};

\end{tikzpicture}

\caption{The quiver $Q$}
\label{Q}
\end{figure}
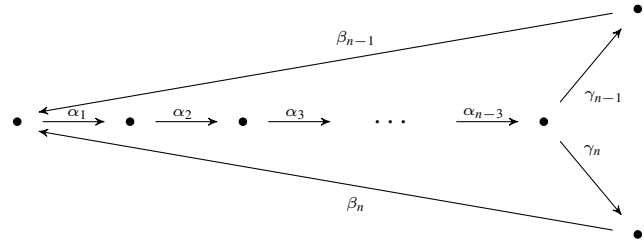

\begin{proof}
Using the identifications $a_i \leftrightarrow \alpha_i$ for $1\leq i \leq n-3$ and $a_j \leftrightarrow \gamma_{j+1}$ for $j=n-2$ and $n-1$ we can consider $\f{K} \Gamma$ as a subalgebra of $\f{K} Q/ I$. 
Observe that $\f{K} Q/ I$ decomposes as a direct sum $\f{K} \Gamma \oplus V $ and $V$ is generated by $\beta_{n-1}$ and $\beta_n$ as a $\f{K} \Gamma$-bimodule. 
Moreover, as $\f{K} \Gamma$-bimodules, $V$ and the dual of $\f{K} \Gamma$ are isomorphic via
\begin{align*}
\psi : V & \to (\f{K} \Gamma)^{\vee} \\
\beta_{n-1} &\mapsto (a_{n-2} a_{n-3} \cdots a_2 a_1 ) ^{\vee} \\
\beta_n &\mapsto  (a_{n-1} a_{n-3} \cdots a_2 a_1 ) ^{\vee}
\end{align*}

It is straightforward to check that this map is a well-defined isomorphism.

In fact, $(\f{K} \Gamma)^{\vee}$ can also be considered as a subalgebra of $\f{K} Q/ I$ by identifying the dual $p^{\vee}$ of a path $p \in \f{K} \Gamma$ with the path $q \in \f{K} Q/ I$ such that 
$$q \cdot p = \tau^t (\beta_n\gamma_n\alpha_{n-3} \cdots \alpha_1) = \tau^t (\beta_{n-1}\gamma_{n-1}\alpha_{n-3} \cdots \alpha_1) \in \f{K} Q/ I \ , $$
where $\tau$ denotes the simple rotation action on the cycles and $t$ is the distance between the initial points of $p$ and $\alpha_1$.  
\end{proof}
 
As a consequence of this lemma and the discussions in the previous section, there is an
isomorphism between the Hochschild cohomology rings of the zigzag algebra $A_\Gamma$ and $B_\Gamma$.
On the other hand, the Hochschild cohomology of $B_\Gamma$ as a trivially graded algebra was
described in detail in \cite{VG, volkov}. As in the case of $\Gamma = A_n$ (see Thm.~(\ref{hha})),
we determine the internal grading $s$ induced by the zigzag algebra and obtain the following result.
This extra information does not appear in \cite{VG, volkov} and the determination of this grading is the main contribution given in the following theorem. 

\begin{theorem}\label{hhd}
Let $\Gamma =D_n$, $n\geq 4$. The $(r,s)$-bigraded Hochschild cohomology algebra 
$$ HH^*(A_\Gamma) = \bigoplus_{r+s=*} HH^r(A_\Gamma, A_\Gamma [s]) $$
of the graded $\mathrm{k}$-algebra $A_\Gamma$ is (graded) commutative and given by the following
    generators and relations. (The subscripts of the generators, except for the $s_i$, refer to total
    degrees.)  

\begin{enumerate}

	\item Suppose $\mathrm{char} \f{K}\neq2$.
We have generators labeled along with their bidegrees $(r,s)$ 
\begin{align*} 
			s_1, \dots, s_n &\text{\ \ \ \ \ } (0, 2) \\
			t_1 &\text{\ \ \ \ \ } (1, 0) \\
			r_1 &\text{\ \ \ \ \ } (2n-3,-2n+4) \\
			t_0 &\text{\ \ \ \ \ } (4, -4) \\
			r_0 &\text{\ \ \ \ \ } (2n-4, -2n+4 ) \\
			t_{\text{\scriptsize{-}}2} &\text{\ \ \ \ \ } (4n-6, -4n+4) \\  
\end{align*}
and relations 
$$s_i s_j=s_i t_j= s_ir_j=t_1^2=t_1r_1=r_1^2=t_0^{n-1}=0$$
\vskip-0.5ex
\begin{center} \begin{tabular}{r|c|c} 
& (if $n$ is even) & (if $n$ is odd) \\ [0.5ex]
$t_1r_0=$&$\left(\frac{n}{2}\right) t_1t_0^{(n-2)/2} -(n-1) r_1$ & $\left(\frac{n-1}{2}\right) r_1 $\\
$2t_0r_1 =$&$ t_1t_0^{n/2}$ & $ 0$\\
$2r_1r_0=$&$0 $&$t_1t_0^{n-2}$\\
$2t_0r_0 = $&$t_0^{n/2}$ & $0$ \\
$2r_0^2=$&$ \left(\frac{n}{2}\right) t_0^{n-2}$ & $\left(\frac{n-1}{2}\right) t_0^{n-2}$
\end{tabular} \end{center}
\vskip1cm

\item  Suppose $\mathrm{char} \f{K}= 2$. 
We have generators labeled along with their bidegrees $(r,s)$ 
\begin{align*} s_1, \dots, s_n &\text{\ \ \ \ \ } (0, 2) \\
			t_1 &\text{\ \ \ \ \ } (1, 0) \\
			u_1 &\text{\ \ \ \ \ } (3,-2) \\
			t_0 &\text{\ \ \ \ \ } (4, -4) \\
			r_0 &\text{\ \ \ \ \ } (2n-4, -2n+4 ) \\
			u_0 &\text{\ \ \ \ \ } \left(4 \left \lfloor \frac{n}{2} \right \rfloor  , - 4 \left \lfloor \frac{n}{2} \right \rfloor  \right) \\
			u_{\text{\scriptsize{-}}1} &\text{\ \ \ \ \ } \left(4 \left \lfloor \frac{n-1}{2} \right \rfloor  +1, - 4 \left \lfloor \frac{n-1}{2} \right \rfloor -2 \right) \\
			t_{\text{\scriptsize{-}}2} &\text{\ \ \ \ \ } (4n-6, -4n+4 ) 
\end{align*}
and relations 
$$ s_i s_j=s_i t_1= s_i u_1 = s_iu_0=0$$
$$ t_1^2= u_1^2= u_0^2=u_1u_0=0$$
$$t_0^{\left \lfloor \frac{n}{2} \right \rfloor}= u_1t_0^{\left \lfloor \frac{n-1}{2} \right \rfloor}=0$$
$$r_0^2=\left \lfloor \frac{n}{2} \right \rfloor u_0t_0^{\left \lfloor \frac{n-3}{2} \right \rfloor} $$
$$ s_jt_0= t_1u_1  $$ 
\vskip-0.5ex
\begin{center} \begin{tabular}{r|c|c} 
& (if $n$ is even) & (if $n$ is odd) \\ [0.5ex]

$u_{\text{\scriptsize{-}}1}^2 =$&$ t_{\text{\scriptsize{-}}2}$ & $ t_{\text{\scriptsize{-}}2}t_0$\\
$u_1u_{\text{\scriptsize{-}}1}=$&$ u_0$ & $u_0t_0$\\
$t_0r_0=$&$ u_1u_{\text{\scriptsize{-}}1}$ & $t_1u_{\text{\scriptsize{-}}1}$\\
$u_1r_0=$&$0$ & $t_1u_0 $\\
$s_ju_{\text{\scriptsize{-}}1}=$&$\begin{cases} \ \ \left(\frac{n-2}{2}\right) t_1t_0^{(n-2)/2} +t_1r_0  &, \ \text{if} \ \ j \leq n-1 , \\  \ \ \ \ \ \ \left(\frac{n}{2}\right) t_1t_0^{(n-2)/2} +t_1r_0 &, \ \text{if} \ \ j =  n  \end{cases}$ & $ u_1r_0 $\\
$s_jr_0=$&$\begin{cases} \ \ t_1u_1t_0^{ (n-4)/2} &, \ \text{if} \ \ j \leq n-1 , \\ \ \ \  \ \ \ 0 &, \ \text{if} \ \ j =  n  \end{cases}$ &$ 0 $ \\
$u_{\text{\scriptsize{-}}1}r_0 = $& & $t_1t_{\text{\scriptsize{-}}2} $ \\
$t_1r_0 = $& & $\left(\frac{n-1}{2}\right) u_1t_0^{(n-3)/2} $ \\
$s_jt_{\text{\scriptsize{-}}2}=$& & $r_0u_0 $
\end{tabular} \end{center}

\end{enumerate}

\end{theorem}

\begin{proof}
The presentation of the algebra structure of $HH^*(B_\Gamma)$ in \cite[Thm. 4]{volkov} provides all the generators with their $r$ gradings and relations.  The derived equivalence between $A_\Gamma$ and $B_\Gamma$ gives
$$ HH^r(B_\Gamma) \cong \bigoplus_s HH^r(A_\Gamma, A_\Gamma [s]) \ . $$
Therefore it suffices to determine the $s$ gradings of the generators in the statement.
Extending the notation in Fig.~(\ref{D_n quiver}), we consider the decomposition of the graded algebra $A_\Gamma$ into homogeneous $\f{K}$-subspaces $A_0$, $A_1$, and $A_2$ spanned by 
$$\{ e_1, \dots , e_n \}, \{ a_1 , b_1, \dots , a_{n-1} , b_{n-1}  \}, \ \text{and} \  \{ s_1 , \dots , s_n\}, $$ 
respectively, where 
$$e_{i+1} a_i e_i = a_i, \ e_i b_i e_{i+1} = b_i, e_n a_{n-1} e_{n-2} = a_{n-1} , \ e_{n-2} b_{n-1}  e_n = b_{n-1} , $$
$$s_1=b_1a_1, \ s_{i+1} = a_{i} b_{i} = b_{i+1} a_{i+1}, \ s_{n-2} = a_{n-3} b_{n-3} = b_j a_j, \ \text{and} \ s_{j+1} = a_j b_j , $$ 
for $1\leq i \leq n-4$ and $j=n-2, n-1$.  

As in the proof of Thm.~(\ref{hha}), we will again use the reduced bar-resolution associated to $A=A_\Gamma$ and denote the Hochschild differential by $\delta$. 
Consequently, the discussion for $r=0,1$ is exactly the same as in the proof of Thm.~(\ref{hha}). 
We identify the $s$ gradings of $s_1 , \dots , s_n$ and $t_1$ as in the statement. 

For every nonnegative integer $r$, the dimension of $\bigoplus_s HH^r(A, A [s]) \cong  HH^r (B_\Gamma)$ can be deduced from the presentation in the statement and it is explicitly given in \cite[Thm. 3]{volkov}. 
We will make extensive use of this information.
To begin with, note that $\bigoplus_s HH^2(A, A [s])$ is trivial over any field $\f{K}$, and $\bigoplus_s HH^3(A, A [s])$ is 1-dimensional  if $\mathrm{char} \f{K} =2$ and trivial otherwise. 
Over a field $\f{K}$ of characteristic $2$, for $c \in CC^3(A, A[s])$, the Hochschild differential $\delta$ is given by:
$$ \delta(c) (x_4, x_3, x_2, x_1) = x_4 c(  x_3, x_2, x_1) + c (x_4, x_3, x_2 ) x_1 + c (x_4 x_3, x_2, x_1) + c (x_4, x_3 x_2, x_1) + c (x_4, x_3, x_2 x_1). $$

We claim that, if $\mathrm{char}\f{K}=2$, there is a cocycle $\upsilon_1 \in CC^3(A, A[-2])$ which is not the coboundary of any $\kappa \in  CC^2(A, A[s])$. 
This and the fact that $\bigoplus_s HH^3(A, A [s])$ is 1-dimensional imply that the $s$ grading of $u_1$ must be $-2$, same as $\upsilon_1$. 
To describe the graded homomorphism $\upsilon_1 : \bar{A}^{\otimes 3} \to A[-2]$ uniquely, it suffices to list the generators of $\bar{A}^{\otimes 3}$ on which $\upsilon_1$ is nonzero. 
It necessarily vanishes on any element of degree $5$ or $6$ in $\bar{A}^{\otimes 3}$ since $A$ is supported in gradings between $0$ and $2$. 
We declare $\upsilon_1$ to be nonzero exactly on those nontrivial elements $(x_3, x_2 , x_1) \in \bar{A}^{\otimes 3}$ which satisfy one of the following conditions:
\begin{itemize}
\item One of $x_1$, $x_2$ and $x_3$ is of the form $a_i$ and the other two is of the form $b_i$, possibly with different indices, and  $(x_3, x_2 , x_1) \neq (b_{n-1}, a_{n-1} ,b_{n-2})$. 
\item Exactly one of $x_1$, $x_2$ and $x_3$ is of the form $s_k$, and the initial point of $x_1$ matches the terminal point of $x_3$.
\item $(x_3, x_2 , x_1) = (a_{n-2} , b_{n-1} , a_{n-1})$. 
\end{itemize}

It is straightforward to check that $\upsilon_1$ is a cocycle. To see that it is not a coboundary, suppose that $c \in CC^2(A, A[-2])$. Then 
$$ \delta (\kappa)  ((b_2, a_2, s_2)+(a_2,s_2,b_2)+(s_2,b_2,a_2)) = b_2 \kappa(a_2,s_2) + a_2\kappa(s_2,b_2)+\kappa(a_2,s_2)b_2+\kappa(s_2,b_2)a_2 $$
after cancellations.
Observe that the right-hand side is either $s_2+s_3$ or $0$, depending on the values of $\kappa(a_2,s_2)$ and $\kappa(s_2,b_2)$. Since \[ \upsilon_1 ((b_2, a_2, s_2)+(a_2,s_2,b_2)+(s_2,b_2,a_2)) = s_3, \] $\upsilon_1$ cannot be a coboundary. 

Next we determine the $s$ grading of $t_0$. 
Consider the case  $\mathrm{char}\f{K} = 2$.
If $n=4$, then $\bigoplus_s HH^4(A, A[s])$ has generators $t_0, r_0$ and $t_1u_1$. 
Note that any relation satisfied by $t_0$ and $r_0$ is also satisfied by $t_0-\gamma t_1 u_1$ and $r_0 -\gamma t_1 u_1$, respectively, for any $\gamma \in \f{K}$. 
Therefore, without loss of generality, we may assume that there are $s$-homogeneous generators $t'_0$, $r'_0$ and constants $\alpha, \beta \in \f{K}$ such that
$$ t_0 = t'_0 + \alpha r'_0 \ \mbox{ and } \ r_0 = r'_0 + \beta t'_0 \ . $$
From the relations regarding $s_nr_0$ and $s_n t_0$ we obtain
$$ 0 = s_n r_0 = s_n r'_0 + \beta s_n t'_0 \ \mbox{ and } \ 0 \neq u_1t_1 = s_n t_0 = s_nt'_0 + \alpha s_n r'_0$$
Since the gradings of $u_1, t_1$ and $s_n$ are established above, the second equation implies that at least one of $t'_0$ and $r'_0$ has $s$ grading $-4$; in fact they both do, as the following arguments show.
If $s_nr'_0 \neq 0$, then the first equation proves that $r'_0$ and $t'_0$ have the same $s$ grading which is necessarily $-4$. 
So suppose $s_nr'_0 = 0$. 
Now the second equation gives $s_nt'_0 \neq 0$. 
Moreover, the first equation implies $\beta=0$ which means $r_0=r'_0$; in particular $r_0$ is $s$-homogeneous. 
So we can use the relation $s_1 r_0 = t_1u_1$ to establish the $s$ grading of $r'_0$ as $-4$ .
On the other hand, under the assumption $s_nr'_0 = 0$, the second equation becomes $s_nt'_0 = u_1t_1$ implying that $t'_0$ has $s$ grading $-4$ as well.
Therefore, regardless of the value of $s_nr'_0$, $s$ gradings of $t_0$ and $r_0$ are both $-4$.

If $n>4$ and $\mathrm{char} \f{K} =2$, then $\bigoplus_s HH^4(A, A[s])$ has rank $2$ with generators $t_0$ and $t_1u_1$, hence we may assume that there is an $s$-homogeneous generator $t'_0$ and $\alpha \in \f{K}$ such that $ t_0 = t'_0 + \alpha t_1u_1$. 
The relation $s_n t_0 = t_1 u_1$ implies that the $s$ grading of $t'_0$ is $-4$. The $s$ grading of $t_1u_1$ is $-2$ by previous computations. 
If $n$ is even, then any relation in the statement holds for $t_0$ if and only if it holds for $t'_0$. 
Therefore, without loss of generality, we may assume that $t_0=t'_0$ is $s$-homogeneous with grading $-4$, at least when $n$ is even.
The same conclusion holds for odd $n$ as well, but we will not prove (nor use) it until Case {\bf 3} below.  

Let us now consider the $s$ grading of $t_0$ when $\mathrm{char}\f{K} \neq 2$. 
Regardless of whether $n=4$ or not, the argument uses the universal coefficient theorem (\ref{uct}) as in the proof of Thm.~(\ref{hha}).
First of all, considering that $\bigoplus_s HH^2(A, A[s])$ is trivial for any field $\f{K}$ and using (\ref{uct}) for $r=2$, we conclude that $ \bigoplus_s HH_{\f{Z}}^3 (A, A[s])$ has no torsion. 
Since $\bigoplus_s HH^3(A, A[s])$ is trivial when $\mathrm{char} \f{K}  \neq 2$, applying the universal coefficient theorem (\ref{uct}) for $r=3$ implies that $ \bigoplus_s HH_{\f{Z}}^3 (A, A[s])$ has no free component either, hence it is trivial. 
Moreover, the same exact sequence and the fact that for $\mathrm{char} \f{K} =2$, $\bigoplus_s HH^3(A, A[s])$ is generated by $u_1$ whose $s$ grading is computed as $-2$ above, establish the torsion of $ \bigoplus_s HH_{\f{Z}}^4 (A, A[s])$ as $\f{Z}_2 [2]$. 

The argument above for $\mathrm{char} \f{K} =2$ shows that $ \bigoplus_s HH^4 (A, A[s]) \cong \f{K}^d[4] \oplus \f{K}[2]$, where $d=2$ if $n=4$ and $d=1$ otherwise.
Using the fact that $ \bigoplus_s HH^4 (A, A[s])$ is $d$-dimensional for any field $\f{K}$ with $\mathrm{char} \f{K} \neq 2$, and applying the universal coefficient theorem (\ref{uct}) for $r=4$ to infinitely many characteristics, we conclude that $ \bigoplus_s HH_{\f{Z}}^4 (A, A[s])$ is in fact $\f{Z}^d[4] \oplus \f{Z}_2[2]$. 
In particular, $ \bigoplus_s HH^4 (A, A[s])$ is supported in $s$ grading $-4$ whenever $\mathrm{char}\f{K} \neq 2$, and the $s$ grading of $t_0$ is $-4$ unless $n$ is odd and $\mathrm{char}\f{K}=2$.

The rest of the argument varies slightly according to the parity of $n$ and the characteristic of the base field.

{{\bf Case 1} ($n$ even and $\mathrm{char} \f{K} = 2$)} 

We need to determine the $s$ gradings of the rest of the generators, namely $u_{\text{\scriptsize{-}}1}, t_{\text{\scriptsize{-}}2}, u_0$ and $r_0$. 
Since
$$ \{ u_{\text{\scriptsize{-}}1}, t_1r_0, t_1 t_0^{(n-2)/2}\} $$ 
forms a basis of $\bigoplus_s HH^{2n-3}(A, A[s])$, $$u_{\text{\scriptsize{-}}1} = u'_{\text{\scriptsize{-}}1} + \alpha t_1r_0 + \beta t_1 t_0^{(n-2)/2}$$ for some $s$-homogeneous $u'_{\text{\scriptsize{-}}1}\neq 0$ and some $\alpha, \beta \in \f{K}$. Observe that any relation satisfied by $u_{\text{\scriptsize{-}}1}$ is satisfied by $u'_{\text{\scriptsize{-}}1}$ as well. Therefore, without loss of generality, we may assume that $u_{\text{\scriptsize{-}}1}=u'_{\text{\scriptsize{-}}1}$ and its $s$ grading is $-2n+2$ as a result of the relation
$$s_n u_{\text{\scriptsize{-}}1} - s_1 u_{\text{\scriptsize{-}}1} = t_1t_0^{(n-2)/2} \ . $$
Moreover, by the relations $u_0 = u_1 u_{\text{\scriptsize{-}}1}$ and $t_{\text{\scriptsize{-}}2} = u_{\text{\scriptsize{-}}1}^2$, both $u_0$ and $t_{\text{\scriptsize{-}}2}$ are $s$-homogeneous with gradings $-2n$ and $-4n+4$, respectively. 
Regarding $r_0$, note that 
$$ \{ r_0, t_0^{(n-2)/2},  t_1u_1t_0^{(n-4)/2}\} $$ 
forms a basis of $\bigoplus_s HH^{2n-4}(A, A[s])$. Hence 
$$ r_0 =r'_0 +  \alpha t_0^{(n-2)/2} + \beta t_1u_1t_0^{(n-4)/2}$$ 
for some $s$-homogeneous $r'_0 \neq 0$ and some $\alpha, \beta \in \f{K}$. It is straightforward to check that any relation satisfied by $r_0$ is also satisfied by $r_0 - \beta t_1u_1t_0^{(n-4)/2}$, so we may assume that $ r_0 =r'_0 +  \alpha t_0^{(n-2)/2}$. Moreover, the relation $u_0= t_0r_0 = t_0r'_0$ implies that the $s$ grading of $r'_0$ is $-2n+4$, the same as that of  $t_0^{(n-2)/2}$. Therefore, $r_0$ is $s$-homogeneous with this grading as well. 

{{\bf Case 2 } ($n$ even and $\mathrm{char}\ \f{K} \neq 2)$} 

We have a single argument for the $s$ grading of $r_0$ and $r_1$ which belong to 2-dimensional spaces $\bigoplus_s HH^{2n-4}(A, A[s])$ and $\bigoplus_s HH^{2n-3}(A, A[s])$, respectively. 
We take $s$-homogeneous elements $r'_0\neq 0$ and $r'_1\neq 0$ such that 
$$ r_0 = r'_0 + \alpha t_0^{(n-2)/2} \ \mbox{ and } \ r_1=r'_1+\beta t_1 t_0^{(n-2)/2} \ . $$

Suppose that $\mathrm{char}\f{K} \nmid n-1$. 
By way of contradiction, assume that $r_0$ is not $s$-homogeneous, i.e. $\alpha \neq 0$ and the $s$-grading of $r'_0$ is not $-2n+4$. 
Then $t_1r_0=\left(\frac{n}{2}\right) t_1t_0^{(n-2)/2} -(n-1) r_1$ implies that $-(n-1)r'_1 = t_1r'_0$ for grading reasons. 
Consequently, the $s$-gradings of $r'_0$ and $r'_1$ should match. 
Moreover, since $2t_0r_1=t_1t_0^{n/2}$, and again for grading reasons, $\beta \neq 0$. 
But then,  $\alpha\beta t_1t_0^{n-2} \neq 0$ and its $s$ grading does not match with the $s$-grading of any other term in the product $r_1r_0$ contradicting with $r_1r_0=0$. 
Therefore $r_0$ is $s$-homogeneous, and so is $r_1$, in fact with the same $s$ grading, as a consequence of 
$$t_1r_0=\left(\frac{n}{2}\right) t_1t_0^{(n-2)/2} -(n-1) r_1 \ .$$
In order to account for the possibility that $\mathrm{char}\f{K} \mid \frac{n}{2}$, instead of the relation above we use the relation $2t_0r_0 = t_0^{n/2}$ to obtain the common $s$ grading of $r_0$ and $r_1$.

For a field $\f{K}$ with $\mathrm{char}\f{K} \neq 2$, both $\bigoplus_s HH^{2n-3}(A, A[s])$ and $\bigoplus_s HH^{2n-4}(A, A[s])$ are 2-dimensional, and moreover we just proved that when $\mathrm{char}\f{K} \nmid n-1$, each of these spaces are supported in $s=-2n+4$. 
By using the universal coefficient theorem (\ref{uct}) for $r=2n-4$ we conclude that, as long as $\mathrm{char}\f{K} \neq 2$ (even if $\mathrm{char}\f{K}$ divides 
$n-1$) both $\bigoplus_s HH^{2n-3}(A, A[s])$ and $\bigoplus_s HH^{2n-4}(A, A[s])$ are supported in $s=-2n+4$. 
In particular, the common $s$ grading of $r_0$ and $r_1$ is $-2n+4$.

The $s$ grading of the remaining generator $t_{\text{\scriptsize{-}}2}$ is obtained by the following argument which applies to odd $n$ as well. 
First of all, $t_{\text{\scriptsize{-}}2}$ is $s$-homogeneous as it belongs to the 1-dimensional space $\bigoplus_s HH^{4n-6}(A, A[s])$.
On the other hand, $\bigoplus_s HH^{4n-5}(A, A[s])$ is 1-dimensional over any field $\f{K}$ and it is generated by $t_1t_{\text{\scriptsize{-}}2}$. 
Since we already have the $s$ grading of $t_1t_{\text{\scriptsize{-}}2}$ for $\mathrm{char}\f{K} =2$ from the previous case, we obtain the $s$ grading of  $t_{\text{\scriptsize{-}}2}$ over any field using the universal coefficient theorem (\ref{uct}) for $r=4n-5$.

{\bf Case 3 }($n$ odd and $\mathrm{char}\ \f{K} =2$) 

In this case, the $s$ grading of $r_0$ can be obtained by an argument which works regardless of $\mathrm{char}\f{K}$.
Over any $\f{K}$, $\bigoplus_s HH^{2n-4}(A, A[s])$ is 1-dimensional and generated by $r_0$ which is therefore $s$-homogeneous. 
Applying the universal coefficient theorem (\ref{uct}) for $r=2n-4$ and infinitely many different characteristics, we conclude that $\bigoplus_s HH_{\f{Z}}^{2n-4}(A, A[s]) \cong \f{Z}$ and to  
establish the $s$ grading of this group, it suffices to use the relation $2r_0^2 = \left( \frac{n-1}{2} \right) t_0^{n-2}$ over a field of characteristic $0$. 
In particular, $r_0$ has $s$ grading $-2n+4$ for any field $\f{K}$.

The generator $u_0$ belongs to the 1-dimensional space $\bigoplus_s HH^{2n-2}(A, A[s])$, hence it is $s$-homogeneous, and its $s$ grading is determined by the relation $u_1r_0=t_1u_0$.

Next we consider $u_{\text{\scriptsize{-}}1}$. 
It belongs to $\bigoplus_s HH^{2n-1}(A, A[s])$ which is generated by $u_{\text{\scriptsize{-}}1}$ and $u_1r_0$. 
So $u_{\text{\scriptsize{-}}1}= u'_{\text{\scriptsize{-}}1} + \alpha u_1r_0$ for some $\alpha \in \f{K}$ and $s$-homogeneous $u'_{\text{\scriptsize{-}}1} \neq 0$. 
Observe that any relation which involves $u_{\text{\scriptsize{-}}1}$ is satisfied by $u'_{\text{\scriptsize{-}}1}$ as well.
Hence we may assume that $u_{\text{\scriptsize{-}}1}$ is $s$-homogeneous. 
Its $s$ grading is obtained from 
$$t_1u_{\text{\scriptsize{-}}1} = t_0r_0 = t'_0r_0 \ .$$
Note that we have not established the $s$ homogeneity of $t_0$ in this case yet, and that is why we had to refer to $t'_0$ in the relation above and use the fact that $t_0r_0 -t'_0r_0=0$ since  it is a multiple of $t_1u_1r_0=s_jt_0r_0=0$.

Finally, we determine the $s$ gradings of $t_0$ and $t_{\text{\scriptsize{-}}2}$ simultaneously. 
In the case we consider, they belong to 2-dimensional spaces 
\[ \bigoplus_s HH^{4}(A, A[s]) \text{\ \ \ and\ \ \ } \bigoplus_s HH^{4n-6}(A, A[s]),\] with respective bases $\{ t_0 , t_1u_1 \}$ and $\{ t_{\text{\scriptsize{-}}2} , r_0u_0 \}$.
So there are $s$-homogeneous elements $t'_0$ and $t'_{\text{\scriptsize{-}}2}$ with constants $\alpha, \beta \in \f{K}$ such that 
$$ t_0 = t'_0 + \alpha  t_1u_1 \ \mbox { and } \ t_{\text{\scriptsize{-}}2} = t'_{\text{\scriptsize{-}}2}+ \beta r_0 u_0 \ . $$
In fact, the $s$ gradings of $t'_0$ and $t'_{\text{\scriptsize{-}}2}$ are $-4$ and $-4n+4$, respectively since $s_jt'_0 = t_1u_1$ and $s_j t'_{\text{\scriptsize{-}}2} = r_0u_0$.
It is straightforward to check that any relation in the statement, except for $u^2_{\text{\scriptsize{-}}1} =  t_{\text{\scriptsize{-}}2} t_0$, holds for $t_0$ and $t_{\text{\scriptsize{-}}2}$ if and only if it holds for $t'_0$ and $t'_{\text{\scriptsize{-}}2}$. 
To check that the remaining relation holds, we use
$$ u^2_{\text{\scriptsize{-}}1} =  t_{\text{\scriptsize{-}}2} t_0 = t'_{\text{\scriptsize{-}}2} t'_0 + \alpha t'_{\text{\scriptsize{-}}2} t_1u_1 + \beta t'_0r_0u_0+\alpha\beta t_1u_1r_0u_0 $$
and observe that the only term on the right-hand-side of the above relation whose $s$ grading matches that of $ u^2_{\text{\scriptsize{-}}1} $ is $t'_{\text{\scriptsize{-}}2} t'_0 $.
Therefore, without loss of generality, we may assume that $ t_0 = t'_0 $ and $t_{\text{\scriptsize{-}}2} = t'_{\text{\scriptsize{-}}2}$.

{\bf Case 4} ($n$ odd and $\mathrm{char}\ \f{K} \neq 2$) 

The $s$ gradings of $t_{\text{\scriptsize{-}}2}$ and $r_0$ are already obtained in Cases 2 and 3 above. 

The remaining generator $r_1$ is $s$-homogeneous since it belongs to the 1-dimensional space $\bigoplus_s HH^{2n-3}(A, A[s])$ and its $s$ grading is determined by the relation $2r_1r_0=t_1t_0^{n-2}$.
\end{proof}

Using Thm.~(\ref{intrinsic}) which is due to Seidel and Thomas, one gets the following consequence of the computation above.

\begin{corollary}
If $\mathrm{char}\f{K} \neq 2$ and $\Gamma$ is of type $D_n$, $n\geq4$, then the zigzag algebra $A_\Gamma$ is intrinsically formal. 
\end{corollary}

One can write explicit bases for the relevant $\f{K}$-vector subspaces of $HH^*(A_\Gamma)$ as follows.

If $\mathrm{char}\f{K} \neq 2$, then $\bigoplus_{r+s=2} HH^r(A,A[s])$ is spanned by $\{ s_1, \dots , s_n\}$, and for any nonnegative integer $m$ and $i=0,1$ a basis of $\bigoplus_{r+s=i-2m} HH^r(A,A[s])$ is given by
$$ \{ r_it_{\text{\scriptsize{-}}2}^m, t_1^it_0^k t_{\text{\scriptsize{-}}2}^m\ : \ 0 \leq k \leq n-2\} \ . $$

When $\mathrm{char}\ \f{K} =2$, the increase in the dimensions of these spaces is immediate from the statement of Thm.~(\ref{hhd}). The subspace $\bigoplus_{r+s=2} HH^r(A,A[s])$ is spanned by 
$$\left\{ s_j, t_1u_1t_0^k = s_nt_0^{k+1}\ : \ 1\leq j \leq n, \ 0 \leq k \leq \left \lfloor \frac{n-4}{2} \right \rfloor \right\}  $$
and depending on the parity of $n$, $\bigoplus_{r+s=1} HH^r(A,A[s])$ is spanned by
$$\left\{ u_1t_0^k , t_1t_0^l, t_1r_0t_0^l \ : \ 0 \leq k \leq \frac{n-4}{2} , \ 0 \leq l \leq \frac{n-2}{2}  \right\} $$
if $n$ is even, and by
$$\left\{ u_1t_0^l , t_1t_0^l, t_1u_0t_0^l \ : \ 0 \leq l \leq  \frac{n-3}{2}  \right\}$$
if $n$ is odd.

If $n$ is even and $m$ is nonnegative, then a basis of $\bigoplus_{r+s=-m} HH^r(A,A[s])$ can be given as
$$\left\{  t_0^lu_{\text{\scriptsize{-}}1}^m ,r_0t_0^lu_{\text{\scriptsize{-}}1}^m, t_1 t_0^lu_{\text{\scriptsize{-}}1}^{m+1}, r_0t_1t_0^l u_{\text{\scriptsize{-}}1}^{m+1} \ :  \ 0 \leq l \leq \frac{n-2}{2} \right\} $$

If $n$ is odd and $m$ is nonnegative, then 
\[ \bigoplus_{r+s=-2m} HH^r(A,A[s]) \text{\ \ \ and\ \ \ } \bigoplus_{r+s=-2m-1} HH^r(A,A[s])\]are spanned by
$$\left\{  t_0^lt_{\text{\scriptsize{-}}2}^m,r_0t_0^lt_{\text{\scriptsize{-}}2}^m, u_0 t_0^lt_{\text{\scriptsize{-}}2}^{m}, u_0 r_0t_0^l t_{\text{\scriptsize{-}}2}^{m}\ :  \ 0 \leq l \leq  \frac{n-3}{2}  \right\}$$
and
$$\left\{  u_{\text{\scriptsize{-}}1}t_0^lt_{\text{\scriptsize{-}}2}^m,u_{\text{\scriptsize{-}}1}r_0t_0^lt_{\text{\scriptsize{-}}2}^m, u_{\text{\scriptsize{-}}1}u_0 t_0^lt_{\text{\scriptsize{-}}2}^{m}, u_{\text{\scriptsize{-}}1}u_0 r_0t_0^l t_{\text{\scriptsize{-}}2}^{m}\ :  \ 0 \leq l \leq  \frac{n-3}{2}  \right\}, $$
respectively. 

Therefore, the group $\bigoplus_{r+s=*} HH^r(A_\Gamma,A_\Gamma[s])$ is nontrivial if and only if $*\leq 2$. 
If the ground field has characteristic $2$, the rank is $n+\lfloor{\frac{n-2}{2}\rfloor}$ for $*=2,1$ and $4\lfloor{\frac{n}{2}\rfloor}$ for $*\leq 0$. Otherwise the rank is $n$ at each $*\leq 2$. Therefore, it follows from Thm.~(\ref{voila})
that we have:
\begin{corollary} \label{shd}
The symplectic cohomology group $SH^{*}(X_\Gamma)$ over a field of characteristic zero is of rank $n$ if $*\leq 2$ and it is trivial otherwise.
\end{corollary} 

As before, for convenient access, we give tables listing the ranks of a truncated piece of our calculation. As mentioned in Sec.~(\ref{sec-trivial}), $A_\Gamma$ has a graded periodic resolution as a graded bimodule, from which it follows easily that for $\Gamma = D_n$, $n\geq 4$, the ranks of the Hochschild cohomology groups obeys the following periodicity:
\[ \text{rank} HH^{r}(A, A[s]) = \text{rank} HH^{r+(4n-6)}(A , A[s-(4n-4)]) , \ \ \text{for} \ \ r >0 \] 
In this presentation, multiplication by the generator $t_{\text{\scriptsize{-}}2}$ gives rise to this periodicity. The tables below give the truncation which includes a fundamental domain of the period in the cases $\Gamma= D_4, D_5, D_6$.
We have also performed computer-aided checks in these cases.

\begin{table}[htb!]
        \centering
\begin{tikzpicture}
\matrix (mymatrix) [matrix of nodes, nodes in empty cells, text height=1.5ex, text width=2.5ex, align=right]
{
\begin{scope} \tikz\node[overlay] at (-2.2ex,-0.6ex){\footnotesize r+s};\tikz\node[overlay] at (-1ex,0.8ex){\footnotesize s}; \end{scope} 
   & 2 & 1 & 0 & $-1$ & $-2$ & $-3$ & $-4$ & $-5$ & $-6$ & $-7$ & $-8$ & $-9$ & $-10$ \\ 
 2 & 4 & 0 & 0 &  0 &  x &  0 &  0 &  0 &  0 &  0 &  0  &  0 &  0 \\ 
 1 & 0 & 0 & 1 &  0 &  x &  0 &  2 &  0 &  0 &  0 &  1  &  0 &  0 \\
 0 & 0 & 0 & 1 &  0 &  0 &  0 &  2 &  0 &  x &  0 &  1  &  0 &  2x  \\
-1 & 0 & 0 & 0 &  0 &  0 &  0 &  0 &  0 &  x &  0 &  0  &  0 &  2x \\
};
\draw (mymatrix-1-1.south west) ++ (-0.2cm,0) -- (mymatrix-1-14.south east);
\draw (mymatrix-1-2.north west) -- (mymatrix-5-2.south west);
\draw (mymatrix-1-1.north west) -- (mymatrix-1-1.south east);%n k diagonal line
\end{tikzpicture}
\caption{$\Gamma= D_4$.  $x$ is 1 if $\text{char} \f{K} = 2$, $0$ otherwise.} 
\label{d4table}
\end{table}

\begin{table}[htb!]
\centering
\begin{tikzpicture}
\matrix (mymatrix) [matrix of nodes, nodes in empty cells, text height=1.5ex, text width=2.8ex, align=right]
{
\begin{scope} \tikz\node[overlay] at (-2.2ex,-0.6ex){\footnotesize r+s};\tikz\node[overlay] at (-1ex,0.8ex){\footnotesize s}; \end{scope} 
   & 2 & 1 & 0 & $-1$ & $-2$ & $-3$ & $-4$ & $-5$ & $-6$ & $-7$ & $-8$ & $-9$ & $-10$ & $-11$ & $-12$ & $-13$ & $-14$\\ 
 2 & 5 & 0 & 0 &  0 &  x &  0 &  0 &  0 &  0 &  0 &  0 & 0 & 0 & 0 & 0 & 0 & 0 \\ 
 1 & 0 & 0 & 1 &  0 &  x &  0 &  1 &  0 &  1 &  0 &  1 & 0 & 0 & 0 & 1 & 0 & 0 \\ 
 0 & 0 & 0 & 1 &  0 &  0 &  0 &  1 &  0 &  1 &  0 &  1 & 0 & x & 0 & 1 & 0 & x \\
-1 & 0 & 0 & 0 &  0 &  0 &  0 &  0 &  0 &  0 &  0 &  0 & 0 & x & 0 & 0 & 0 & x \\
};
\draw (mymatrix-1-1.south west) ++ (-0.2cm,0) -- (mymatrix-1-18.south east);
\draw (mymatrix-1-2.north west) -- (mymatrix-5-2.south west);
\draw (mymatrix-1-1.north west) -- (mymatrix-1-1.south east);%n k diagonal line
\end{tikzpicture}
\caption{$\Gamma= D_5$.  $x$ is 1 if $\text{char} \f{K} = 2$, $0$ otherwise.} 
\end{table}

\begin{table}[htb!]
\centering
\begin{tikzpicture}
\matrix (mymatrix) [matrix of nodes, nodes in empty cells, text height=1.5ex, text width=2.9ex, align=right]
{
\begin{scope} \tikz\node[overlay] at (-2.2ex,-0.6ex){\footnotesize r+s};\tikz\node[overlay] at (-1ex,0.8ex){\footnotesize s}; \end{scope} 
   & 2 & 1 & 0 & $-1$ & $-2$ & $-3$ & $-4$ & $-5$ & $-6$ & $-7$ & $-8$ & $-9$  & $-10$ & $-11$ & $-12$   & $-13$ & $-14$ & $-15$ & $-16$ & $-17$ & $-18$\\ 
 2 & 6 & 0 & 0 &  0 &  x &  0 &  0 &  0 &  x &  0 &  0  & 0 & 0 & 0 & 0 & 0 &0 &0 &0 &0 &0\\ 
 1 & 0 & 0 & 1 &  0 &  x &  0 &  1 &  0 &  x &  0 &  2  & 0 & 0 & 0 & 1 & 0 &0 &0 &1&0 &0 \\
 0 & 0 & 0 & 1 &  0 &  0 &  0 &  1 &  0 &  0 &  0 &  2  & 0 & x & 0 & 1 &0 &x &0 &1 &0 &2x\\
-1 & 0 & 0 & 0 &  0 &  0 &  0 &  0 &  0 &  0 &  0 &  0  & 0 & x & 0 & 0 &0 &x &0 &0 &0 &2x\\
};
\draw (mymatrix-1-1.south west) ++ (-0.2cm,0) -- (mymatrix-1-22.south east);
\draw (mymatrix-1-2.north west) -- (mymatrix-5-2.south west);
\draw (mymatrix-1-1.north west) -- (mymatrix-1-1.south east);%n k diagonal line
\end{tikzpicture}
\caption{$\Gamma= D_6$.  $x$ is 1 if $\text{char} \f{K} = 2$, $0$ otherwise.} 
\end{table}

\ 

\begin{remark} As a result of the computation for $\Gamma= D_n$, we have $HH^{2}(A_\Gamma,A_\Gamma [s])=0$ for all $s$ over any field $\f{K}$. This rigidity has a
	useful implication in Floer theory: namely, if one has a
	$D_n$-configuration of Lagrangian spheres $S_v$ in a symplectic
	4-manifold $M$, then the Floer cohomology algebra $\bigoplus_{v,w} HF_M
	^*(S_v,S_w)$ is isomorphic to $A_\Gamma$, i.e. it is independent of the
	symplectic manifold $M$. Furthermore, if $\mathrm{char}\f{K} \neq 2$,
	intrinsic formality implies that in fact the $A_\infty$-algebra
	$\bigoplus_{v,w} CF_M^*(S_v,S_w)$ is quasi-isomorphic to $A_\Gamma$.
\end{remark}

\section{Conclusion}

\subsection{Comparison with geometric viewpoint}

We would like to discuss the algebraic computations given in Sec.
(\ref{typeA}) in terms of the symplectic geometry of the Milnor fibre
$X_\Gamma$. We shall omit some of the details but the geometric set-up that we
are about to lay out is taken from \cite{seidelgraded}. Consider $\f{C}^3$ with
its standard symplectic form $d\alpha$ where 
\[ \alpha = -\frac{1}{4} d^c (|z_1|^2 + |z_2|^2 + |z_3|^2) \]
Let $p : \f{C}^3 \to \f{C}$ be the polynomial: \[ p(z_1,z_2,z_3) =
z_1^{n+1} + z_2^2+ z_3^2 \] which has an isolated singularity at the origin of
type $A_n$.  Consider also the Hamiltonian function $H: \f{C}^3 \to \f{R}$ given by \[
H(z_1,z_2,z_3) = 2|z_1|^2 + (n+1)|z_2|^2 + (n+1) |z_3|^2 \] Let $\psi$ be a
cut-off function such that $\psi(t^2) =1$ for $t \leq 1/3$ and $\psi(t^2) = 0 $
for $t \geq 2/3$. For $u\in \f{C} \backslash \{ 0\}$ with $0< |u| < \epsilon$ for
sufficiently small $\epsilon$, we consider the Milnor fibre: \[ \{ z \in
\f{C}^3 : p(z) = \psi(H(z))u \} \] For sufficiently small $\epsilon$, this is a
symplectic submanifold of $\f{C}^3$ and can be symplectically identified with
$X_\Gamma$. For $r \geq 2/3$, we let $L_r = F \cap \{ H=r \}$ be the link of
the singularity. In other words, for such $r$, we have \[ L_r = \{ z \in
\f{C}^3 : 2 |z_1|^2 + (n+1) |z_2|^2 + (n+1) |z_3|^2 = r, p(z) =0 \}.
\] For $r>0$, $L_r$ inherits a contact structure $\alpha|_{L_r}$ and outside of a compact set $X_\Gamma$ can be identified with the positive symplectization of $L_r$. 
The appealing feature of this set-up is that the Reeb vector field $R_r$ on $L_r$ has a periodic flow given by:
\[ t \cdot (z_1, z_2, z_3) = ( e^{\frac{4it}{r}} z_1,  e^{\frac{2 (n+1) it}{r}} z_2 , e^{\frac{2 (n+1) it}{r}} z_3 )   \]
Thus, all the Reeb orbits are along the circle direction of a Seifert fibred structure on the Lens space $L_r \cong L(n+1,n)$. Furthermore, since the Reeb flow is explicit, we can actually write down all the orbits. Let us take $Y_\Gamma = L_1$ as our contact boundary. There are two types of simple orbits:
\begin{itemize} 	
	\item Generic simple orbits of period $\frac{\pi}{2n+2} \mathrm{lcm}(2,n+1)$. These are orbits through points $(z_1,z_2,z_3) \in Y_\Gamma$ such that $z_1 \neq 0$. The $N^{th}$ multiple cover of these orbits have Conley-Zehnder index $2N$ if $n$ is odd, $4N$ if $n$ is even. 
	\item Exceptional simple orbits of period $\frac{\pi}{n+1}$. These are orbits through points $(0,z_2, z_3) \in Y_\Gamma$. The $N^{th}$ multiple cover of this orbit has Conley-Zehnder index $2 \lfloor \frac{2N}{n+1}\rfloor +1 $ except when $2N= M(n+1)$ for some $M \in \f{Z}$, in which case the index is $2M$.
\end{itemize}

For each $N \in \f{Z}_+$, we can consider $N$-fold multiple covers of generic simple orbits together with $(n+1)N$-fold (resp. $(n+1)N/2)$-fold) for $n$ even (resp. $n$ odd) multiple covers of exceptional orbits as parametrized by the manifold $L(n+1,n)$ and the $N$-fold cover of exceptional orbits for each $N \in \f{Z}_+$ not divisible by $n+1$ (resp. $(n+1)/2$) for $n$ even (resp. $n$ odd) as parametrized by $S^1 \sqcup S^1$. This leads to a standard Morse-Bott type spectral sequence converging to $SH^*(X_\Gamma)$ (see
\cite{biased} and/or \cite{otto} for a more recent exposition). For example, for $n=2$, the  $E_1$ page is given by
\begin{equation} 
	E_1^{pq} = \begin{cases}
		H^{q}(X_\Gamma;\f{K}) & \text{if\ } p = 0, \\
		H^{q-p-2}((S^1 \sqcup S^1); \f{K}) & \text{if\ } p =2l+1 < 0, \\
		H^{q-p}(L(3,2);\f{K}) \oplus H^{q-p-2}((S^1 \sqcup S^1); \f{K}) & \text{if\ } p= 2l < 0, \\
		0 & \text{if\ } p>0.
	\end{cases}
\end{equation}
The higher differentials come from contributions of holomorphic cylinders counted in the differential of symplectic cohomology. A finite truncation of the $E_1$ page of this spectral sequence is showin in Table \ref{e1page}. 
\begin{table}[htb!]
\centering
\begin{tikzpicture}
\matrix (mymatrix) [matrix of nodes, nodes in empty cells, text height=1.5ex, text width=2.5ex, align=right]
{
\begin{scope} \tikz\node[overlay] at (-2.6ex,-0.2ex){\footnotesize p+q};\tikz\node[overlay] at (-1ex,0.8ex){\footnotesize q}; \end{scope} 
   & 2 & 1 & 0 & $-1$ & $-2$ & $-3$ & $-4$ &  \\ 
 2 & 2 & 0 & 0 & 0 &  0 &  0 &  0 &     \\ 
 1 & 2 & 0 & 0 & 0 &  0 &  0 &  0 &     \\
 0 & 0 & 2 & 1 & 0 &  0 &  0 &  0 &     \\
-1 & 0 & 3 & 0 & 0 &  0 &  0 &  0 &     \\
-2 & 0 & 0 &x+2& 0 &  0 &  0 &  0 &     \\
-3 & 0 & 0 & 2 & x &  0 &  0 &  0 &     \\
-4 & 0 & 0 & 0 & 2 &  1 &  0 &  0 &     \\
-5 & 0 & 0 & 0 & 3 &  0 &  0 &  0 &     \\
-6 & 0 & 0 & 0 & 0 & x+2 &  0 &  0 &     \\
-7 & 0 & 0 & 0 & 0 &  2 &  x &  0 &     \\
-8 & 0 & 0 & 0 & 0 &  0 &  2 &  1 &     \\
};
\draw (mymatrix-1-1.south west) ++ (-0.2cm,0) -- (mymatrix-1-8.south east);
\draw (mymatrix-1-2.north west) -- (mymatrix-12-2.south west);
\draw (mymatrix-1-1.north west) -- (mymatrix-1-1.south east);%n k diagonal line
\end{tikzpicture}
\caption{$E_1$ page of the Morse-Bott spectral sequence for $\Gamma = A_2$.  $x$ is 1 if $\text{char} \f{K} = 3$, $0$ otherwise.} 
\label{e1page}
\end{table}

\vspace{-2mm}

Comparing this with our results from Sec.~(\ref{typeA}), which correspond to a calculation of the total complex at the $E_\infty$ page of the spectral sequence, gives us information about the holomorphic cylinders contributing to the differential of symplectic cohomology. 
For example, if $\mathrm{char}\f{K}=3$, the spectral sequence has to be degenerate but otherwise there has to be a non-trivial differential. 
See also the appendix of \cite{otto} for a similar spectral sequence obtained via another natural choice of a contact form on the lens space $L(n+1,n)$. 

In conclusion, even though this geometric point of view leads to an appealing
description of the generators of the chain complex, it seems harder to
determine the cohomology this way, let alone its multiplicative structure.
However, it is reassuring that the algebraic approach taken in this paper and
the geometric picture just outlined are compatible. 

\vspace{-3mm}

\subsection{Generalizations} 

In this paper, we have studied Legendrian links $\Lambda \subset (S^3,
\xi_{std})$ which are obtained by plumbing Legendrian unknots according to a
plumbing tree $\Gamma$.  One might wonder what Koszul duality has to say when
$\Lambda$ is a more general Legendrian submanifold. Of course, one can study
this plumbing construction in higher dimensions. Both Ginzburg DG-algebra and
the zigzag algebra have analogues corresponding to higher-dimensional
plumbings, and we expect that our calculations can be extended in a
straightforward way. 

Perhaps, a more interesting direction to pursue is the following. One of our
main observations was that the Legendrian cohomology DG-algebra of $\Lambda$
admits a certain natural augmentation $\epsilon : LCA^*(\Lambda) \to
\mathrm{k}$ such that  \begin{equation} \label{generalize}
	\text{RHom}_{LCA^*(\Lambda)^{op}}(\mathrm{k},\mathrm{k})
\end{equation} is quasi-isomorphic to a finite-dimensional associative algebra
$A$, whose Hochschild complex is isomorphic to that of $LCA^*(\Lambda)$ by an
$A_\infty$-version of Koszul duality.

One could contemplate generalizing this construction to an arbitrary Legendrian link $\Lambda$
whose $LCA^*(\Lambda)$ admits an augmentation $\epsilon$. In general, one cannot expect to have the
connectedness and the finiteness conditions required in Thm.~(\ref{dgkoszul}). Furthermore, in
general, $LCA^*(\Lambda)$ is not graded over $\f{Z}$ but over $\f{Z}/N$ for some $N > 0$. These pose
important restrictions, analogous to the assumption of simply connectedness that appears in the classical story
discussed in the introduction. One could partially extend Koszul duality to these more general
situations if one takes completions with respect to the augmentation ideal.

\end{document}